\documentclass[reqno, 12pt]{amsart}
\usepackage{amssymb}
\usepackage{graphicx}
\usepackage{xcolor} 
\usepackage{tensor}
\usepackage{enumitem}

\usepackage{fullpage} 

\definecolor{green}{rgb}{0,0.8,0} 

\newtheorem{theorem}{Theorem}[section]
\newtheorem{corollary}[theorem]{Corollary}
\newtheorem{lemma}[theorem]{Lemma}
\newtheorem{proposition}[theorem]{Proposition}
\theoremstyle{definition}
\newtheorem{definition}[theorem]{Definition}

\theoremstyle{remark}
\newtheorem{remark}[theorem]{Remark}
\numberwithin{equation}{section}
\newcommand{\nrm}[1]{\Vert#1\Vert}
\newcommand{\abs}[1]{\vert#1\vert}
\newcommand{\brk}[1]{\langle#1\rangle}
\newcommand{\set}[1]{\{#1\}}

\newcommand{\supp}{{\mathrm{supp}}}

\renewcommand{\Re}{\mathrm{Re}}
\renewcommand{\Im}{\mathrm{Im}}

\newcommand{\aeq}{\sim}
\newcommand{\aleq}{\lesssim}

\newcommand{\lap}{\triangle}

\newcommand{\ud}{\mathrm{d}}
\newcommand{\rd}{\partial}
\newcommand{\nb}{\nabla}

\newcommand{\imp}{\Rightarrow}

\newcommand{\bb}{\Big}

\newcommand{\0}{\emptyset}

\newcommand{\alp}{\alpha}
\newcommand{\bt}{\beta}
\newcommand{\gmm}{\gamma}

\newcommand{\dlt}{\delta}

\newcommand{\eps}{\epsilon}
\newcommand{\veps}{\varepsilon}

\newcommand{\lmb}{\lambda}

\newcommand{\sgm}{\sigma}

\newcommand{\tht}{\theta}

\newcommand{\Tht}{\Theta}

\newcommand{\omg}{\omega}
\newcommand{\Omg}{\Omega}

\newcommand{\zt}{\zeta}

\newcommand{\bfa}{{\bf a}}
\newcommand{\bfb}{{\bf b}}
\newcommand{\bfc}{{\bf c}}
\newcommand{\bfd}{{\bf d}}
\newcommand{\bfe}{{\bf e}}

\newcommand{\bfm}{{\bf m}}
\newcommand{\bfn}{{\bf n}}

\newcommand{\bfD}{{\bf D}}
\newcommand{\bfE}{{\bf E}}

\newcommand{\mfa}{\mathfrak{a}}
\newcommand{\mfb}{\mathfrak{b}}

\newcommand{\bbC}{\mathbb C}
\newcommand{\bbD}{\mathbb D}

\newcommand{\bbH}{\mathbb H}

\newcommand{\bbR}{\mathbb R}
\newcommand{\bbS}{\mathbb S}

\newcommand{\bbZ}{\mathbb Z}

\newcommand{\calD}{\mathcal D}
\newcommand{\calE}{\mathcal E}
\newcommand{\calF}{\mathcal F}
\newcommand{\calG}{\mathcal G}
\newcommand{\calH}{\mathcal H}

\newcommand{\calL}{\mathcal L}

\newcommand{\calO}{\mathcal O}
\newcommand{\calP}{\mathcal P}
\newcommand{\calQ}{\mathcal Q}

\newcommand{\calS}{\mathcal S}

\newcommand{\calX}{\mathcal X}
\newcommand{\calY}{\mathcal Y}

\newcommand{\covD}{\bfD}

\newcommand{\met}{\bfm}
\newcommand{\uL}{\underline{L}}

\newcommand{\near}{\mathrm{near}}
\newcommand{\far}{\mathrm{far}}

\newcommand{\ualp}{\underline{\alp}}

\newcommand{\M}{\mathrm{M}}
\newcommand{\KG}{\mathrm{KG}}
\newcommand{\defT}[1]{{}^{(#1)} \pi}
\newcommand{\vC}[1]{{}^{(#1)} J}
\newcommand{\sC}[1]{{}^{(#1)} K}

\newcommand{\mvC}[1]{{}^{(#1)} P}			
\newcommand{\msC}[1]{{}^{(#1)} Q}			

\newcommand{\scovD}{\! \not \!\! \covD}
\newcommand{\snb}{\hskip-.25em \not \hskip-.25em \nb}

\newcommand{\pfstep}[1]{\vspace{.5em} \noindent {\bf #1.} }
\DeclareMathOperator*{\esssup}{ess\,sup\,}

\newcommand{\extr}{\mathrm{ext}}

\newcommand{\rst}{\!\upharpoonright}		

\newcommand{\covED}{\bfE \bfD}			
\newcommand{\ED}{ED}					

\newcommand{\EM}{\calQ}				

\newcommand{\thED}{\eps}		
\newcommand{\thcovED}{\bfe}		
\newcommand{\thE}{\eps_{\ast}}		

\newcommand{\Xw}{\calX^{w}}						
\newcommand{\Yw}{\calY^{w}}						

\newcommand{\Hardy}{\mathcal{{H}}}
\newcommand{\Null}{\mathcal{{N}}}

\newcommand{\EFlux}{\calF}
\newcommand{\G}{\calG}

\newcommand{\weakto}{\rightharpoonup}

\begin{document}

\title[GWP and scattering of (4+1)-d MKG]{Global well-posedness and scattering of the \\ (4+1)-dimensional Maxwell-Klein-Gordon equation}
\author{Sung-Jin Oh}%
\address{Department of Mathematics, UC Berkeley, Berkeley, CA, 94720}%
\email{sjoh@math.berkeley.edu}%

\author{Daniel Tataru}%
\address{Department of Mathematics, UC Berkeley, Berkeley, CA, 94720}%
\email{tataru@math.berkeley.edu}%


\begin{abstract}
  This article constitutes the final and main part of a three-paper
  sequence \cite{OT1, OT2}, whose goal is to prove global
  well-posedness and scattering of the energy critical
  Maxwell-Klein-Gordon equation (MKG) on $\mathbb{R}^{1+4}$ for
  arbitrary finite energy initial data. Using the successively
  stronger continuation/scattering criteria established in the
  previous two papers \cite{OT1, OT2}, we carry out a blow-up analysis
  and deduce that the failure of global well-posedness and scattering
  implies the existence of a nontrivial stationary or self-similar
  solution to MKG. Then, by establishing that such solutions do not
  exist, we complete the proof.
\end{abstract}
\maketitle

\setcounter{tocdepth}{1}
\tableofcontents

\section{Introduction}
In this article we prove global well-posedness and scattering of the
energy critical Maxwell-Klein-Gordon equation on $\bbR^{1+4}$
for any finite energy initial data data. In Section~\ref{subsec:bg},
we present some background material concerning the
Maxwell-Klein-Gordon equation on $\bbR^{1+4}$. Readers already
familiar with this equation may skip to Section~\ref{subsec:result},
where we give a precise statement of the main theorem
(Theorem~\ref{thm:main}). This paper is the main and logically the
final part of the three-paper sequence \cite{OT1, OT2}.  In
Sections~\ref{sec:overview} and \ref{sec:OT3} below, we provide an
overview of the entire proof of Theorem~\ref{thm:main} spanning the
whole sequence.

\subsection{$(4+1)$-dimensional Maxwell-Klein-Gordon
  system} \label{subsec:bg} Let $\bbR^{1+4}$ be the
$(4+1)$-dimensional Minkowski space with the metric
\begin{equation*}
  \met_{\mu \nu} := \mathrm{diag}\,(-1,+1,+1,+1,+1)
\end{equation*}
in the standard rectilinear coordinates $(t=x^{0}, x^{1}, \cdots,
x^{4})$. Consider the trivial complex line bundle $L = \bbR^{1+4}
\times \bbC$ over $\bbR^{1+4}$ with structure group $\mathrm{U}(1) =
\set{e^{i \chi} \in \bbC}$. Global sections of $L$ may be identified
with $\bbC$-valued functions on $\bbR^{1+4}$. Using the identification
$u(1) \equiv i \bbR$ and taking the trivial connection $\ud$ as a
reference, any connection $\covD$ on $L$ takes the form
\begin{equation*}
  \covD = \ud + i A
\end{equation*}
for some real-valued 1-form $A$ on $\bbR^{1+4}$. The
\emph{Maxwell-Klein-Gordon system} is a Lagrangian field theory for a
pair $(A, \phi)$ of a connection on $L$ and a section of $L$ with the
action functional
\begin{equation*}
  \calS[A, \phi] = \int_{\bbR^{1+4}} \frac{1}{4} F_{\mu \nu} F^{\mu \nu} + \frac{1}{2} \covD_{\mu} \phi \overline{\covD^{\mu} \phi} \, \ud t \ud x,
\end{equation*}
where $F_{\mu \nu} = (\ud A)_{\mu \nu} = \rd_{\mu} A_{\nu} - \rd_{\nu}
A_{\mu}$ is the \emph{curvature 2-form} associated to $\covD$. We
follow the usual convention of raising/lowering indices by the
Minkowski metric $\bfm$, and also of summing over repeated upper and
lower indices.  Computing the Euler-Lagrange equations, we arrive at
the \emph{Maxwell-Klein-Gordon equations} (MKG)
\begin{equation} \label{eq:MKG} \tag{MKG} \left\{
    \begin{aligned}
      \rd^{\mu} F_{\nu \mu} =& \Im(\phi \overline{\covD_{\nu} \phi})  \\
      \Box_{A} \phi =& 0,
    \end{aligned}
  \right.
\end{equation}
where $\Box_{A} := \covD^{\mu} \covD_{\mu}$ is the (gauge) covariant
d'Alembertian.

A basic feature of \eqref{eq:MKG} is \emph{gauge
  invariance}. Geometrically, a gauge transform is a change of basis
in the fiber $\bbC$ over each point on $\bbR^{1+4}$ by an element of
the gauge group $\mathrm{U}(1)$. Accordingly, we refer to a
real-valued function $\chi : \bbR^{1+4} \to \bbR$ (hence $e^{i \chi}
\in \mathrm{U}(1)$) as a \emph{gauge transformation} and define the
corresponding gauge transform of a pair $(A, \phi)$ as
\begin{equation} \label{eq:MKG-gt}
  (A, \phi) \mapsto (\widetilde{A}, \widetilde{\phi}) := (A - \ud \chi, e^{i \chi} \phi).
\end{equation}
Observe that $\covD$ and $\Box_{A}$ are covariant under gauge
transforms (i.e., $e^{i \chi} \covD \phi = \widetilde{\covD}
\widetilde{\phi}$ etc), whereas $F$ and $\Im(\phi
\overline{\covD_{\mu} \phi})$ are invariant. Hence \eqref{eq:MKG} is
invariant under gauge transforms. Since $\mathrm{U}(1)$ is an abelian group,
\eqref{eq:MKG} is said to be an \emph{abelian gauge theory}.

We now formulate the \emph{initial value problem} for \eqref{eq:MKG},
in a way that is consistent with the gauge invariance of the system. An
\emph{initial data set} for \eqref{eq:MKG} consists of a pair of
1-forms $(a_{j}, e_{j})$ and a pair of $\bbC$-valued functions $(f,
g)$ on $\bbR^{4}$. We say that $(a, e, f, g)$ is the initial data for
a solution $(A, \phi)$ at time $t_{0}$ if
\begin{equation*}
  (A_{j}, F_{0j}, \phi, \covD_{t} \phi) \rst_{\set{t = t_{0}}} = (a_{j}, e_{j}, f, g).
\end{equation*}
We usually take the initial time $t_{0}$ to be zero. Observe that the
$\nu = 0$ component of \eqref{eq:MKG} imposes a constraint on any
initial data for \eqref{eq:MKG}, namely
\begin{equation} \label{eq:gaussEq} \rd^{\ell} e_{\ell} = \Im(f
  \overline{g})
\end{equation}
This equation is called the \emph{Gauss} (or \emph{constraint})
\emph{equation}.

There is a \emph{conserved energy} for \eqref{eq:MKG}, which is one of
the basic ingredients of the non-perturbative analysis performed in
this paper. We define the conserved energy of a solution $(A, \phi)$
at time $t$ to be
\begin{equation}
  \calE_{\set{t} \times \bbR^{4}}[A, \phi] := \frac{1}{2} \int_{\set{t} \times \bbR^{4}} \sum_{0 \leq \mu < \nu \leq 4} \abs{F_{\mu \nu}}^{2} + \sum_{0 \leq \mu \leq 4} \abs{\covD_{\mu} \phi}^{2} \, \ud x.
\end{equation}
For a suitably regular solution to \eqref{eq:MKG} defined on a
connected interval $I$, this quantity is constant. This conservation
law is in fact a consequence of N\"other's principle (i.e., continuous
symmetry of the field theory corresponds to a conserved quantity)
applied to the time translation symmetry of \eqref{eq:MKG}; we refer
to Section~\ref{sec:energy} for further discussion and a proof.

Observe that the conserved energy is invariant under the scaling
\begin{equation*}
  (A, \phi)(t,x) \mapsto (\lmb^{-1} A, \lmb^{-1} \phi)(\lmb^{-1} t, \lmb^{-1} x) \quad \hbox{ for any } \lmb > 0,
\end{equation*}
which also preserves the system \eqref{eq:MKG}. Hence \eqref{eq:MKG}
on $\bbR^{1+4}$ is \emph{energy critical}.

\subsection{Statement of the main theorem}\label{subsec:result}
Our goal now is to give a precise statement of the global
well-posedness/scattering theorem proved in this paper. For this
purpose, we first borrow some definitions from \cite{Krieger:2012vj,
  OT1}.

We say that a \eqref{eq:MKG} initial data set $(a, e, f, g)$ (i.e., a
solution to the Gauss equation) is \emph{classical} and write $(a, e,
f, g) \in \calH^{\infty}$ if each of $a, e, f, g$ belongs to
$H^{\infty}_{x} := \cap_{n=0}^{\infty} H^{n}_{x}$. Correspondingly, we
say that a smooth solution $(A, \phi)$ to \eqref{eq:MKG} on $I \times
\bbR^{4}$ (where $I \subseteq \bbR$ is an interval) is a
\emph{classical solution} if $A_{\mu}, \phi \in \cap_{n, m=0}^{\infty}
C_{t}^{m} (I; H^{n}_{x})$. 

Define the space $\calH^{1} = \calH^{1}(\bbR^{4})$ of \emph{finite
  energy initial data sets} to be the space of \eqref{eq:MKG} initial
data sets for which the following norm is finite:
\begin{equation}
  \nrm{(a, e, f, g)}_{\calH^{1}} := \sup_{j=1, \ldots, 4} \nrm{(a_{j}, e_{j})}_{\dot{H}^{1}_{x} \times L^{2}_{x}(\bbR^{4})} + \nrm{(f, g)}_{\dot{H}^{1}_{x} \times L^{2}_{x}(\bbR^{4})}.
\end{equation}
Given a pair $(A, \phi)$ on $I \times \bbR^{4}$, we define its $C_{t}
\calH^{1}(I \times \bbR^{4})$ norm as
\begin{equation*}
  \nrm{(A, \phi)}_{C_{t} \calH^{1}(I \times \bbR^{4})} 
  := \esssup_{t \in I} \bb( \nrm{A[t]}_{\dot{H}^{1}_{x} \times L^{2}_{x}} + \nrm{\phi[t]}_{\dot{H}^{1}_{x} \times L^{2}_{x}} \bb),
\end{equation*}
where $A[t]$ and $\phi[t]$ are shorthands for $(A, \rd_{t} A)(t)$ and
$(\phi, \rd_{t} \phi)(t)$, respectively. We then define the notion of
an \emph{admissible $C_{t} \calH^{1}$ solution} to \eqref{eq:MKG} via
approximation by classical solutions as follows.

\begin{definition}[Admissible $C_{t} \calH^{1}$ solutions to
  \eqref{eq:MKG}] \label{def:admSol} Let $I \subseteq \bbR$ be an
  interval.  We say that a pair $(A, \phi) \in C_{t} \calH^{1}(I
  \times \bbR^{4})$ is an \emph{admissible $C_{t} \calH^{1}(I \times
    \bbR^{4})$ solution} to \eqref{eq:MKG} if there exists a sequence
  $(A^{(n)}, \phi^{(n)})$ of classical solutions to \eqref{eq:MKG} on
  $I \times \bbR^{4}$ such that
  \begin{equation*}
    \nrm{(A, \phi) - (A^{(n)}, \phi^{(n)})}_{C_{t} \calH^{1}(J \times \bbR^{4})} \to 0  \quad \hbox{ as } n \to \infty,
  \end{equation*}
  for every compact subinterval $J \subseteq I$.
\end{definition}
The necessity of restricting the class of energy solutions under
consideration to the admissible ones as defined above is a relatively
standard matter in the realm of low regularity solutions for
nonlinear dispersive equations. Often uniqueness statements require
additional regularity properties for solutions, which are then proved
to hold for the solutions which are limits of smooth solutions, but
might not be true or straightforward in general. In our case the
difficulties are compounded by the need to have a good notion of
finite energy solution which is gauge invariant.

\begin{remark}
  The above definitions can be localized to an open subset $O
  \subseteq \bbR^{4}$ or $\calO \subseteq \bbR^{1+4}$ in an obvious
  manner; see \cite[Sections 3 and 5]{OT1}.
\end{remark}

Next, we recall the \emph{global Coulomb gauge condition}
\begin{equation} \label{eq:g-Coulomb} \rd^{\ell} A_{\ell} = \sum_{\ell
    = 1, \ldots, 4} \rd_{\ell} A_{\ell} = 0.
\end{equation}
The role of this condition is to fix the ambiguity arising from the
gauge invariance of \eqref{eq:MKG}, which is an immediate formal
obstruction for well-posedness.

Finally, given an interval $I \subseteq \bbR$, we borrow the
space-time norms $Y^{1}(I \times \bbR^{4})$ and $S^{1}(I \times
\bbR^{4})$ from \cite{Krieger:2012vj, OT1, OT2}. We define the
$S^{1}$ norm of a solution $(A, \phi)$ on $I \times \bbR^{4}$ to be
\begin{equation*}
  \nrm{(A, \phi)}_{S[I]} := \nrm{A_{0}}_{Y^{1}(I \times \bbR^{4})} + \nrm{A_{x}}_{S^{1}(I \times \bbR^{4})} + \nrm{\phi}_{S^{1}(I \times \bbR^{4})}.
\end{equation*}
In particular, the $S^{1}$ norm captures the dispersive properties of
$A_{x}$ and $\phi$. The precise definition of the $S^{1}$ norm is
rather intricate; instead of the full definition, in this paper we
only rely on a few basic properties of the spaces $Y^{1}$ and $S^{1}$,
such as those below (see also Remark~\ref{rem:SY-norms}).
\begin{equation*}
  \nrm{(\varphi, \rd_{t} \varphi)}_{C_{t}(I; \dot{H}^{1}_{x} \times L^{2}_{x})} \aleq \nrm{\varphi}_{S^{1}(I \times \bbR^{4}) }, \quad
  \nrm{(\varphi, \rd_{t} \varphi)}_{C_{t}(I; \dot{H}^{1}_{x} \times L^{2}_{x})} \aleq \nrm{\varphi}_{Y^{1}(I \times \bbR^{4}) }.
\end{equation*}

We are now ready to state our main theorem.
\begin{theorem}[Main Theorem] \label{thm:main} Let $(a, e, f, g) \in
  \calH^{1}$ be a finite energy initial data set for \eqref{eq:MKG}
  obeying the global Coulomb gauge condition $\rd^{\ell} a_{\ell} =
  0$.  Then there exists a unique admissible $C_{t} \calH^{1}$
  solution $(A, \phi)$ to the initial value problem defined on the
  whole $\bbR^{1+4}$ which satisfies the global Coulomb gauge
  condition $\rd^{\ell} A_{\ell} = 0$. Moreover, the $S^{1}$ norm of $(A,
  \phi)$ is finite, i.e.,
  \begin{equation} \label{eq:main-apriori}
    \nrm{A_{0}}_{Y^{1}(\bbR^{1+4})} + \nrm{A_{x}}_{S^{1}(\bbR^{1+4})}
    + \nrm{\phi}_{S^{1}(\bbR^{1+4})} < \infty.
  \end{equation}
\end{theorem}
\begin{remark} \label{rem:scat} The a-priori bound above implies scattering
  towards both $t \to \pm \infty$; see Theorem~\ref{thm:finite-S}. It also implies
continuity of the data to solution map on compact time intervals, though not 
on the full real line.  
\end{remark}

\begin{remark}
  We do not lose any generality by restricting to initial data sets in
  the global Coulomb gauge, since any finite energy initial data set
  can be gauge transformed to obey the condition $\rd^{\ell} a_{\ell}
  = 0$. See \cite[Section 3]{OT1}.
\end{remark}

\begin{remark} 
  We note that an independent proof of global well-posedness and
  scattering of MKG-CG has been recently announced by
  Krieger-L\"uhrman, following a version of the Bahouri-G\'erard
  nonlinear profile decomposition \cite{MR1705001} and Kenig-Merle
  concentration compactness/rigidity scheme \cite{MR2257393,
    MR2461508} developed by Krieger-Schlag \cite{KrSch} for
  the energy critical wave maps.
\end{remark}

\subsection{A brief history and broader context} \label{subsec:literature} 
A natural point of view is to place the present papers and results within the 
larger context of nonlinear wave equations, of which the starting point is
 the semilinear wave equation $\Box u = \pm |u|^{p} u$. More accurately,
the (MKG) equation belongs to the class of geometric wave 
equations, which includes wave maps (WM), Yang-Mills (YM), Einstein equations,
as well as many other coupled models. Two common features of all these 
problems are that they admit a Lagrangian formulation, and have some natural 
gauge invariance properties. Following are some of the key developments
that led to the present work.

\medskip

{\em 1. The null condition.} A crucial early observation in the study
of both long range and low regularity solutions to geometric wave
equations was that the nonlinearities appearing in the equations have
a favorable algebraic structure, which was called {\em null
  condition}, and which can be roughly described as a cancellation
condition in the interaction of parallel waves. In the low regularity
setting, this was first explored in work of Klainerman and
Machedon~\cite{MR1231427}, and by many others later on.
 
\medskip

{\em 2. The $X^{s,b}$ spaces.}  A second advance was the introduction
of the $X^{s,b}$ spaces\footnote{The concept, and also the notation,
  is due to Bourgain, in the context of KdV and NLS type problems.},
also first used by Klainerman and Machedon~\cite{MR1381973} in the context of
the wave equation.  Their role was to provide enough structure in
order to be able to take advantage of the null condition in bilinear
and multilinear estimates. Earlier methods, based on energy bounds,
followed by the more robust Strichartz estimates, had proved inadequate 
to the task.

\medskip

{\em 3. The null frame spaces.}  To study nonlinear problems at
critical regularity one needs to work in a scale invariant
setting. However, it was soon realized that the homogeneous $X^{s,b}$
spaces are not even well defined, not to mention suitable for
this. The remedy, first introduced in work of the second
author~\cite{MR1827277} in the context of wave maps, was to produce a better
description of the fine structure of waves, combining frequency and
modulation localizations with adapted frames in the physical space.
This led to the {\em null frame spaces}, which played a key role
in subsequent developments for wave maps. We remark that another
scale invariant  alternative to $X^{s,b}$ spaces are the $U^p$ and $V^p$ 
spaces, also originally developed by the second author; while these 
played a role in the study of  other  nonlinear dispersive problems at
critical regularity, they play no role in the present story.

\medskip

{\em 4. Renormalization.}  A remarkable feature of all semilinear
geometric wave equations is that while at high regularity
(and locally in time) the nonlinearity is perturbative, this is no longer 
the case at critical regularity. Precisely, isolating the 
non-perturbative component of the nonlinearity, one can see that this 
is of  paradifferential type; in other words, the high frequency waves 
evolve on a variable low frequency background.  To address this 
difficulty, the idea of Tao~\cite{MR1869874}, also in the wave map context, was to 
{\em renormalize} the paradifferential problem, i.e., to find a suitable 
approximate conjugation to the corresponding constant coefficient problem.

\medskip

{\em 5. Induction of energy.} The ideas discussed so far seem to suffice
for small data critical problems. Attacking the large data problem
generates yet another range of difficulties. One first step in this
direction is Bourgain's {\em induction of energy} idea \cite{Bo},
which is a convenient mechanism to transfer information to higher and
higher energies. We remark that an alternate venue here, which
sometimes yields more efficient proofs, is the Kenig-Merle idea
\cite{MR2461508} of constructing {\em minimal blow-up
  solutions}. However, the implementation of this method in problems
which require renormalization seems to cause considerable trouble.
For a further discussion on this issue, we refer to \cite{KrSch}, where this 
method was carried out in the case of energy critical wave maps into the hyperbolic plane.

\medskip

{\em 6. Energy dispersion.} One fundamental goal in the study of large
data problems is to establish a quantitative dichotomy between
dispersion and concentration.  The notion of {\em energy dispersion},
introduced in joint work~\cite{MR2657817, MR2657818} of the second author and Sterbenz in
the wave map context, provides a convenient measure for pointwise
concentration. Precisely, at each energy there is an energy dispersion
threshold below which dispersion wins.  We remark that, when it can be
applied, the Kenig-Merle method \cite{MR2461508} yields more accurate
information; for instance, see \cite{KrSch}. However, the energy dispersion idea,
which is what we follow in the present series of papers, is much easier to
implement in conjunction with renormalization.

\medskip

{\em 7. The frequency gap.}  One obstacle in the transition from small
to large data in renormalizable problems is that the low frequency
background may well correspond to a large solution. Is this fatal to
the renormalized solution? The answer to that, also originating in
\cite{MR2657817, MR2657818}, is that there may be a second hidden source of
smallness, namely a large {\em frequency gap} between the
high frequency wave and the low frequency background it evolves on.

\medskip

{\em 8. Morawetz estimates.} The outcome of the ideas above is a dichotomy 
between dispersion and scattering on one hand, and very specific concentration
patterns, e.g., solitons, self-similar solutions on the other hand. The Morawetz estimates,
first appearing in this role in the work of Grillakis~\cite{MR1078267},
are a convenient and relatively simple tool to eliminate such concentration scenarios.

\medskip
We now recall some earlier developments on geometric wave equations related to the present
paper. We start our discussion with the (MKG) problem above the scaling critical
regularity. In the two and three dimensional cases, which are energy subcritical, global regularity 
of sufficiently regular solutions was shown in the early works \cite{MR579231, MR649158, MR649159}. 
The latter two in fact handled the more general Yang-Mills-Higgs system.
In dimension $d=3$, this result was greatly improved by \cite{MR1271462}, 
which established global well-posedness for any finite energy data. 
In this work, the quadratic null structure of (MKG) in the Coulomb gauge was uncovered and used for the first time. 
Subsequent developments were made by \cite{Cu} and
more recently \cite{MaSte}, where an essentially optimal local
well-posedness result was established. An important observation in \cite{MaSte}
is that \eqref{eq:MKG} in Coulomb gauge exhibits a secondary multilinear cancellation feature. The related paper
\cite{MR2784611} is concerned with global well-posedness of the same problem at low regularity.
We also mention the work \cite{Selberg:2010ig}, in which finite energy global well-posedness was established in the Lorenz gauge.
In the higher dimensional case $d \geq 4$, an essentially optimal local
well-posedness result for a model problem closely related to (MKG) was 
obtained in \cite{KlainTat}. This was followed by further refinements in \cite{Se, Ste}.

The progress for the closely related Yang-Mills system (YM) in the subcritical regularity has largely paralleled 
that of (MKG), at least for small data. Indeed, (YM) exhibits a  
null structure in the Coulomb gauge which is very similar to (MKG). In particular, the aforementioned work \cite{KlainTat}
is also relevant for the small data problem for (YM) in the Coulomb gauge at an essentially optimal regularity.


However, a new difficulty arises in the large data\footnote{More precisely, a suitable scaling critical norm of the connection $A$ (e.g., $\nrm{A}_{L^{d}_{x}}$) or the curvature $F$ (e.g., $\nrm{F}_{L^{\frac{d}{2}}_{x}}$) is large.} problem for (YM): Namely, the gauge transformation law is \emph{nonlinear} due to the non-abelian gauge group. In particular, gauge transformations into the Coulomb gauge obey a nonlinear elliptic equation, 
for which no suitable large data regularity theory is available. Note, in comparison, that such gauge transformations obey a linear Poisson equation in the case of (MKG).
In \cite{Klainerman:1995hz}, where finite energy global well-posedness of the 3+1 dimensional (YM) problem was proved, 
this issue was handled by localizing in space-time via the finite speed of propagation to gain smallness, and then working in local Coulomb gauges.
An alternative, more robust approach without space-time localizations to the same problem has been put forth by the first author in \cite{Oh1, Oh2},
inspired by \cite{Tao-large}. The idea is to use an associated geometric flow, namely the \emph{Yang-Mills heat flow}, to select 
a global-in-space Coulomb-like gauge for data of any size. 

Before turning to the (MKG) and (YM) problems at critical regularity,
we briefly recall some recent developments on the wave map equation (WM), where many of the
methods we implement here have their roots. We confine our discussion
to the energy critical problem in $2+1$ dimensions, which is both the
most difficult and the most relevant to our present paper.  For the small data problem, 
global well-posedness was established in \cite{MR1869874}, \cite{Tao:2001gb},
\cite{MR2130618}.  More recently, the \emph{threshold theorem} for large data wave maps, 
which asserts that global well-posedness and scattering hold below the ground state energy, 
was proved in \cite{MR2657817, MR2657818} in general, and independently in
\cite{KrSch} and \cite{Tao-large} for specific targets (namely the hyperbolic
space). See also \cite{Lawrie:2015rr} for a sharp refinement in the case of a two-dimensional target, 
taking into account an additional topological invariant (namely, the degree of the wave map). 
Our present strategy was strongly influenced by \cite{MR2657817, MR2657818}, which can be seen as the first predecessor of
this work.

Despite the many similarities, there is a key structural difference between (WM) on the one hand 
and (MKG), (YM) on the other, whose understanding is crucial for making progress on the latter two problems.
Roughly speaking, all three equations can be written in a form where the main `dynamic variables', 
which we denote by $\phi$, obey a possibly nonlinear gauge covariant wave equation $\Box_{A} \phi = \cdots$, 
and the associated curvature $F[A]$ is determined by $\phi$.
In the case of (WM), this dependence is simply algebraic, whereas for 
(MKG) and (YM) the curvature $F[A]$ obeys a wave equation with a nonlinearity 
depending on $\phi$. This difference manifests in the renormalization procedure for each equation: 
For (WM) it suffices to use a physical space gauge transformation, whereas for (MKG) and (YM)
it is necessary to use a microlocal (more precisely, \emph{pseudo-differential}) gauge transformation that exploits the fact that $A$
solves a wave equation in a suitable gauge. 

The first (MKG) renormalization argument appeared in \cite{MR2100060}, 
in which global regularity of (MKG) for small critical Sobolev data was established in dimensions $d \geq 6$.
This work was followed by a similar high dimensional result for (YM) in \cite{KrSte}. 
Finally, the small data result in the energy critical dimension $4+1$ was obtained in \cite{Krieger:2012vj}, which may be
viewed as the second direct predecessor to the present work. In particular we borrow a good deal of notations, ideas and estimates
from \cite{Krieger:2012vj}.
  
We end our introduction with a few remarks on the energy critical (YM) problem in $4+1$ dimensions, 
which is a natural next step after the present work. The issue of non-abelian gauge group for the large data problem 
has already been discussed. Another important difference between (MKG) and (YM) 
in $4+1$ dimensions is that the latter problem admits \emph{instantons}, which are nontrivial static solutions with finite energy. 
Therefore, in analogy with (WM), it is reasonable to put forth the \emph{threshold conjecture} for the energy critical (YM) problem, namely that global well-posedness and scattering hold below the energy of the first instanton. Finally, (YM) is more `strongly coupled' as a system compared to (MKG), in the sense that the connection $A$ itself obeys a covariant wave equation. This feature seems to necessitate a more involved renormalization procedure compared to (MKG).

\subsection*{Acknowledgements}
Part of the work was carried out during the trimester program
`Harmonic Analysis and Partial Differential Equations' at the
Hausdorff Institute for Mathematics in Bonn; the authors thank the
institute for hospitality. S.-J. Oh is a Miller Research Fellow, and
thanks the Miller Institute for support. D. Tataru was partially
supported by the NSF grant DMS-1266182 as well as by the Simons
Investigator grant from the Simons Foundation.

\section{Overview of the proof I: Summary of the first two
  papers} \label{sec:overview} The basic strategy for proving
Theorem~\ref{thm:main} is by contradiction, following the scheme
successfully developed in \cite{MR2657817, MR2657818} in the setting
of energy critical wave maps. In the first two papers of the sequence
\cite{OT1, OT2} we establish successively stronger continuation and
scattering criteria, whose contrapositives provide precise information
about the nature of a finite time blow-up (i.e., failure of global
well-posedness) or non-scattering. In the present paper, we use this
information, as well as conservation laws and Morawetz-type monotonicity formulae
for \eqref{eq:MKG}, to perform a blow-up analysis and show that the
failure of Theorem~\ref{thm:main} implies the existence of a
nontrivial finite energy stationary or self-similar solution to
\eqref{eq:MKG}.  Since such a solution does not exist (see
Section~\ref{sec:stationary-self-sim} below), Theorem~\ref{thm:main}
must hold.

In this section we review the main results and ideas of the
earlier two papers in the sequence \cite{OT1, OT2}. In
Section~\ref{sec:OT3} we summarize the argument given in the present
paper. To steer away from unnecessary technical details we only
consider smooth data and solutions; however we remark that the results
also apply to merely finite energy data and admissible $C_{t}
\calH^{1}$ solutions. For the notation, we refer to
Section~\ref{sec:prelim}.

\subsection{Local well-posedness in the global Coulomb gauge and non-concentration of energy}
The main result of the first paper \cite{OT1} of the sequence is local
well-posedness of \eqref{eq:MKG} in the global Coulomb gauge with a
lower bound on the lifespan in terms of the \emph{energy concentration
  scale}
\begin{equation*}
  r_{c} = r_{c}(E)[a, e, f, g] := \sup \set{r > 0 : \forall x \in \bbR^{4}, \ \calE_{B_{r}(x)}[a, e, f, g] < \dlt_{0}(E, \thE^{2})},
\end{equation*}
where $B_{r}(x)$ denotes the open ball of radius $r$ with center $x$,
$\dlt_{0}(E, \thE^{2}) = c^2 \thE^{2} \min \set{1, \thE E^{-1}})$ and $c$ is an absolute constant
(see Theorem~\ref{thm:smallEnergy}). A simplified version is as follows:
\begin{theorem} \label{thm:OT1-simple} Given any $E > 0$ let 
  $\dlt_{0}(E,\thE^{2}) > 0$ be as above. Let $(a,
  e, f, g)$ be a smooth finite energy initial data for \eqref{eq:MKG}
  satisfying the global Coulomb gauge condition $\sum_{j} \rd_{j}
  a_{j} = 0$. Then there exists  a unique smooth solution $(A, \phi)$ to
  \eqref{eq:MKG} in the global Coulomb gauge on $[-r_{c}, r_{c}] 
 \times \bbR^{4}$.
\end{theorem}
Theorem~\ref{thm:OT1-simple} implies that finite time blow-up is
always accompanied by concentration of energy (i.e., $r_{c} \to 0$ at
the end of the maximal lifepan). For a precise statement, see
Theorem~\ref{thm:lwp4MKG}. In what follows we explain the ideas
involved in the proof of local existence, which lies at the heart of
Theorem~\ref{thm:OT1-simple}.

\subsubsection*{Strategy of proof in model cases}
For many other semi-linear equations, such as $\Box u = \pm
u^{\frac{d+2}{d-2}}$ or the wave map equation, a result analogous to
Theorem~\ref{thm:OT1-simple} is a rather immediate consequence of
small energy global well-posedness and finite speed of
propagation. Roughly speaking, the proof (of local existence) proceeds
in the following three steps:
\begin{itemize}[leftmargin=50px]
\item [{\it Step A}.] One truncates the initial data locally in space
  to achieve small energy.
\item [{\it Step B}.] By the small energy global well-posedness, the
  truncated data give rise to global solutions. Restricting these
  global solutions to the domain of dependence of the truncated
  regions, one obtains a family of local-in-spacetime solutions that
  agree with each other on the intersection of their domains by finite
  speed of propagation.
\item [{\it Step C}.] One patches together these solutions to obtain a
  local-in-time solution to the original initial data.
\end{itemize}
In particular, the lifespan of the solution constructed by this scheme
depends on the size of spatial truncation in Step A, which in turn is
dictated by the energy concentration scale $r_{c}$ of the initial
data.

\subsubsection*{Non-locality of \eqref{eq:MKG} in the global Coulomb gauge}
When carrying out the above strategy in our setting, however, we face
difficulties arising from non-local features of \eqref{eq:MKG} in the
global Coulomb gauge. One source of non-locality is the Gauss (or the
constraint) equation
\begin{equation} \label{eq:OT1-gauss} \rd^{\ell} e_{\ell} = \Im(f
  \overline{g}),
\end{equation}
which must be satisfied by every \eqref{eq:MKG} initial data
set. Another source is the presence of the elliptic equation for
$A_{0}$ in the global Coulomb gauge
(cf. \eqref{eq:OT2-simple:MKG-Coulomb}); in particular, finite speed
of propagation \emph{fails} in the global Coulomb gauge.

In the remainder of this subsection, we give an overview of the
techniques developed in \cite{OT1} for overcoming these issues, and
explain how these can be used to essentially execute Steps A-C above
to obtain Theorem~\ref{thm:OT1-simple} from the small energy global
well-posedness theorem proved in
\cite{Krieger:2012vj} (see Theorem~\ref{thm:smallEnergy}).


\subsubsection*{Execution of Step A: Initial data excision and gluing}
Consider the problem of truncating a \eqref{eq:MKG} initial data
set\footnote{In application $a$ obeys the global Coulomb gauge
  condition $\rd^{\ell} a_{\ell} = 0$, but this fact is irrelevant for
  the discussion here.}  $(a, e, f, g)$ to a ball $B$. A naive way to
proceed would be to apply a smooth cutoff to each of $a, e, f,
g$. However, integrating the Gauss equation \eqref{eq:OT1-gauss} by
parts over balls of large radius, we see that $e_{j}$ must in general
be nontrivial on the boundary spheres outside $B$, even if $f$ and $g$
are supported in $B$.

Instead, the idea of \emph{initial data excision and
  gluing}\footnote{We remark that similar techniques have been
  developed in mathematical general relativity, as a means to
  construct a large class of interesting initial data sets for the
  Einstein equations. Our setting involves a simpler constraint
  equation, but we require sharp techniques which are applicable at
  the critical regularity.} is as follows: Rather than just
\emph{excising} the unwanted part, we \emph{glue} it to another
initial data set (i.e., solution to the Gauss equation) which has an
explicit description, so that the Gauss equation is still
satisfied. For example, in the exterior of a ball $B$ we may glue to
the data
\begin{equation*}
  (e_{(q)j} = \frac{q}{2 \pi^{2}} \frac{x_{j}}{\abs{x}^{4}}, 0, 0, 0)
\end{equation*}
with an appropriate $q$. Note that $e_{(q)}$ is precisely the electric
field of an electric monopole of charge $q$ placed at the origin.

Using this idea we may truncate $(a, e, f, g)$ to balls to make the 
energy sufficiently small. The minimum size of these
balls, which later dictates the lifespan of the solution, can be chosen to be
proportional to the energy concentration scale. This procedure is our
analogue of Step A.

\subsubsection*{Execution of Step B: Geometric uniqueness of admissible solution to \eqref{eq:MKG}}
Though finite speed of propagation fails for \eqref{eq:MKG} in certain
gauges such as the global Coulomb gauge, it is still true up to gauge
transformations. We refer to this statement as \emph{local geometric
  uniqueness} for \eqref{eq:MKG}, and use it as a substitute for the
usual finite speed of propagation property.

Applying a suitable gauge transformation to each truncated initial
data set to impose the global Coulomb gauge condition, we are in
position to apply the small energy global well-posedness theorem
(Theorem~\ref{thm:smallEnergy}) and construct a family of global
smooth solutions. Restricting these solutions to the domain of
dependence of the truncated regions and appealing to local geometric
uniqueness, we obtain local-in-spacetime Coulomb solutions (i.e., obey
$\rd^{\ell} A_{\ell} = 0$ on the domains) which are gauge equivalent
to each other on the interaction of their domains. We refer to such
solutions as \emph{compatible pairs}\footnote{See also
  Section~\ref{subsec:cp} of the present paper, where this notion
  arises naturally from local limits of a sequence of solutions.};
geometrically, these are precisely local descriptions of a globally
defined pair of a connection and a section on local trivializations of
the bundle $L$.


\subsubsection*{Execution of Step C: Patching local Coulomb solutions}
The final task is to patch together the local-in-spacetime
descriptions of a solution (i.e., compatible pairs) to produce a
global-in-space solution $(A, \phi)$ in the global Coulomb gauge. We
first adapt a patching argument of Uhlenbeck \cite{Uhlenbeck:1982vna}
to produce a single global-in-space solution $(A', \phi')$ obeying an
appropriate $S^1$ norm bound. The fact that a gauge transformation
$\chi$ between Coulomb gauges obeys the Laplace equation $\lap \chi$,
and hence possesses improved regularity, is important for this step. The
solution $(A', \phi')$ obtained by this patching process is not
necessarily in the global Coulomb gauge; it is however
\emph{approximately Coulomb} (i.e., $\rd^{\ell} A'_{\ell}$ obeys an
improved bound), since it arose from patching together local Coulomb
solutions. It is thus possible to find a nicely behaved gauge
transformation into the exact global Coulomb gauge, leading us to the
desired local-in-time solution.

\subsection{Continuation of energy dispersed solutions}
We now describe the content of \cite{OT2}. The main theorem of
\cite{OT2} is a continuation/scattering criterion in the global
Coulomb gauge for a large energy solution $(A, \phi)$ to
\eqref{eq:MKG} in terms of its \emph{energy dispersion}
$\ED[\phi](I)$, defined as
\begin{equation} \label{eq:overview:ED} \ED[\phi](I) = \sup_{k} \bb(
  2^{-k} \nrm{P_{k} \phi}_{L^{\infty}_{t,x}(I \times \bbR^{4})} +
  2^{-2k} \nrm{\rd_{t} P_{k} \phi}_{L^{\infty}_{t,x}(I \times
    \bbR^{4})} \bb)
\end{equation}
for any time interval $I \subseteq \bbR$. A simple version is as
follows:
\begin{theorem} \label{thm:OT2-simple} Given any $E > 0$, there exist
  positive numbers $\thED = \thED(E) > 0$ and $F = F(E)$ such that the
  following holds. Let $(A, \phi)$ be a smooth solution to
  \eqref{eq:MKG} in the global Coulomb gauge (MKG-CG) on $I \times
  \bbR^{4}$ with energy $\leq E$. If $\ED[\phi](I) \leq \thED(E)$, then 
  the following a-priori $S^{1}$ norm bound holds:
  \begin{equation} \label{eq:OT2-simple} \nrm{A_{0}}_{Y^{1}[I]} + \nrm{A_{x}}_{S^{1}[I]} +
    \nrm{\phi}_{S^{1}[I]} \leq F(E).
  \end{equation}
  Moreover, $(A, \phi)$ extends as a smooth solution past finite
  endpoints of $I$.
\end{theorem}
Theorem~\ref{thm:OT2-simple} is analogous to the main result in
\cite{MR2657817} for energy critical wave maps.  Thanks to the {\it a
  priori} bound \eqref{eq:OT2-simple}, the solution $(A, \phi)$ scatters towards each infinite endpoint in the sense of
Remark~\ref{rem:scat}. For a more precise formulation, see
Theorems~\ref{thm:ED} and \ref{thm:finite-S}.

We now describe the main ideas of the proof of
Theorem~\ref{thm:OT2-simple}. In what follows, we only consider
solutions to \eqref{eq:MKG} in the global Coulomb gauge.

\subsubsection*{Decomposition of the nonlinearity}
We begin by describing the structure of the Maxwell-Klein-Gordon
system in the global Coulomb gauge (MKG-CG), which take the form
\begin{equation} \label{eq:OT2-simple:MKG-Coulomb} \left\{
    \begin{aligned}
      \lap A_{0} = & \Im(\phi \overline{\rd_{t} \phi})+ \hbox{(cubic terms)} \\
      \Box A_{j} = & \calP_{j} \Im(\phi \overline{\rd_{x} \phi}) + \hbox{(cubic terms)} \\
      \Box \phi = & - 2 i A_{\mu} \rd^{\mu} \phi+ \hbox{(cubic terms)}
    \end{aligned}
  \right.
\end{equation}
where $\calP$ is the Leray $L^{2}$-projection to the space of
divergence-free vector fields. We omitted cubic terms as they are
strictly easier to handle. The elliptic equation for $A_{0}$ allows us
to obtain the appropriate $Y^{1}$ bound once we establish $S^{1}$
bounds for $A_{x}$ and $\phi$; henceforth we focus on the wave
equations for $A_{x}$ and $\phi$.

As in the case of small energy global well-posedness
\cite{Krieger:2012vj}, the null structure of \eqref{eq:MKG} in the
global Coulomb gauge plays an essential role in the proof of
Theorem~\ref{thm:OT2-simple}.  All quadratic terms in the wave
equations exhibit null structure, i.e., cancellation in the angle
between inputs in Fourier space. There is also a secondary multilinear 
null structure in the term $2 i A_{\mu} \rd^{\mu} \phi$ which arises by plugging in
the equations for $A_{0}, A_{j}$. All of this structure is necessary
for controlling the $S^{1}$ norm of $(A, \phi)$, but it is by no means
sufficient as we discuss below.

\subsubsection*{Renormalization for large energy}
Even in the case of small energy global well-posedness
\cite{Krieger:2012vj}, the null structure alone is not enough to bound
the $S^{1}$ norm of $(A, \phi)$ due to the paradifferential term in the
$\phi$-equation
\begin{equation*}
  - \sum_{k} 2i P_{<k} A^{\mathrm{free}} \cdot \rd_{x} P_{k} \phi.
\end{equation*}
Here $A_{j}^{\mathrm{free}}$ is the free wave evolution of $A_{j}[0]
:= (A_{j}, \rd_{t} A_{j}) \rst_{\set{t = 0}}$. As in \cite{MR2100060,
  Krieger:2012vj}, we handle this term by a renormalization argument.
  More precisely, we treat the problematic term as a part of the linear operator
   and construct a \emph{paradifferential parametrix}. The construction in 
   \cite{MR2100060, Krieger:2012vj}, however, relied on smallness of the energy, 
   which we lack in our setting. Instead we consider the linear operator with a 
   \emph{frequency gap} $m$
\begin{equation*}
  \Box_{A^{\mathrm{free}}}^{p, m} \psi := \Box \psi + \sum_{k}  2 i P_{<k-m} A_{x}^{\mathrm{free}} \cdot \rd_{x} P_{k}\psi,
\end{equation*}
and gain smallness by taking $m$ sufficiently large. This idea is akin
to the gauge renormalization procedure for wave maps in
\cite{MR2657817}, where a large frequency gap was used to control the
large paradifferential term.

\subsubsection*{Role of energy dispersion}
We now describe the role of small energy dispersion
$\ED[\phi]$. Roughly speaking, small energy dispersion allows us to
gain in transversal balanced frequency interactions. This complements
the gain in parallel interactions, due to the null condition, and the
gain in the $high \times high \to low$ interactions due to the
favorable frequency balance. For instance, by interpolation with
(non-sharp) Strichartz norms controlled by the $S^{1}$ norm, we
have\footnote{Note that \eqref{eq:ED-ex} is symmetric in $\phi$ and
  $\psi$, so we may choose to use the energy dispersion norm of
  either. Note also that all nonlinearity of \eqref{eq:MKG} involve at
  least one factor of $\phi$. This is why it suffices to assume
  smallness of just $\ED[\phi]$ and not $A$.  }
\begin{equation} \label{eq:ED-ex} \nrm{P_{k}(P_{k_{1}} \phi P_{k_{2}}
    \psi)}_{L^{2}_{t,x}(I \times \bbR^{4})} \aleq 2^{-\frac{1}{2}
    \min\set{k_{1}, k_{2}}} \ED[\phi]^{\tht} \nrm{P_{k_{1}}
    \phi}_{S^{1}[I]}^{1-\tht} \nrm{P_{k_{2}} \psi}_{S^{1}[I]},
\end{equation}
which is useful when $k_{1} = k+O(1)$, $k_{2} = k + O(1)$ and $\phi,
\psi$ are at a large angle so that the output modulation is high.

To see how this gain is useful, we return to the full nonlinear system
\eqref{eq:MKG} in the global Coulomb gauge. Upon decomposing the
inputs and output into Littlewood-Paley pieces, most of the
nonlinearity exhibits an off-diagonal exponential decay in
frequency. For example, the nonlinearity in the $A_{x}$-equation obeys
\begin{equation*}
  \nrm{P_{k} \calP_{x} ( P_{k_{1}} \phi \rd_{x} \overline{P_{k_{2}} \phi})}_{N[I]} \aleq 2^{-\dlt (\abs{k-k_{1}} + \abs{k - k_{2}})} \nrm{P_{k_{1}} \phi}_{S^{1}[I]} \nrm{P_{k_{2}} \phi}_{S^{1}[I]}.
\end{equation*}
Introducing again a large frequency gap $m$, we gain smallness except
when $k_{1} = k + O_{m}(1)$ and $k_{2} = k + O_{m}(1)$. Furthermore,
thanks to the null structure, we also gain extra smallness except for
angled interaction; then we are precisely in position to use
$\ED[\phi]$. In conclusion, we gain smallness from $\ED[\phi] \leq
\eps$ for the nonlinearity in the $A_{x}$-equation.

\subsubsection*{Linear well-posedness of $\Box_{A} \psi = f$}
Unfortunately the a-priori estimate \eqref{eq:OT2-simple} does
not close yet, as there exists a nonlinear term in the $\phi$-equation
with no off-diagonal exponential decay. This part is precisely the
$low \times high \to high$ frequency and $high \times low \to low$
modulation interaction\footnote{We note that this term is where the
  secondary multilinear cancellation structure of MKG-CG is needed.} in the term $-2 i A
\cdot \rd_{x} \phi$, i.e.,
\begin{equation} \label{eq:OT2-no-exp} -2 i \sum_{\substack{k_{1} < k
      \\ k_{2} = k+O(1)}} \sum_{j < k_{1}} P_{k} Q_{<j}( P_{k_{1}}
  Q_{j} A \cdot \rd_{x} P_{k_{2}} Q_{<j} \phi).
\end{equation}
Nevertheless, this term has the redeeming feature that it can be bounded
by a \emph{divisible} norm: Given any $\veps > 0$ the interval $I$ can
be split into smaller pieces $I_{k}$ on each of which the $N$ norm of
the above expression is bounded by $\leq \veps^{2} \nrm{P_{k_{2}}
  \phi}_{S^{1}[I]}$, where the number of such intervals is
$O_{\nrm{\phi}_{S^{1}[I]}, \veps}(1)$.  For a solution $(A, \phi)$ to
\eqref{eq:MKG}, this observation leads to {\it linear well-posedness}
of the magnetic wave equation\footnote{More precisely, the observation
  regarding \eqref{eq:OT2-no-exp}, combined with the paradifferential
  parametrix construction mentioned above, implies well-posedness of
  the equation $\Box^{p, m}_{A} \psi := \Box \psi + 2 i \sum_{k}
  P_{<k-m} A_{\mu} \rd^{\mu} P_{k} \psi = f$ for sufficiently large
  $m$ with bound \eqref{eq:lin-wp}. The terms in $\Box_{A} - \Box^{p,
    m}_{A}$ also turn out to be bounded by divisible norms, which
  leads to the well-posedness of $\Box_{A} \psi = f$.} $\Box_{A} \psi
= f$ with bound
\begin{equation} \label{eq:lin-wp} \nrm{\psi}_{S^{1}[I]}
  \aleq_{\nrm{(A_{x}, \phi)}_{S^{1}[I]}} \nrm{\psi[0]}_{\dot{H}^{1}_{x} \times
    L^{2}_{x}} + \nrm{f}_{N[I]},
\end{equation}
where $\psi[0] := (\psi, \rd_{t} \psi) \rst_{\set{t = 0}}$.  The bound
\eqref{eq:lin-wp} allows us to setup an \emph{induction on energy}
scheme to establish \eqref{eq:OT2-simple}, which we now turn to
explain.


\subsubsection*{Induction on energy}
The starting point of our induction is the small energy global
well-posedness theorem \cite{Krieger:2012vj}, which implies that
\eqref{eq:OT2-simple} holds with $F(E) = C \sqrt{E}$ when the energy
$E$ is sufficiently small. Our goal is to show the existence of a
non-increasing positive function $c_{0}(\cdot)$ on the whole interval $[0,
\infty)$ such that if the conclusion of Theorem~\ref{thm:OT2-simple}
holds for energy up to $E$, then it also holds for energy up to $E +
c_{0}(E)$. Monotonicity of $c_{0}(\cdot)$ implies that it has a uniform
positive lower bound on every finite interval; thus the continuous
induction works for all energy.

In what follows, we describe the construction of $c_{0}(E)$, $F := F(E +
c_{0}(E))$ and $\eps := \eps(E + c_{0}(E))$ under the induction hypothesis
that Theorem~\ref{thm:OT2-simple} holds up to energy $E$ for some
$F(E)$ and $\eps(E)$. For the scheme to work, it is crucial to let
$c_{0}(E)$ depend only on $E$ and \emph{not} on $F(E)$ or $\eps(E)$. On
the other hand, $F$ and $\eps$ may depend on $F(E)$ and $\eps(E)$.

Let $(A, \phi)$ be a solution on $I \times \bbR^{4}$ with energy $
E + c_{0}(E)$ and $\ED[\phi] \leq \eps$. To prove \eqref{eq:OT2-simple} for
$(A, \phi)$, we compare it with another solution $(\tilde{A},
\tilde{\phi})$ with frequency truncated initial data\footnote{In the
  global Coulomb gauge, $A_{x}[0] = (A_{x}, \rd_{t} A_{x})(0)$ and
  $\phi[0] = (\phi, \rd_{t} \phi)(0)$ determine the whole initial data
  set $(a, e, f, g)$, as we can solve for $A_{0}$ in the constraint
  equation $- \lap A_{0} = \Im(\phi \overline{\rd_{t} \phi}) -
  \abs{\phi}^{2} A_{0}$.}
\begin{equation*}
  (\tilde{A}_{j}[0], \tilde{\phi}[0]) = (P_{\leq k^{\ast}} A_{j}[0], P_{\leq k^{\ast}} \phi[0])
\end{equation*}
where the `cut frequency' $k^{\ast} \in \bbR$ is chosen so that
$(\tilde{A}, \tilde{\phi})$ has energy $E$.  By taking $c_{0}(E)$ and
$\eps$ sufficiently small, we aim for the following two goals:
\begin{itemize}[leftmargin=50px]
\item[{\it Goal A}.] The energy dispersion $\ED[\tilde{\phi}](I)$ is
  sufficiently small so that the induction hypothesis applies to
  $(\tilde{A}, \tilde{\phi})$. Hence
  \begin{equation} \label{eq:S-norm-tilde}
    \nrm{\tilde{A}_{0}}_{Y^{1}[I]} + \nrm{\tilde{A}_{x}}_{S^{1}[I]} +
    \nrm{\tilde{\phi}}_{S^{1}[I]} \leq F(E).
  \end{equation}
\item[{\it Goal B}.] The difference $(B^{high},
  \psi^{high}) := (A_{\mu} - \tilde{A}_{\mu}, \phi -
  \tilde{\phi})$ obeys
  \begin{equation} \label{eq:S-norm-high}
    \nrm{B^{high}_{0}}_{Y^{1}[I]} +
    \nrm{B^{high}_{x}}_{S^{1}[I]} +
    \nrm{\psi^{high}}_{S^{1}[I]} \leq C_{E, F(E)}.
  \end{equation}
\end{itemize}
Adding \eqref{eq:S-norm-tilde} and \eqref{eq:S-norm-high}, the desired
bound \eqref{eq:OT2-simple} would follow with $F := F(E) + C_{E,
  F(E)}$.

Goal A is accomplished by showing that if $\eps$ is sufficiently
small, then $(\tilde{A}, \tilde{\phi})$ is arbitrarily close (i.e., within $\epsilon^\delta$) to the
frequency truncated solution $(P_{\leq k^{\ast}} A, P_{\leq k^{\ast}}
\phi)$ which has small energy dispersion.  For Goal B, the idea is to
view $(B^{high}, \psi^{high})$ as a perturbation
around $(\tilde{A}, \tilde{\phi})$.  To ensure that $c_{0}(E)$ is
independent of $F(E)$, we rely on two observations: First, by the
\emph{weak divisibility}\footnote{This terminology should be compared
  with full \emph{divisibility}, which means that $I$ can be split
  into a controlled number of subintervals, on each of which the
  restricted norm is arbitrarily small. Weak divisibility of the
  $S^{1}$ norm is a quick consequence of the \emph{energy inequality}
  $\nrm{\psi}_{S^{1}[I]} \aleq \nrm{\psi[0]}_{\dot{H}^{1}_{x} \times
    L^{2}_{x}} + \nrm{\Box \psi}_{N[I]}$ and (full) divisibility of
  the $N$ norm.  } of the $S^{1}$ norm, the interval $I$ can be split into
$O_{F(E)}(1)$ many subintervals $I_{k}$ on each of which we have
\begin{equation} \label{eq:divided-S-norm-tilde}
  \nrm{\tilde{A}_{0}}_{Y^{1}[I_{k}]} +
  \nrm{\tilde{A}_{x}}_{S^{1}[I_{k}]} +
  \nrm{\tilde{\phi}}_{S^{1}[I_{k}]} \aleq_{E} 1.
\end{equation}
Second, by conservation of energy for $(A, \phi)$ and $(\tilde{A},
\tilde{\phi})$, as well as the approximation $(\tilde{A},
\tilde{\phi}) \approx (P_{\leq k^{\ast}} A, P_{\leq k^{\ast}} \phi)$,
it follows that the $\dot{H}^{1}_{x} \times L^{2}_{x}$ norm of the
data for $(B^{high}, \psi^{high})$ can be
reinitialized to be of size $\aleq c_{0}(E)$ on each $I_{k}$.

With these two observations in hand, we claim that
$(B^{high}, \psi^{high})$ obeys the following
$S^{1}$ norm bound on each $I_{k}$:
\begin{equation} \label{eq:divided-S-norm-high}
  \nrm{B^{high}_{0}}_{Y^{1}[I_{k}]} +
  \nrm{B^{high}_{x}}_{S^{1}[I_{k}]} +
  \nrm{\psi^{high}}_{S^{1}[I_{k}]} \aleq_{E} c_{0}(E) + O_{F}(\eps^{\dlt}).
\end{equation}
Indeed, in the equation for $(B^{high},
\psi^{high})$, all nonlinear terms in $(B^{high},
\psi^{high})$ can be handled by taking $c_{0}(E) \ll_{E} 1$ and $\eps \ll_{F} 1$. Furthermore, exploiting small energy dispersion, all linear
terms can be made appropriately small except $- 2 i A_{\mu} \rd^{\mu}
\psi^{high}$. Nevertheless, the $S^{1}$ norm of $(A, \phi)$ on
$I$ can be assumed to be $\aleq_{E} 1$ by
\eqref{eq:divided-S-norm-tilde} and a bootstrap
assumption\footnote{More precisely, in proving
  \eqref{eq:divided-S-norm-high} we may assume, using a continuous
  induction in time, that the same bound holds with a worse
  constant. Combined with \eqref{eq:divided-S-norm-tilde} this bound
  is sufficient for ensuring that the $S^{1}$ norm of $(A, \phi)$ is
  $\aleq_{E} 1$.}; hence we can group this term with $\Box$ and use
\eqref{eq:lin-wp} (linear well-posedness of $\Box_{A} \,
\psi^{high}$) to arrive at
\eqref{eq:divided-S-norm-high}. Goal B now follows by summing up this
bound on $O_{F(E)}(1)$ intervals.

\section{Overview of the proof II: Content of the present
  paper} \label{sec:OT3} This section is a continuation of the
previous section. Section~\ref{subsec:OT3} provides an overview of the
argument in the present paper, while Section~\ref{subsec:structure} contains
an outline of the structure of the remainder of the paper.

\subsection{Blow-up analysis} \label{subsec:OT3}
Here we give an overview of the final blow-up analysis of \eqref{eq:MKG}, which is carried out in the present paper. This part is analogous to \cite{MR2657818} for energy critical wave maps. We refer to Section~\ref{sec:prelim} for the notation used below.
\subsubsection*{Main ingredients}
In addition to the continuation/scattering criteria established in
\cite{OT1, OT2} (see Theorems~\ref{thm:OT1-simple} and
\ref{thm:OT2-simple}), our blow-up analysis of \eqref{eq:MKG} relies
on the following three key ingredients:
\begin{itemize}[leftmargin=1.5em]
\item[-] ({\it Monotonicity formula for \eqref{eq:MKG}}) Besides the
  conservation of energy, we use the following \emph{monotonicity} (or
  \emph{Morawetz}) \emph{formula} for \eqref{eq:MKG}. Let $\rho :=
  \sqrt{t^{2} - \abs{x}^{2}}$ and
  \begin{equation*}
    X_{0} := \frac{1}{\rho} (t \rd_{t} + x \cdot \rd_{x})
  \end{equation*}
  be the normalized scaling vector field. To avoid the degeneracy of
  $\rho$ on $\rd C = \set{ t = \abs{x}}$, we also define the
  translates
  \begin{equation*}
    \rho_{\veps} := \sqrt{(t+\veps)^{2} - \abs{x}^{2}}, \quad X_{\veps} := \frac{1}{\rho_{\veps}} ((t + \veps) \rd_{t} + x \cdot \rd_{x}).
  \end{equation*}
  Given a smooth solution $(A, \phi)$ to \eqref{eq:MKG} on the
  truncated cone $C_{[\veps, 1]}$ satisfying
  \begin{equation*}
    \calE_{S_{1}}[A, \phi] \leq E, \quad
    \EFlux_{\rd C_{[\veps, 1]}} [A, \phi] \leq \veps^{\frac{1}{2}} E, \quad
    \G_{S_{1}} [\phi] \leq \veps^{\frac{1}{2}} E,
  \end{equation*}
  where $\EFlux_{\rd C_{[t_{0}, t_{1}]}} := \calE_{S_{t_{1}}} -
  \calE_{S_{t_{0}}}$ is the energy flux through $\rd C_{[t_{0},
    t_{1}]}$ and $\G_{S_{t}} := \frac{1}{t}\int_{S_{t}}
  \abs{\phi}^{2}$, we have
  \begin{equation} \label{eq:OT3-simple:mono}
    \begin{aligned}
      & \hskip-2em \int_{S_{1}} \mvC{X_{\veps}}_{T}[A, \phi] \, \ud x
      + \iint_{C_{[\veps, 1]}} \frac{1}{\rho_{\veps}} \abs{\iota_{X_{\veps}} F}^{2} + \frac{1}{\rho_{\veps}} \abs{(\covD_{X_{\veps}}  + \frac{1}{\rho_{\veps}}) \phi}^{2} \, \ud t \ud x \\
      \aleq & \int_{S_{\veps}} \mvC{X_{\veps}}_{T}[A, \phi] \, \ud x +
      E.
    \end{aligned}\end{equation}
  Here $\mvC{X_{\veps}}_{T}[A, \phi]$ is a non-negative weighted
  energy density; we refer to Lemma~\ref{lem:monotonicity} for an
  explicit formula for $\mvC{X_{\veps}}_{T}[A, \phi]$. We remark that
  the entire right-hand side of \eqref{eq:OT3-simple:mono} is bounded
  by $\aleq E$. Finiteness of the space-time integral term `breaks the
  scaling' and implies that $\iota_{X_{\veps}} F$ and
  $(\covD_{X_{\veps}} + \frac{1}{\rho_{\veps}}) \phi$ decay near the
  tip of the cone $C$.

\item[-] ({\it Strong local compactness result}) Given a sequence
  $(A^{(n)}, \phi^{(n)})$ of solutions whose energy is uniformly small
  and $\iota_{X} F^{(n)} \to 0$ and $(\covD_{X}^{(n)} + b) \phi^{(n)}
  \to 0$ in $L^{2}_{t,x}$ on a space-time cube for some smooth
  time-like vector field $X$ and smooth function $b$, we show that
  there exists a subsequence which converges strongly in (essentially)
  $H^{1}_{t,x}$ in a smaller subcube; see Proposition~\ref{prop:cpt}
  for more details. The proof relies on the initial data
  excision/gluing technique and the small energy global well-posedness
  theorem.

\item[-] ({\it Triviality of finite energy stationary/self-similar
    solutions}) We say that $(A, \phi)$ is a \emph{stationary}
  solution to \eqref{eq:MKG} if for some constant time-like vector
  field $Y$
  \begin{equation*}
    \iota_{Y} F = 0, \quad \covD_{Y} \phi = 0,
  \end{equation*}
  and that $(A, \phi)$ is a \emph{self-similar} solution if
  \begin{equation*}
    \iota_{X_{0}} F = 0, \quad (\covD_{X_{0}} + \frac{1}{\rho}) \phi = 0.
  \end{equation*}
  Using the method of stress tensor, we show that every smooth
  stationary or self-similar solution with finite energy is trivial
  (i.e., $F = 0$ and $\phi = 0$); see
  Propositions~\ref{prop:trivial:st} and \ref{prop:trivial:ss}.  We
  also establish a regularity result (Proposition~\ref{prop:reg}),
  which says that all stationary and self-similar solutions arising
  from the above strong local compactness result
  (Proposition~\ref{prop:cpt}) are smooth.
\end{itemize}

With these in mind, we now sketch the blow-up analysis of
\eqref{eq:MKG}, which is performed in full detail in
Section~\ref{sec:proof}.

 \subsubsection*{Finite time blow-up/non-scattering scenarios and
   initial reduction}
 Suppose that the conclusion of Theorem~\ref{thm:main} fails for a
 smooth finite energy data $(a, e, f, g)$ in the forward time
 direction. Then the corresponding smooth solution either blows up in
 finite time, or does not scatter as $t \to \infty$.  The first step
 of the blow-up analysis is to construct in both scenarios a sequence
 of global Coulomb solutions $(A^{(n)}, \phi^{(n)})$ on $[\veps_{n},
 1] \times \bbR^{4}$ (where $\veps_{n} \to 0$) obeying the following
 properties:
 \begin{itemize}[leftmargin=1.5em]
 \item[-] ({\it Bounded energy in the cone}) $\calE_{S_{t}} [A^{(n)},
   \phi^{(n)}] \leq E$ for every $t \in [\veps_{n}, 1]$
 \item[-] ({\it Small energy outside the cone}) $\calE_{(\set{t}
     \times \bbR^{4}) \setminus S_{t}} [A^{(n)}, \phi^{(n)}] \ll E$ for every $t \in [\veps_{n}, 1]$
 \item[-] ({\it Decaying flux on $\rd C$}) $\EFlux_{[\veps_{n}, 1]}
   [A^{(n)}, \phi^{(n)}] + \G_{S_{1}}[\phi^{(n)}] \leq
   \veps_{n}^{\frac{1}{2}} E$,
 \item[-] ({\it Pointwise concentration at $t = 1$}) There exist
   $k_{n} \in \bbZ$ and $x_{n} \in \bbR^{4}$ such that
   \begin{equation} \label{eq:OT3-simple:conc} 2^{-k_{n}}
     \abs{\zt_{2^{-k_{n}}} \ast \phi^{(n)}(1,x_{n})} + 2^{-2k_{n}}
     \abs{\zt_{2^{-k_{n}}} \ast \covD_{t}^{(n)} \phi^{(n)}(1, x_{n})}
     > \thcovED
   \end{equation}
   for some $\thcovED = \thcovED(E) > 0$.
 \end{itemize}
 Here $\zt$ is a smooth function supported in the unit ball $B_{1}(0)$
 and $\zt_{2^{-k}}(x) := 2^{4k} \zt(2^{k} x)$. In view of the next
 step, we require $\zt$ to be non-negative. See
 Lemma~\ref{lem:ini-seq} for details.

 Key to this construction are Theorems~\ref{thm:OT1-simple} and
 \ref{thm:OT2-simple}, which provide detailed information about finite
 time blow-up or non-scattering scenarios. In particular, the tip of
 the cone $C$ is the point of energy concentration (which exists by
 Theorem~\ref{thm:OT1-simple}) in the finite time blow-up case. ({\it
   Pointwise concentration at $t=1$}) follows from the failure of the
 energy dispersion bound in Theorem~\ref{thm:OT2-simple}. ({\it
   Decaying flux on $\rd C$}) is a consequence of the local
 conservation of energy and localized Hardy's inequality; see
 Lemma~\ref{lem:flux-decay-prelim} and
 Corollary~\ref{cor:flux-decay}. ({\it Smallness of the energy outside
   the cone}) is achieved using the initial data excision/gluing
 technique in the finite time blow-up case; in the non-scattering
 case, this property is trivial to establish.

 \subsubsection*{Elimination of the null concentration scenario}
 Thanks to the above properties, we may apply the monotonicity formula
 \eqref{eq:OT3-simple:mono} to each solution in the sequence
 $(A^{(n)}, \phi^{(n)})$.  Using the weighted energy term (i.e., the
 first term on the left-hand side) in \eqref{eq:OT3-simple:mono}, we
 show in Lemma~\ref{lem:no-null} that the null concentration scenario
 (i.e., $\abs{x_{n}} \to 1$ and $k_{n} \to \infty$) is impossible.
 Unlike in the case of wave maps \cite{MR2657818}, however, the
 weighted energy involves the covariant derivatives $\covD^{(n)}_{\mu}
 \phi^{(n)} = \rd_{\mu} \phi^{(n)} + i A^{(n)}_{\mu} \phi^{(n)}$, and the
 term involving $A^{(n)}$ could be problematic. We avoid this issue by
 first working with the gauge invariant amplitude $\abs{\phi^{(n)}}$,
 for which we have the \emph{diamagnetic inequality}
 \begin{equation*}
   \abs{X^{\mu} \rd_{\mu} \abs{\phi^{(n)}}} \leq \abs{\covD_{X} \phi^{(n)}}	
 \end{equation*}
 in the sense of distributions, for any smooth vector field $X$. We
 then transfer the bound to $\phi^{(n)}$ using the inequality
 \begin{equation*}
   \abs{\zt_{2^{-k}} \ast \phi^{(n)}} \leq \zt_{2^{-k}} \ast \abs{\phi^{(n)}},
 \end{equation*}
 which holds if $\zt$ is chosen to be non-negative.

 \subsubsection*{Nontrivial energy in a time-like region}
 The absence of the null concentration scenario implies the following
 uniform lower bound for $\phi^{(n)}$ away from the boundary at $t =
 1$: There exist $E_{1} = E_{1}(E) > 0$ and $\gmm_{1} = \gmm_{1}(E)
 \in (0, 1)$ such that
 \begin{equation} \label{eq:OT3-simple:t-like-e:t=1}
   \int_{S_{1}^{1-\gmm_{1}}} \sum_{\mu = 0}^{4} \abs{\covD_{\mu}^{(n)}
     \phi^{(n)}}^{2} + \frac{1}{r^{2}} \abs{\phi^{(n)}}^{2} \, \ud x
   \geq E_{1}.
 \end{equation}
 See Lemma~\ref{lem:t-like-e:t=1}. Using a localized version of the
 monotonicity formula \eqref{eq:OT3-simple:mono}, this lower bound can
 be propagated towards $t=0$. More precisely, there exist $E_{2} =
 E_{2}(E)$ and $\gmm_{2} = \gmm_{2}(E) \in (0, 1)$ and $E_{2} =
 E_{2}(E) > 0$ such that
 \begin{equation} \label{eq:OT3-simple:t-like-e} \int_{S^{(1-\gmm_{2})
       t}_{t}} \mvC{X_{0}}_{T}[A^{(n)}, \phi^{(n)}] \, \ud x \geq
   E_{2} \quad \hbox{ for all } t \in [\veps_{n}^{\frac{1}{2}},
   \veps_{n}^{\frac{1}{4}}].
 \end{equation}
 \subsubsection*{Final rescaling}
 Thanks to the space-time integral term in \eqref{eq:OT3-simple:mono},
 $(A^{(n)}, \phi^{(n)})$ obeys
 \begin{equation*}
   \iint_{C_{[\veps_{n}, 1]}} \frac{1}{\rho_{\veps_{n}}} \abs{\iota_{X_{\veps_{n}}} F^{(n)}}^{2} + \frac{1}{\rho_{\veps_{n}}} \abs{(\covD_{X_{\veps_{n}}}^{(n)} + \frac{1}{\rho_{\veps_{n}}}) \phi^{(n)}}^{2} \, \ud t \ud x \aleq E.
 \end{equation*}
 which implies an integrated decay of $\iota_{X_{\veps_{n}}} F^{(n)}$
 and $(\covD_{X_{\veps_{n}}}^{(n)} + \frac{1}{\rho_{\veps_{n}}})
 \phi^{(n)}$ near the tip of the cone $C$. Applying the pigeonhole
 principle and rescaling, we obtain a new sequence of solutions which
 is asymptotically self-similar. More precisely, there exist a
 sequence of solutions on $[1, T_{n}] \times \bbR^{4}$ (where $T_{n}
 \to \infty$) to \eqref{eq:MKG}, which we still denote by $(A^{(n)},
 \phi^{(n)})$, obeying the following properties (see
 Lemma~\ref{lem:final-rescale}):
 \begin{itemize}[leftmargin=1.5em]
 \item[-]({\it Bounded energy in the cone}) $\calE_{S_{t}}[A^{(n)},
   \phi^{(n)}] \leq E$ for every $t \in [1, T_{n}]$,
 \item[-]({\it Small energy outside the cone}) $\calE_{\set{t} \times
     \bbR^{4} \setminus S_{t}}[A^{(n)}, \phi^{(n)}] \ll
   E$ for every $t \in [1, T_{n}]$,
 \item[-]({\it Nontrivial energy in a time-like region}) For every $t
   \in [1, T_{n}]$ we have
   \begin{equation} \label{eq:OT3-simple:nontrivial}
     \int_{S^{(1-\gmm_{2}) t}_{t}} \mvC{X_{0}}_{T}[A^{(n)},
     \phi^{(n)}] \, \ud x \geq E_{2},
   \end{equation}
 \item[-]({\it Asymptotic self-similarity}) For every compact subset
   $K$ of the interior of $C_{[1, \infty)}$, we have
   \begin{equation} \label{eq:OT3-simple:asymp-st} \iint_{K}
     \abs{\iota_{X_{0}} F^{(n)}}^{2} + \abs{(\covD_{X_{0}}^{(n)} +
       \frac{1}{\rho}) \phi^{(n)}}^{2} \, \ud t \ud x \to 0 \quad
     \hbox{ as } n \to \infty.
   \end{equation}
 \end{itemize}

 \subsubsection*{Extraction of concentration scales and
   compactness/rigidity argument}
 Let $(A^{(n)}, \phi^{(n)})$ be a sequence obtained by the final
 rescaling argument. Using a combinatorial argument, we show in
 Lemma~\ref{lem:conc-scales} that one of the following two scenarios
 holds:
 \begin{itemize}
 \item[A.] Either we can identify a sequence of points and decreasing
   scales at which energy concentrates, or
 \item[B.] There is a uniform non-concentration of energy.
 \end{itemize}

 In Scenario A we obtain a fixed number $r > 0$ and a sequence of
 times $t_{n} \to t_{0}$, points $x_{n} \to x_{0}$ and scales $r_{n}
 \to 0$ such that
 \begin{equation*}
   \sup_{x \in B_{r}(x_{n})} \calE_{\set{t_{n}} \times B_{r_{n}}(x)}[A^{(n)}, \phi^{(n)}]
 \end{equation*}
 is uniformly small but nontrivial, and
 \begin{equation*}
   \frac{1}{4 r_{n}} \int_{t_{n}-2r_{n}}^{t_{n}+2r_{n}} \int_{B_{r}(x_{n})} \abs{\iota_{Y} F^{(n)}}^{2} + \abs{\covD_{Y}^{(n)} \phi^{(n)}}^{2} \, \ud t \ud x \to 0 \quad \hbox{ as } n \to \infty.
 \end{equation*}
 where $Y = X_{0}(t_{0}, x_{0})$. Applying Proposition~\ref{prop:cpt},
 we obtain as a limit a nontrivial finite energy solution to
 \eqref{eq:MKG} which is stationary with respect to $Y$. As discussed
 above, however, such solutions do not exist.

 In Scenario B we can cover each truncated cone $\widetilde{C}_{j} :=
 C^{1/2}_{[1/2, \infty)} \cap \set{2^{j} \leq t < 2^{j+1}}$ with
 spatial balls of radius $r = r(j)$, on each of which the energy of
 $(A^{(n)}, \phi^{(n)})$ is uniformly small and
 \begin{equation*}
   \iint_{\widetilde{C}_{j}} \abs{\iota_{X_{0}} F^{(n)}}^{2} + \abs{(\covD_{X_{0}}^{(n)} + \frac{1}{\rho}) \phi^{(n)}}^{2} \, \ud t \ud x \to 0 \quad \hbox{ as } n \to \infty.
 \end{equation*}
 Hence we are again in position to apply Proposition~\ref{prop:cpt}
 and extract a finite energy self-similar solution to \eqref{eq:MKG}
 on $C_{[1/2, \infty)}^{1/2}$. By self-similarity, this limit easily
 extends to the whole forward cone $C$. By
 \eqref{eq:OT3-simple:nontrivial} this limit is necessarily
 nontrivial, which contradicts the triviality of finite energy
 self-similar solutions.

 In conclusion, we have seen that neither of the two scenarios can
 hold, which is a contradiction. This completes the proof of the main
 theorem.

\subsection{Structure of the present paper} \label{subsec:structure}
The remainder of the paper is structured as follows. 
\subsubsection*{Section~\ref{sec:prelim}.} We provide the setup for our arguments to follow. In particular, we precisely state the results that we need from the other papers of the series \cite{OT1, OT2} in Section~\ref{subsec:prev-results}.
\subsubsection*{Section~\ref{sec:energy}.} We state and prove all the conservation laws and monotonicity formulae that are used in this paper.
\subsubsection*{Section~\ref{sec:cpt}.} We use the small energy global well-posedness theorem (Theorem~\ref{thm:smallEnergy}) and the technique of initial data excision/gluing to prove a strong local compactness statement (Proposition~\ref{prop:cpt}) that we rely on in our blow-up analysis. We also formulate a notion of weak solutions to \eqref{eq:MKG} and their local descriptions (weak compatible pairs), which naturally arise as limits from Proposition~\ref{prop:cpt}.
\subsubsection*{Section~\ref{sec:stationary-self-sim}.} We show that there does not exist any nontrivial stationary or self-similar solutions to \eqref{eq:MKG} with finite energy. We also prove regularity theorems for weak stationary or self-similar solutions to \eqref{eq:MKG} considered in Section~\ref{sec:cpt}.
\subsubsection*{Section~\ref{sec:proof}.} We finally carry out the blow-up analysis as outlined in Section~\ref{subsec:OT3}, thereby completing the proof of global well-posedness and scattering of \eqref{eq:MKG}.

\section{Preliminaries} \label{sec:prelim}
\subsection{Notation for constants and asymptotics} \label{subsec:asymp}
Throughout the paper we use $C$ for a general positive constant, which may vary from line to line. For a constant $C$ that depends on, say, $E$, we write $C = C(E)$. We write $A \aleq B$ when there exists a constant $C > 0$ such that $A \leq C B$. When the implicit constant should be regarded as small, we write $A \ll B$. The dependence of the constant is specified by a subscript, e.g., $A\aleq_{E} B$. We write $A \approx B$ when both $A \aleq B$ and $B \aleq A$ hold.

\subsection{Coordinate systems on $\bbR^{1+4}$} \label{subsec:coords}
Several different coordinate systems on $\bbR^{1+4}$ will be used in this paper. A basic choice, which has already been mentioned in the introduction, is the \emph{rectilinear coordinates} $(x^{0}, x^{1}, \ldots, x^{4})$ on $\bbR^{1+4}$, in which the Minkowski metric takes the diagonal form $\met = - (\ud x^{0})^{2} + (\ud x^{1})^{2} + \cdots + (\ud x^{4})^{2}$. Alternatively, we will often write $t = x^{0}$ and $x = (x^{1}, \ldots, x^{4})$ as well. We reserve the greek indices $\mu, \nu, \ldots$ for expressions in the rectilinear coordinates, and the latin indices $j, k, \ell, \ldots$ expressions only in terms of the spatial coordinates $x^{1}, x^{2}, x^{3}, x^{4}$.

We also introduce the \emph{polar coordinates} $(t, r, \Tht)$ on $\bbR^{1+4}$, where 
\begin{equation*}
r = \abs{x}, \quad \Tht = \frac{x}{\abs{x}} \in \bbS^{3}, 
\end{equation*}
and the \emph{null coordinates} $(u, v, \Tht)$, defined by
\begin{equation*}
u = t - r, \quad v = t + r.
\end{equation*}
We can furthermore specify a spherical coordinate system for $\Tht$, but it will not be necessary.
We also define the null vector fields $L, \uL$ as
\begin{equation*}
	L = \rd_{t} + \rd_{r} = 2 \rd_{v}, \quad
	\uL = \rd_{t} - \rd_{r} = 2 \rd_{u}.
\end{equation*}
In these coordinates, the metric takes the form 
\begin{equation*}
\met = - \ud t^{2} + \ud r^{2} + r^{2} g_{\bbS^{3}} 
				= - \ud u \ud v + r^{2}(u, v) g_{\bbS^{3}}.
\end{equation*} 
where $g_{\bbS^{3}}$ is the standard metric on $\bbS^{3}$ in the coordinates $\Tht$.

Finally, we will also use the \emph{hyperbolic polar coordinates} (in short, \emph{hyperbolic coordinates}) $(\rho, y, \Tht)$ on the future light cone $C_{(0, \infty)} = \set{(t, r, \Tht) : 0 \leq r < t}$ (see below), where
\begin{equation*}
	\rho = \sqrt{t^{2} - r^{2}}, \quad y = \tanh^{-1} (r/t).
\end{equation*}
The Minkowski metric takes the form
\begin{equation*}
	\met = - \ud \rho^{2} + \rho^{2} (\ud y^{2} + \sinh^{2} y \, g_{\bbS^{3}}).
\end{equation*}
Every constant $\rho$ hypersurface $\calH_{\rho}$ is isometric to the simply connected space of constant sectional curvature $-\frac{1}{\rho^{2}}$; in particular, $\calH_{1}$ is the hyperboloidal model for the hyperbolic $4$-space $\bbH^{4}$. Using the coordinates $(y, \Tht)$, the metric on $\bbH^{4}$ can be written as
\begin{equation*}
	g_{\bbH^{4}} = \ud y^{2} + \sinh^{2} y \, g_{\bbS^{3}}.
\end{equation*}


\subsection{Geometric notation}
To ease the transition from one coordinate system to another, we shall
use the tensor formalism.
We will denote by $\nb$ the Levi-Civita connection on $\bbR^{1+4}$ to
distinguish from coordinate vector fields $\rd_{\mu}$. The gauge
covariant connection associated to $A$ for $\bbC$-valued tensors takes
the form $\covD = \nb + i A$.  Similarly, we shall denote the
Levi-Civita connection on $\bbH^{4}$ by $\nb_{\bbH^{4}}$, and the
gauge covariant connection by $\covD_{\bbH^{4}} = \nb_{\bbH^{4}} + i
A$. We use the bold latin indices $\bfa, \bfb, \ldots$ for expressions
in a general coordinate system. We also employ the usual convention of
raising and lowering indices using the Minkowski metric $\bfm$, and
summing up repeated upper and lower indices.

We now introduce some notation for geometric subsets of $\bbR^{1+4}$
and $\bbR^{4}$. The forward light  cone
\begin{equation*}
  C := \set{(t,x) : 0 < t < \infty, \abs{x} \leq t} 
\end{equation*}
will play a central role in this paper. For $t_{0} \in \bbR$ and $I
\subset \bbR$, we define
\begin{align*}
  C_{I} :=& \set{(t, x) : t \in I, \abs{x} \leq t}, &
  \rd C_{I} :=& \set{(t, x) : t \in I, \abs{x} = t}, \\
  S_{t_{0}} :=& \set{(t, x) : t = t_{0}, \abs{x} \leq t}, & \rd
  S_{t_{0}} :=& \set{(t, x) : t = t_{0}, \abs{x} = t}.
\end{align*}
For $\dlt \in \bbR$, we define the translated cones
\begin{align*}
  C^{\dlt} :=& \set{(t,x) : \max\set{0,\dlt} \leq t < \infty, \abs{x}
    \leq t-\dlt}.
\end{align*}
The corresponding objects $C^{\dlt}_{I}$, $\rd C^{\dlt}_{I}$,
$S^{\dlt}_{t_{0}}$ and $\rd S^{\dlt}_{t_{0}}$ are defined in the
obvious manner.

We also define $B_{r}(x)$ to be the ball of radius $r$ centered at $x$
in $\bbR^{4}$.

\subsection{Frequency projections and function spaces}
Let $m_{\leq 0}$ be a smooth cutoff that equals $1$ on $\set{r \leq 1}$ and $0$ on $\set{r \geq 2}$. For $k \in \bbZ$, we define
\begin{equation*}
	m_{\leq k}(r) := m_{\leq 0}(r/2^{k}), \quad
	m_{k}(r) := m_{\leq k}(r) - m_{\leq k-1}(r).
\end{equation*}
so that $\supp \, m \subseteq \set{2^{k-1} \leq r \leq 2^{k+1}}$ and $\sum_{k} m_{k}(r) = 1$. We introduce the Littlewood-Paley projections $P_{k}$, $Q_{j}$ and $S_{\ell}$, which are used in this paper:
\begin{align*}
	P_{k} \varphi =& \calF^{-1}[m_{k}(\abs{\xi}) \calF[\varphi]], \\
	Q_{j} \varphi =& \calF^{-1}[m_{j}(\abs{\abs{\tau} - \abs{\xi}}) \calF[\varphi]], \\
	S_{\ell} \varphi =& \calF^{-1}[m_{\ell}(\abs{(\tau, \xi)}) \calF[\varphi]], 
\end{align*}
where $\calF$ [resp. $\calF^{-1}$] is the [resp. the inverse] space-time Fourier transform.

Given a normed space $X$ of function on $\bbR^{1+4}$, we define the restriction space $X(\calO)$ on a measurable subset $\calO \subseteq \bbR^{1+4}$ by the norm
\begin{equation*}
	\nrm{\varphi}_{X(\calO)} := \inf_{\psi = \varphi \hbox{ on } \calO} \nrm{\psi}_{X(\bbR^{1+4})}.
\end{equation*}
In application, the set $\calO$ is often an open set with (piecewise) smooth boundary, and hence there exists a bounded linear extension operator from $X(\calO)$ to $X(\bbR^{1+4})$ for many standard function spaces $X$ (e.g., $X = H^{1}$).

\subsection{Results from previous papers} \label{subsec:prev-results}
Here we give precise statements of results from \cite{Krieger:2012vj}
and the first two papers in the sequence \cite{OT1, OT2}, which are
used in the present paper.  Given a measurable subset $S \subseteq
\set{t} \times \bbR^{4}$ for some $t$, we define the energy of a pair
$(A, \phi)$ on $S$ by
\begin{equation*}
  \calE_{S}[A, \phi] := \int_{S} \frac{1}{2} \sum_{0 \leq \mu < \nu \leq 4} \abs{F_{\mu \nu}}^{2} 
  + \frac{1}{2} \sum_{\mu = 0}^{4} \abs{\covD_{\mu} \phi}^{2} \, \ud x.
\end{equation*}
Accordingly, for a measurable subset $S \subseteq \bbR^{4}$, we define
\begin{equation*}
  \calE_{S}[a, e, f, g] := \int_{S} \frac{1}{2} \sum_{1 \leq j < k \leq 4} \abs{(\ud a)_{j k}}^{2} + \frac{1}{2} \sum_{j=1}^{4} \abs{e_{j}}^{2}
  + \frac{1}{2} \sum_{j = 1}^{4} \abs{\covD_{j} f}^{2} 
  + \frac{1}{2} \abs{g}^{2} \, \ud x. 
\end{equation*}

The following is the main theorem of \cite{Krieger:2012vj}.
\begin{theorem}[Small energy global well-posedness in global Coulomb
  gauge] \label{thm:smallEnergy} There exists $\thE > 0$ such that the
  following holds. Let $(a, e, f, g)$ be a $\calH^{1}$ initial data
  set on $\bbR^{4}$ satisfying the global Coulomb gauge condition
  $\rd^{\ell} a_{\ell} = 0$, whose energy does not exceeding
  $\thE^{2}$, i.e.,
  \begin{equation}
    \calE_{\bbR^{4}} [a, e, f, g] \leq \thE^{2}.
  \end{equation}
  \begin{enumerate}
  \item Then there exists a unique $C_{t} \calH^{1}$ admissible
    solution $(A, \phi)$ to \eqref{eq:MKG} on $\bbR^{1+4}$ satisfying
    the global Coulomb gauge condition $\rd^{\ell} A_{\ell} = 0$ with
    $(a, e, f, g)$ as its initial data at $t = 0$, i.e., $(A_{j},
    F_{0j}, \phi, \covD_{t} \phi) \rst_{\set{t=0}} = (a_{j}, e_{j}, f,
    g)$.

  \item Moreover, $(A, \phi)$ obeys the $S^{1}$ norm bound
    \begin{equation} \label{eq:smallEnergy:sbd}
      \nrm{A_{0}}_{Y^{1}(\bbR^{1+4})} +
      \nrm{A_{x}}_{S^{1}(\bbR^{1+4})} + \nrm{\phi}_{S^{1}(\bbR^{1+4})}
      \aleq \nrm{(a, e, f, g)}_{\calH^{1}}.
    \end{equation}

  \item If the initial data set $(a, e, f, g)$ is more regular, then
    so is the solution $(A, \phi)$; in particular, if $(a, e, f, g)$
    is classical, then $(A, \phi)$ is a classical solution to
    \eqref{eq:MKG}.

  \item Finally, given a sequence $(a^{(n)}, e^{(n)}, f^{(n)},
    g^{(n)}) \in \calH^{1}(\bbR^{4})$ of Coulomb initial data sets
    such that $\calE[a^{(n)}, e^{(n)}, f^{(n)}, g^{(n)}] \leq
    \thE^{2}$ and $(a^{(n)}, e^{(n)}, f^{(n)}, g^{(n)}) \to (a, e, f,
    g)$ in $\calH^{1}(\bbR^{4})$, we have
    \begin{equation} \label{eq:smallEnergy:contDep} \nrm{A_{0}^{(n)} -
        A_{0}}_{Y^{1}(I \times \bbR^{4})} + \nrm{A_{x}^{(n)} -
        A_{x}}_{S^{1}(I \times \bbR^{4})} + \nrm{\phi^{(n)} -
        \phi}_{S^{1}(I \times \bbR^{4})} \to 0
    \end{equation}
    as $n \to \infty$, for every compact interval $I \subseteq \bbR$.
  \end{enumerate}
\end{theorem}

\begin{remark} \label{rem:SY-norms} For the purpose of the present
  paper, the precise structure of the norms $S^{1}$ and $Y^{1}$ are
  not necessary. Instead, we rely on the following embedding
  properties:
  \begin{align*}
    \nrm{\rd_{t,x} \phi}_{L^{\infty}_{t} L^{2}_{x}}+\nrm{\Box
      \phi}_{L^{2}_{t} \dot{H}^{-\frac{1}{2}}_{x}}
    \aleq& \nrm{\phi}_{S^{1}}, \\
    \nrm{\rd_{t,x} A}_{L^{\infty}_{t} L^{2}_{x}}+\nrm{\rd_{t,x}
      A}_{L^{2}_{t} \dot{H}^{\frac{1}{2}}_{x}} \aleq& \nrm{A}_{Y^{1}},
  \end{align*}
  where all norms are taken on $\bbR^{1+4}$. Furthermore, $S^{1}$ and
  $Y^{1}$ are closed under multiplication by $\eta \in
  C^{\infty}_{0}(\bbR^{1+4})$, i.e., $\eta S^{1}(\bbR^{1+4}) \subseteq
  S^{1}(\bbR^{1+4})$ and $\eta Y^{1}(\bbR^{1+4}) \subseteq
  Y^{1}(\bbR^{1+4})$; we refer to \cite[Sections 6 and 7]{OT1}.
\end{remark}

Given a positive number $E \gtrsim \thE$ and a $\calH^{1}$ initial data set
$(a, e, f, g)$ on $\bbR^{4}$ with energy $\calE[a, e, f, g] \leq E$,
we define its \emph{energy concentration scale} $r_{c} = r_{c}[a, e,
f, g]$ (with respect to energy $E$), in terms of the function $\dlt_{0}(E, \thE^{2})
= c \thE^{2} \min \set{1,  \thE^{2} E^{-1}}$ with a small universal constant $c$, by
 \begin{equation} \label{eq:EC:def}
 r_{c} = r_{c}(E)[a, e, f, g] :=
  \sup \set{r \geq 0 : \forall x \in \bbR^{4}, \ \calE_{B_{r}(x)}[a,
    e, f, g] < \dlt_{0}(E, \thE^{2}) }.
\end{equation}

The following is the main result of \cite{OT1}.
\begin{theorem}[Large energy local well-posedness theorem in global
  Coulomb gauge]\label{thm:lwp4MKG}
  Let $(a, e, f, g)$ be an $\calH^{1}$ initial data set satisfying the
  global Coulomb gauge condition $\rd^{\ell} a_{\ell} = 0$ with energy
  $\calE[a, e, f, g] \leq E$. Let $r_{c} = r_{c}[a, e, f, g]$ be
  defined as above. Then the following statements
  hold:
  \begin{enumerate}
  \item (Existence and uniqueness) There exists a unique admissible
    $C_{t} \calH^{1}$ solution $(A, \phi)$ to \eqref{eq:MKG} on $[-r_{c},
    r_{c}] \times \bbR^{4}$ satisfying the global Coulomb gauge condition with $(a, e, f, g)$ as its initial data.
  \item (A-priori $S^{1}$ regularity) We have the additional regularity
    properties
    \begin{equation*}
      A_{0} \in Y^{1}[-r_{c}, r_{c}], \quad A_{x}, \phi \in S^{1}[-r_{c}, r_{c}].
    \end{equation*}
  \item (Persistence of regularity) If the initial data set $(a, e, f, g)$ is more regular, then so is the solution $(A, \phi)$; in particular, the solution $(A, \phi)$ is
    classical if $(a, e, f, g)$ is classical.
    
  \item (Continuous dependence) Consider a sequence $(a^{(n)},
    e^{(n)}, f^{(n)}, g^{(n)})$ of $\calH^{1}$ Coulomb initial data
    sets such that $(a^{(n)}, e^{(n)}, f^{(n)}, g^{(n)}) \to (a, e, f, g)$ in $\calH^{1}$
    Then the lifespan of $(A^{(n)}, \phi^{(n)})$ eventually contains
    $[-r_{c}, r_{c}]$, and we have
    \begin{equation*}
      \nrm{A_{0} - A^{(n)}_{0}}_{Y^{1}[-r_{c}, r_{c}]}
      + \nrm{(A_{x} - A^{(n)}_{x}, \phi - \phi^{(n)})}_{S^{1}[-r_{c}, r_{c}]} \to 0 \quad \hbox{ as } n \to \infty. 
    \end{equation*}
  \end{enumerate}
  \end{theorem}
%
%
%

We also state the initial data excition/gluing theorem from
\cite{OT1}, which is used in several places in the present
paper. Given a measurable subset $O \subseteq \bbR^{4}$, the
$\calH^{1}(O)$ norm is defined as the restriction of the
$\calH^{1}(\bbR^{4})$ norm to $O$, and the space $\calH^{1}(O)$
consists of all initial data sets on $O$ with finite
$\calH^{1}(O)$ norm.
\begin{theorem}[Excision and gluing of initial data
  sets] \label{thm:gluing} Let  $B =
  B_{r_{0}}(x_{0}) \subseteq \bbR^{4}$. Then there exists an operator
  $E^{\extr}$ from $\calH^{1}(2 B \setminus \overline{B})$ to
  $\calH^{1}(\bbR^{4} \setminus \overline{B})$ satisfying the
  following properties.
  \begin{enumerate}
  \item Extension property:
    \begin{equation*}
      E^{\extr}[a, e, f, g] = (a, e, f, g) \quad \hbox{ on the annulus }
 \frac32 B \setminus \overline{B}.
    \end{equation*}  
\item  Uniform bounds:
    \begin{align}
      \nrm{E^{\extr}[a, e, f, g]}_{\calH^{1}(\bbR^{4} \setminus
        \overline{B})}
      \lesssim &  \ \nrm{(a, e, f, g)}_{\calH^{1}(2 B \setminus \overline{B})} \label{eq:gluing:H1} \\
      \calE_{\bbR^{4} \setminus \overline{B}}[E^{\extr}[a, e, f, g]]
      \aleq & \ 
      \nrm{\frac{1}{\abs{x - x_{0}}} f}_{L^{2}_{x}(2 B
        \setminus \overline{B})}^{2} + 
      \calE_{2 B \setminus \overline{B}}[a, e, f,
      g]. \label{eq:gluing:energy}
    \end{align}

  \item Regularity: The operator $E^{\extr}$ is continuous from
    $\calH^{1}(2 B \setminus \overline{B})$ to
    $\calH^{1}(\bbR^{4} \setminus \overline{B})$.  Moreover, if $(a,
    e, f, g)$ is classical, then so is $E^{\extr}[a, e, f, g]$.
  \end{enumerate}

\end{theorem}
In order to gain control of the first norm on the right in \eqref{eq:gluing:energy},
we will repeatedly use the following improvement of the classical 
Hardy inequality, which is a consequence of a result proved in \cite{OT1}, Lemma 6.5:

\begin{lemma}
\label{l:hardy+}
Let $\sigma \geq 2$. Then for any ball $B$ of radius $r$ in $\bbR^4$ we have the bounds
\begin{equation}\label{eq:hardy+}
r^{-1} \| f \|_{L^2_{x}(2B)} \lesssim \| \bfD_x f \|_{L^2_{x}(\sigma B)}
+ \sigma^{-1} \|\bfD_x f \|_{L^2_{x}(\bbR^4\setminus \overline{\sigma B})}
\end{equation}
\begin{equation}\label{eq:hardy+out}
r^{-1} \| f \|_{L^2_{x}(2B \setminus \overline{B})} \lesssim \| \bfD_x f \|_{L^2_{x}(\sigma B \setminus \overline{B})}
+ \sigma^{-1} \|\bfD_x f \|_{L^2_{x}(\bbR^4\setminus \overline{\sigma B})}
\end{equation}
\end{lemma}

Furthermore, we state the local geometric uniqueness result from
\cite{OT1}, which we use in this paper to construct compatible
pairs. For a ball $B = \set{t_{0}} \times B_{r_{0}}(x_{0}) \subseteq
\set{t_{0}} \times \bbR^{4}$, we define its \emph{future domain of
  dependence} $\calD^{+}(B)$ to be the set
\begin{equation*}
  \calD^{+}(B) := \set{(t, x) \in \bbR^{1+4} : t_{0} \leq t < r_{0}, \ \abs{x - x_{0}} < t - t_{0}}.
\end{equation*} 
Given a measurable subset $O \subseteq \bbR^{4}$, the space
$\calG^{2}(O)$ consists of locally integrable gauge transformations
such that the following semi-norm is finite:
\begin{equation*}
  \nrm{\chi}_{\calG^{2}(O)} := \nrm{\rd_{x} \chi}_{L^{4}_{x}(O)} + \nrm{\rd_{x}^{(2)} \chi}_{L^{2}_{x}(O)}.
\end{equation*}
Given a measurable subset $\calO \subseteq \bbR^{1+4}$, define
$\calO_{t} := \calO \cap (\set{t} \times \bbR^{4})$ and $I(\calO) :=
\set{t \in \bbR : \calO_{t} \neq \0}$. Note that $I(\calO)$ is
measurable and $\calO_{t}$ is measurable for almost every
$t$. Accordingly, we define the space $C_{t} \calG^{2}(\calO)$ by the
semi-norm
\begin{equation*}
  \nrm{\chi}_{C_{t} \calG^{2}(\calO)} := \esssup_{t \in I(\calO)} \bb( \nrm{\chi}_{\dot{H}^{2}_{x} \cap \dot{W}^{1, 4}_{x} \cap \mathrm{BMO} (\calO_{t})} + \nrm{\rd_{t} \chi}_{\dot{H}^{1}_{x} \cap L^{4}_{x}(\calO_{t})} + \nrm{\rd_{t}^{2} \chi}_{L^{2}_{x}(\calO_{t})} \bb).
\end{equation*}

\begin{proposition}[Local geometric uniqueness among admissible
  solutions] \label{prop:geomUni} Let $T_{0} > 0$ and let $B \subset
  \bbR^{4}$ be an open ball. Consider $C_{t} \calH^{1}$ admissible
  solutions $(A, \phi)$, $(A', \phi')$ on the region
  \begin{equation*}
    \calD := \calD^{+}(\set{0} \times B) \cap ( [0, T_{0}) \times \bbR^{4}).
  \end{equation*}
  Suppose that the respective initial data $(a, e, f, g)$ and $(a',
  e', f', g')$ are gauge equivalent on $B$, i.e., there exists
  $\underline{\chi} \in \calG^{2}(B)$ such that $(a, e, f, g) = (a' -
  \ud \underline{\chi}, e', e^{i \underline{\chi}} f', e^{i
    \underline{\chi}} g')$. Then there exists a unique gauge
  transformation $C_{t} \calG^{2}(\calD)$ such that $\chi
  \rst_{\set{0} \times B} = \underline{\chi}$ and
  \begin{equation*}
    (A, \phi) = (A' - \ud \chi, e^{i \chi} \phi') \quad \hbox{ on } \calD.
  \end{equation*}
\end{proposition}

We now pass to results from \cite{OT2}. Given an interval $I \subseteq
\bbR$, we define the \emph{energy dispersion} of a function $\phi$ on
$I \times \bbR^{4}$ by
\begin{equation} \label{eq:EDC:def} \ED[\phi](I) := \sup_{k \in \bbZ}
  \bb( 2^{-k} \nrm{P_{k} \phi}_{L^{\infty}_{t,x}(I \times \bbR^{4})} +
  2^{-2k} \nrm{P_{k} (\rd_{t} \phi)}_{L^{\infty}_{t,x}(I \times
    \bbR^{4})} \bb)
\end{equation}
The main theorem of \cite{OT2} is as follows.
\begin{theorem}[Energy dispersed regularity theorem] \label{thm:ED}   For each  $E > 0$  there exist
  positive numbers $\thED = \thED(E)$ and $F = F(E)$ such that the
  following holds.  Let $I \subseteq \bbR$ be an open interval, and
  let $(A, \phi)$ be an admissible $C_{t} \calH^{1}$ solution to
  \eqref{eq:MKG} on $I \times \bbR^{4}$ in the global Coulomb gauge
  $\rd^{\ell} A_{\ell} = 0$ with energy not exceeding $E$, i.e.,
  \begin{equation}
    \calE_{\set{t} \times \bbR^{4}}[A, \phi] \leq E \quad \hbox{ for every } t \in I.
  \end{equation}
  If, furthermore, the energy dispersion of $\phi$ on $I \times
  \bbR^{4}$ is less than or equal to $\thED(E)$, i.e.,
  \begin{equation} \label{eq:ED:small-EDC} \ED[\phi](I) \leq \thED(E),
  \end{equation}
  then the following a-priori estimate for $(A, \phi)$ on $I
  \times \bbR^{4}$ holds:
  \begin{equation}
    \nrm{A_{0}}_{Y^{1}[I]} + \nrm{A_x}_{S^{1}[I]} + \nrm{\phi}_{S^{1}[I]} \leq F(E).
  \end{equation}
\end{theorem}

We also state an continuation and scattering result for Coulomb solutions
with finite $S^{1}$ norm, which is proved in \cite{OT2}.
\begin{theorem}[Continuation and scattering of solutions with finite
  $S^{1}$ norm] \label{thm:finite-S} Let $0 < T_{+} \leq \infty$ and $(A,
  \phi)$ an admissible $C_{t} \calH^{1}$ solution to \eqref{eq:MKG} on
  $[0, T_{+}) \times \bbR^{4}$ in the global Coulomb gauge which obeys
  the bound
  \begin{equation*}
    \nrm{A_{0}}_{Y^{1}([0, T_{+}) \times \bbR^{4})} + \sup_{j=1, \ldots, 4} \nrm{A_{j}}_{S^{1}([0, T_{+}) \times \bbR^{4})}
    + \nrm{\phi}_{S^{1}([0, T_{+}) \times \bbR^{4})} < \infty.
  \end{equation*}
  Then the following statements hold.
  \begin{enumerate}
  \item If $T_{+} < \infty$, then $(A, \phi)$ extends to an admissible
    $C_{t} \calH^{1}$ solution with finite $S^{1}$ norm past $T_{+}$.
  \item If $T_{+} = \infty$, then $(A_{x}, \phi)$ scatters as $t \to
    \infty$ in the following sense: There exist a solution
    $(A^{(\infty)}_{x}, \phi^{(\infty)})$ to the system
    \begin{equation*}
      \left\{
        \begin{aligned}
          \Box A^{(\infty)}_{j} =& 0, \\
          (\Box + 2 i A^{free}_{\ell} \rd^{\ell}) \phi^{(\infty)}
          =& 0,
        \end{aligned}
      \right.
    \end{equation*}
    with initial data $A^{(\infty)}_{x}[0], \phi^{(\infty)}[0] \in
    \dot{H}^{1}_{x} \times L^{2}_{x}$ such that
    \begin{equation*}
      \sup_{j=1,\ldots,4} \nrm{A_{j}[t] - A^{(\infty)}_{j}[t]}_{\dot{H}^{1}_{x} \times L^{2}_{x}}
      + \nrm{\phi[t] - \phi^{(\infty)}[t]}_{\dot{H}^{1}_{x} \times L^{2}_{x}} \to 0 \quad \hbox{ as } T \to \infty.
    \end{equation*}
    Here $A_{x}^{free}$ can be either the homogeneous wave with $A^{free}_{x}[0] = A_{x}[0]$ or $A^{free}_{x} = A^{(\infty)}_{x}$
  \end{enumerate}
  Analogous statements hold in the past time direction as well.
\end{theorem}

\section{Conservation laws and monotonicity formulae} \label{sec:energy}
In this section, we derive key conservation laws and monotonicity formulae that will serve as a basis for proving regularity and scattering. We begin by describing the main results, deferring their proofs until later in the section. We emphasize that all statements in this section apply to \emph{admissible} $C_{t} \calH^{1}$ solutions to \eqref{eq:MKG}, unless otherwise stated.

One of the fundamental conservation laws for \eqref{eq:MKG} is that of the standard energy: Given an admissible $C_{t} \calH^{1}$ solution $(A, \phi)$ to \eqref{eq:MKG} on $I \times \bbR^{4}$, for $t_{0}, t_{1} \in I$ we have
\begin{equation} \label{eq:energyConsv}
	\calE_{\set{t_{0}} \times \bbR^{4}}[A, \phi] = \calE_{\set{t_{1}} \times \bbR^{4}}[A, \phi].
\end{equation}
For self-similar solutions, finite energy condition translates to a weighted $L^{2}$ estimate on $\calH_{\rho}$. This estimate will be used to show that they must in fact be trivial.
\begin{proposition} \label{prop:energy-H-rho}
Let $(A, \phi)$ be a smooth solution to \eqref{eq:MKG} on $C_{(0, \infty)}$ with finite energy, i.e., there exists $E > 0$ such that
\begin{equation*}
	\esssup_{t \in (0, \infty)} \calE_{S_{t}}[A, \phi] \leq E <\infty.
\end{equation*}
Suppose furthermore that $(A, \phi)$ is self-similar, i.e., $\iota_{X_{0}} F = 0$ and $(\covD_{X_{0}} + \frac{1}{\rho}) \phi = 0$, where $X_{0} = \rd_{\rho}$ in the hyperbolic coordinates $(\rho, y, \Tht)$. Then we have
\begin{equation} \label{eq:energy-H-rho}
	\int_{\calH_{\rho}}	\frac{1}{2} 
				\bb( \frac{\cosh y}{\rho^{2}}\abs{\phi}^{2} 
					+ 2 \frac{\sinh y}{\rho^{2}} \Re(\phi \overline{\covD_{y} \phi}) 
					+ \cosh y (\abs{\covD \phi}_{\calH_{\rho}}^{2} + \abs{F}_{\calH_{\rho}}^{2}) \bb) 
	\leq E,
\end{equation}
where $\abs{\covD \phi}_{\calH_{\rho}}^{2}$, $\abs{F}_{\calH_{\rho}}^{2}$ are to be defined in \eqref{eq:norms-H-rho}.
\end{proposition}

The next statement concerns the quantities
\begin{equation} \label{eq:F-G}
	\EFlux_{\rd C_{[t_{0}, t_{1}]}}[A, \phi] := \calE_{S_{t_{1}}}[A, \phi] - \calE_{S_{t_{0}}}[A, \phi], \quad
	\G_{\rd S_{t_{1}}}[\phi] := \frac{1}{t_{1}} \int_{\rd S_{t_{1}}} \abs{\phi}^{2}.
\end{equation}
Here, $\EFlux_{\rd C_{[t_{0}, t_{1}]}}$ is the \emph{energy flux} of $(A, \phi)$ through $\rd C_{[t_{0}, t_{1}]}$. For $\phi \in C_{t} (I; \dot{H}^{1}_{x})$ and $t_{1} \in I$,  observe that $\G_{\rd S_{t_{1}}}[\phi]$ is well-defined by the trace theorem. In fact, $\phi \rst_{\rd S_{t_{1}}} \in H^{1/2}(\rd S_{t_{1}})$.

\begin{lemma} \label{lem:flux-decay-prelim}
Let $(A, \phi)$ be an admissible $C_{t} \calH^{1}$ solution to \eqref{eq:MKG} on $I \times \bbR^{4}$ where $I \subset \bbR^{4}$ is an open interval. Then for every $t_{0}, t_{1} \in I$ with $t_{0} \leq t_{1}$, the following statements hold:
\begin{enumerate}
\item The energy flux on $\EFlux_{\rd C_{[t_{0}, t_{1}]}} [A, \phi]$ is non-negative and additive, i.e.,
\begin{equation} \label{eq:flux-additive}
\EFlux_{\rd C_{[t_{0}, t_{1}]}} [A, \phi] = \EFlux_{\rd C_{[t_{0}, t']}} [A, \phi] + \EFlux_{\rd C_{[t', t_{1}]}} [A, \phi] \quad \hbox{ for } t' \in [t_{0}, t_{1}].
\end{equation}
\item The following \emph{local Hardy's inequality} holds on $\rd C_{[t_{0}, t_{1}]}$:
\begin{equation} \label{eq:local-hardy:G}  
	\G_{\rd S_{t_{0}}}[\phi] + \int_{t_{0}}^{t_{1}} \G_{\rd S_{t}} [\phi] \, \frac{\ud t}{t} \leq \G_{\rd S_{t_{1}}}[\phi] + \EFlux_{\rd C_{[t_{0}, t_{1}]}}[A, \phi].
\end{equation}
	Moreover, we also have
\begin{equation} \label{eq:bound4G}
	\G_{\rd S_{t_{1}}}[\phi] \leq \calE_{(\set{t} \times \bbR^{4}) \setminus S_{t_{1}}}[A, \phi]
\end{equation}
\end{enumerate}
\end{lemma}

A consequence of Lemma~\ref{lem:flux-decay-prelim} is a simple but crucial decay result for the two quantities defined in \eqref{eq:F-G}. 
\begin{corollary} \label{cor:flux-decay}
Let $(A, \phi)$ be an admissible $C_{t} \calH^{1}$ solution to \eqref{eq:MKG} on $I \times \bbR^{4}$ where $I \subset \bbR^{4}$ is an open interval. Then the following statements hold.
\begin{enumerate}
\item If $(0, \dlt] \subseteq I$ for some $\dlt > 0$, then we have
\begin{equation} \label{eq:flux-decay:0}
	\lim_{t_{1} \to 0} \EFlux_{\rd C_{(0, t_{1}]}}[A, \phi] = 0, \quad \lim_{t_{1} \to 0} \G_{\rd S_{t_{1}}}[\phi] = 0.
\end{equation}
where $\EFlux_{\rd C_{(0, t_{1}]}}[A, \phi] := \lim_{t_{0} \to 0} \EFlux_{\rd C_{[t_{0}, t_{1}]}}[A, \phi]$.

\item If $[\dlt, \infty) \subseteq I$ for some $\dlt > 0$, then we have
\begin{equation} \label{eq:flux-decay:infty}
	\lim_{t_{0}, t_{1} \to \infty} \EFlux_{\rd C_{[t_{0}, t_{1}]}}[A, \phi] = 0, \quad \lim_{t_{1} \to \infty} \G_{\rd S_{t_{1}}}[\phi] = 0.
\end{equation}
\end{enumerate}
\end{corollary}
The statements concerning $\EFlux_{\rd C_{[t_{0}, t_{1}]}}$ follow from the monotonicity and boundedness of $\calE_{S_{t}}$, whereas those concerning $\G_{\rd S_{t_{1}}}$ follow from \eqref{eq:local-hardy:G}, \eqref{eq:bound4G}; we omit the straightforward details. 
%
%

The decay statements \eqref{eq:flux-decay:0} and \eqref{eq:flux-decay:infty} imply that the energy flux and the quantity $\G_{\rd S_{t}}[\phi]$ vanish as one approaches $(0, 0)$ or $t \to \infty$. In the ideal case when $\EFlux_{\rd C_{[t_{0}, t_{1}]}} = 0$ and $\G_{\rd S_{t_{1}}} = 0$, the solution $(A, \phi)$ enjoys an additional monotonicity formula, namely 
\begin{equation} \label{eq:monotonicity:exact}
	\int_{S_{t_{1}}} \mvC{X_{0}}_{T}[A, \phi] \, \ud x 
	+ \iint_{C_{[t_{0}, t_{1}]}} \frac{1}{\rho} \abs{\iota_{X_{0}} F}^{2}
			+ \frac{1}{\rho} \abs{(\covD_{X_{0}} + \frac{1}{\rho}) \phi}^{2} \, \ud t \ud x
	= \int_{S_{t_{0}}} \mvC{X_{0}}_{T}[A, \phi] \, \ud x 
\end{equation}
where $X_{0} = \rd_{\rho}$ in the hyperbolic coordinate system $(\rho, y, \Tht)$, $\abs{\iota_{X_{0}} F}^{2} := \met(\iota_{X_{0}} F, \iota_{X_{0}} F)$  (observe that $\abs{\iota_{X_{0}} F}^{2} \geq 0$) and $\mvC{X_{0}}_{T}[A, \phi]$ is to be defined below in Lemma~\ref{lem:monotonicity}. 
It turns out that the right-hand side is uniformly bounded by the conserved energy as $t_{0} \to 0$, thereby breaking the scaling invariance. More precisely, the first term on the left-hand side precludes null concentration of energy, whereas the second term implies that rescalings of $(A, \phi)$ are asymptotically self-similar.

In application, however, the quantities $\EFlux$ and $\G$ will be small but not necessarily zero. Hence we will rely on the following approximate version of \eqref{eq:monotonicity:exact} instead. Define
\begin{equation*}
\rho_{\veps} = \sqrt{(t+\veps)^{2} - r^{2}}, \quad X_{\veps} = \rho_{\veps}^{-1} ((t + \veps) \rd_{t} + r \rd_{r}), \quad 
\abs{\iota_{X_{\veps}} F}^{2} := \met(\iota_{X_{\veps}} F, \iota_{X_{\veps}} F).
\end{equation*}
\begin{proposition} \label{prop:monotonicity}
Let $(A, \phi)$ be an admissible $C_{t} \calH^{1}$ solution to \eqref{eq:MKG} on $[\veps, 1] \times \bbR^{4}$, where $\veps \in (0, 1)$. Suppose furthermore that $(A, \phi)$ satisfies
\begin{equation} \label{eq:monotonicity:hyp}
	\calE_{S_{1}}[A, \phi] \leq E, \quad \EFlux_{\rd C_{[\veps, 1]}} [A, \phi] \leq \veps^{\frac{1}{2}} E, \quad \G_{\rd S_{1}}[\phi] \leq \veps^{\frac{1}{2}} E.
\end{equation}
Then
\begin{equation} \label{eq:monotonicity}
\int_{S_{1}} \mvC{X_{\veps}}_{T}[A, \phi] \, \ud x 
+ \iint_{C_{[\veps, 1]}} \frac{1}{\rho_{\veps}} \abs{\iota_{X_{\veps}} F}^{2} + \frac{1}{\rho_{\veps}} \abs{(\covD_{X_{\veps}}  + \frac{1}{\rho_{\veps}}) \phi}^{2} \, \ud t \ud x \aleq E
\end{equation}
where the implicit constant is independent of $\veps, E$. We refer to Lemma~\ref{lem:monotonicity} for the computation of $\mvC{X_{\veps}}_{T}[A, \phi]$.
\end{proposition}

Using Proposition~\ref{prop:monotonicity}, we can also establish a version of \eqref{eq:monotonicity:exact} that is localized away from the boundary of the cone. This statement will be useful for propagating lower bounds in a time-like region towards $(0, 0)$.

\begin{proposition} \label{prop:monotonicity:t-like}
Let $(A, \phi)$ be an admissible $C_{t} \calH^{1}$ solution to \eqref{eq:MKG} on $[\veps, 1] \times \bbR^{4}$, where $\veps \in (0, 1)$. Suppose furthermore that $(A, \phi)$ satisfies \eqref{eq:monotonicity:hyp}. Then for $2 \veps \leq \dlt_{0} < \dlt_{1} \leq t_{0} \leq 1$, we have
\begin{equation} \label{eq:monotonicity:t-like}
	\int_{S_{1}^{\dlt_{1}}} \mvC{X_{0}}_{T}[A, \phi] \, \ud x
	\leq \int_{S_{t_{0}}^{\dlt_{0}}} \mvC{X_{0}}_{T}[A, \phi] \, \ud x + C\bb( (\dlt_{1} / t_{0})^{\frac{1}{2}}+ \abs{\log (\dlt_{1} / \dlt_{0})}^{-1} \bb) E.
\end{equation}
\end{proposition} 

The rest of this section is devoted to the proofs of the above statements, and is organized as follows. 
In Section~\ref{subsec:noether}, we discuss ways of generating divergence identities for proving the above conservation laws and monotonicity formulae. We also introduce null decomposition, which will assist our computations below. In Section~\ref{subsec:energy}, we use to prove \eqref{eq:energyConsv} and Proposition~\ref{prop:energy-H-rho}. In Section~\ref{subsec:hardy}, we introduce and prove a local version of Hardy's inequality and use it establish Lemma~\ref{lem:flux-decay-prelim}. Lastly, Section~\ref{subsec:monotonicity} is devoted to the proof of \eqref{eq:monotonicity:exact} and Propositions~\ref{prop:monotonicity}, \ref{prop:monotonicity:t-like}.

\subsection{Divergence identities and null
  decomposition} \label{subsec:noether} The goal of this subsection is
two-fold. First, we introduce methods for generating useful divergence
identities for solutions to \eqref{eq:MKG} that essentially arise from
N\"other's principle. Second, we define the notion of a null frame and
the associated null decomposition of $F$ and $\covD \phi$, which will
be useful for the computations below.

We first present the \emph{energy-momentum tensor formalism} for
generating divergence identities. This formalism is a way to exploit
N\"other's principle (continuous symmetries lead to conserved
quantities in a Lagrangian field theory) for external symmetries,
i.e., symmetries of the base manifold $\bbR^{1+4}$ of \eqref{eq:MKG}.
Let $(A, \phi)$ be a smooth solution to \eqref{eq:MKG} on an open
subset $\calO \subseteq \bbR^{1+4}$.  We define the
\emph{energy-momentum tensor} associated to $(A, \phi)$ as
\begin{equation} \label{} \EM_{\bfa \bfb}[A, \phi] =
  {}^{(\M)}\EM_{\bfa \bfb}[A]_{\bfa \bfb} + {}^{(\KG)}\EM_{\bfa
    \bfb}[A, \phi]
\end{equation}
where
\begin{align}
  {}^{(\M)}\EM_{\bfa \bfb}[A] =& \tensor{F}{_{\bfa}^{\bfc}} F_{\bfb \bfc} - \frac{1}{4} \met_{\bfa \bfb }F_{\bfc \bfd} F^{\bfc \bfd} \\
  {}^{(\KG)}\EM_{\bfa \bfb}[A, \phi] = & \Re (\covD_{\bfa} \phi
  \overline{\covD_{\bfb} \phi}) - \frac{1}{2} \met_{\bfa \bfb}
  \covD^{\bfc} \phi \overline{\covD_{\bfc} \phi}
\end{align}

Note that $\EM$ is a symmetric 2-tensor, which is gauge invariant at
each point. Moreover, since $(A, \phi)$ is a smooth solution to
\eqref{eq:MKG}, the energy-momentum tensor satisfies
\begin{equation} \label{eq:div4EMT} \nb^{\bfa} \EM_{\bfa \bfb}[A,
  \phi] = 0.
\end{equation}

Given a vector field $X$ on $\calO$, we define its \emph{deformation
  tensor} to be the Lie derivative of the metric with respect to $X$,
i.e., $\defT{X} := \calL_{X} \bfm$. Using covariant derivatives,
$\defT{X}$ also takes the form
\begin{equation*}
  \defT{X}_{\bfa \bfb} = \nb_{\bfa} X_{\bfb} + \nb_{\bfa} X_{\bfb}
\end{equation*}
We will denote the metric dual of $\defT{X}$ by $\defT{X}^{\sharp}$,
i.e., $(\defT{X}^{\sharp})^{\bfa \bfb} = \met^{\bfa \bfc} \met^{\bfb
  \bfd} \defT{X}_{\bfc \bfd}$.  From its Lie derivative definition,
the following formula for $\defT{X}_{\mu \nu}$ in coordinates can be
immediately derived:
\begin{equation} \label{eq:defTcoord} \defT{X}_{\mu \nu} = X(\bfm_{\mu
    \nu}) + \rd_{\mu} (X^{\alp}) \bfm_{\alp \nu} + \rd_{\nu}
  (X^{\alp}) \bfm_{\alp \mu}
\end{equation}

Using the deformation tensor, we now define the associated $1$- and
$0$-currents of $(A, \phi)$ as
\begin{equation} \label{eq:vsC-X}
  \begin{aligned}
    \vC{X}_{\bfa} [A, \phi] :=& \EM_{\bfa \bfb}[A, \phi] X^{\bfb}, \\
    \sC{X}[A, \phi] :=& \EM_{\bfa \bfb}[A, \phi] (\frac{1}{2}
    \defT{X}^{\sharp})^{\bfa \bfb}.
  \end{aligned}
\end{equation}
Then by \eqref{eq:div4EMT} and the symmetry of $\EM[A, \phi]_{\bfa
  \bfb}$, we obtain
\begin{equation} \label{eq:vsC-X:div} \nb^{\bfa}
  (\vC{X}_{\bfa}[A,\phi]) = \sC{X}[A, \phi].
\end{equation}

\begin{remark} 
  Taking $X = T = \rd_{t}$ in the rectilinear coordinates $(t, x^{1},
  \ldots, x^{4})$, we have $\defT{T} = 0$ (in other words, $T$ is a
  Killing vector field) and hence $\sC{T} = 0$. In fact,
  \eqref{eq:vsC-X:div} is a local form of the standard conservation of
  energy \eqref{eq:energyConsv}. We refer to
  Section~\ref{subsec:energy} for more details.
\end{remark}

For a (smooth) scalar field $\phi$ satisfying the gauge covariant wave
equation $\Box_{A} \phi = 0$, we introduce another way of generating
divergence identities. This method corresponds to using N\"other's
principle for the symmetry of the equation under the action of $\bbC$
viewed as the complexification of the gauge group $U(1)$. Given a
$\bbC$-valued function $w$ on an open subset of $\bbR^{1+4}$, we
define its associated $1$- and $0$-currents by
\begin{equation}
  \begin{aligned}
    \vC{w}_{\bfa}[A, \phi] =& (\Re \, w) \Re(\phi
    \overline{\covD_{\bfa} \phi}) - (\Im \, w) \Im(\phi
    \overline{\covD_{\bfa} \phi })
    - \frac{1}{2} \nb_{\bfa} (\Re \, w) \abs{\phi}^{2}, \\
    \sC{w}[A, \phi] =& (\Re \, w) \covD_{\bfa} \phi
    \overline{\covD^{\bfa}\phi} - \frac{1}{2} \Box (\Re\, w)
    \abs{\phi}^{2} - \nb_{\bfa} (\Im \, w) \Im(\phi
    \overline{\covD^{\bfa} \phi}).
  \end{aligned}
\end{equation}
A simple computation\footnote{Alternatively, the identity below can be
  derived by multiplying the covariant wave equation for $\phi$ by
  $\overline{w \phi}$, taking the real part and differentiating by
  parts.} shows that the following conservation law holds:
\begin{equation} \label{eq:consv2} \nb^{\bfa} (\vC{w}_{\bfa}[A, \phi])
  = \sC{w}[A, \phi].
\end{equation}
\begin{remark}
  Taking $w = -i$, we have
  \begin{equation*}
    \vC{w}_{\bfa} = \Im(\phi \overline{\covD_{\bfa} \phi}), \quad
    \sC{w} = 0,
  \end{equation*}
  and \eqref{eq:consv2} reduces to the well-known \emph{local
    conservation of charge}.
\end{remark}

Finally, we introduce the notion of a null frame and the associated
null decomposition of $\covD \phi$ and $F$, which are useful for
computations concerning the energy-momentum tensor.  At each point $p
= (t_{0}, x_{0}) \in \bbR^{1+4}$, consider orthonormal vectors
$\set{e_{\mfa}}_{\mfa=1,\ldots, 3}$ which are orthogonal to $L$ and
$\uL$. Observe that each $e_{\mfa}$ is tangent to the sphere $\rd
B_{t_{0}, r_{0}} := \set{t_{0}} \times \rd B_{r_{0}}(0)$ where $r_{0}
= \abs{x_{0}}$. The set of vectors $\set{L, \uL, e_{1}, e_{2}, e_{3}}$
at $p$ is called a \emph{null frame at $p$ associated to $L, \uL$}.

The $\bbC$-valued 1-form $\covD \phi$ can be decomposed with respect
to the null frame $\set{L, \uL, e_{\mfa}}$ as $\covD_{L} \phi$,
$\covD_{\uL} \phi$ and $\scovD_{\mfa} \phi := \covD_{e_{\mfa}} \phi$, which
is the \emph{null decomposition} of $\covD \phi$. A simple computation
shows that
\begin{equation} \label{eq:null-decomp:Dphi} {}^{(\KG)}\EM[A, \phi](L,
  L) = \abs{\covD_{L} \phi}^{2}, \ {}^{(\KG)}\EM[A, \phi](\uL, \uL) =
  \abs{\covD_{\uL} \phi}^{2}, \ {}^{(\KG)}\EM[A, \phi](L, \uL) =
  \abs{\scovD \phi}^{2}
\end{equation}
where $\abs{\scovD \phi}^{2} := \sum_{\mfa=1, \ldots, 3} \abs{\scovD_{\mfa}
  \phi}^{2}$.

Next, we define the \emph{null decomposition} of the 2-form $F$ with
respect to $\set{L, \uL, e_{\mfa}}$ as
\begin{equation*}
  \alp_{\mfa} := F(L, e_{\mfa}), \quad
  \ualp_{\mfa} := F(\uL, e_{\mfa}), \quad
  \varrho := \frac{1}{2} F(L, \uL), \quad
  \sgm_{\mfa\mfb} := F(e_{\mfa}, e_{\mfb}).
\end{equation*}
Note that $\varrho$ is a function, $\alp_{\mfa}, \ualp_{\mfb}$ are 1-forms
on $\rd B_{t_{0}, r_{0}}$ and $\sgm_{\mfa\mfb}$ is a 2-form on $\rd
B_{t_{0}, r_{0}}$. We define their pointwise absolute values as
\begin{equation*}
  \abs{\alp}^{2} := \sum_{\mfa=1, \ldots, 3} \alp_{\mfa}^{2}, \quad
  \abs{\ualp}^{2} := \sum_{\mfa=1, \ldots, 3} \ualp_{\mfa}^{2}, \quad
  \abs{\sgm}^{2} := \sum_{1 \leq \mfa < \mfb \leq 3} \sgm_{\mfa\mfb}^{2}.
\end{equation*}
This decomposition leads to the following simple formulae for the $L,
\uL$ components of ${}^{(\M)} \EM$:
\begin{equation} \label{eq:null-decomp:F} {}^{(\M)}\EM[A](L, L) =
  \abs{\alp}^{2}, \quad {}^{(\M)}\EM[A](\uL, \uL) = \abs{\ualp}^{2},
  \quad {}^{(\M)}\EM[A](L, \uL) = \abs{\varrho}^{2} + \abs{\sgm}^{2}.
\end{equation}

\subsection{The standard energy identity and proof of
  Proposition~\ref{prop:energy-H-rho}} \label{subsec:energy} Consider
the vector field $T$, which is equal to the coordinate vector field
$\rd_{t}$ in the rectilinear coordinates $(t, x^{1}, \ldots,
x^{4})$. It can be easily checked that $T$ is Killing, i.e., $\defT{T}
= 0$.  Contracting $T$ with the energy-momentum tensor $\EM[A, \phi]$,
we then obtain the \emph{local conservation of energy}, i.e., given a
smooth solution $(A, \phi)$ to \eqref{eq:MKG} on an open subset $\calO
\subseteq \bbR^{1+4}$, we have
\begin{equation} \label{eq:energyConsv:local} \nb^{\bfa}
  (\vC{T}_{\bfa}[A, \phi] )= 0 \quad \hbox{ on } \calO.
\end{equation}
Since $T = \frac{1}{2}(L + \uL)$, we have
\begin{align}
  \vC{T}_{L} [A, \phi]=& \EM[A, \phi](T, L) = \frac{1}{2} (\abs{\covD_{L} \phi}^{2} + \abs{\scovD \phi}^{2}) + \frac{1}{2} (\abs{\alp}^{2} + \abs{\varrho}^{2} + \abs{\sgm}^{2}),  \label{eq:vC-T:L} \\
  \vC{T}_{\uL} [A, \phi]=& \EM[A, \phi](T, \uL) = \frac{1}{2}
  (\abs{\covD_{\uL} \phi}^{2} + \abs{\scovD \phi}^{2}) + \frac{1}{2}
  (\abs{\ualp}^{2} + \abs{\varrho}^{2} +
  \abs{\sgm}^{2}). \label{eq:vC-T:uL}
\end{align}
Given a (measurable) subset $S \subseteq \set{t} \times \bbR^{4}$ for
some $t \in \bbR$, the above computation implies
\begin{equation*}
  \calE_{S}[A, \phi] = \int_{S} \vC{T}_{T}[A, \phi]\, \ud x.
\end{equation*}

We are now ready to give a quick proof of \eqref{eq:energyConsv}. For
a classical solution $(A, \phi)$ in the class $C_{t} \calH^{1}
([t_{0}, t_{1}] \times \bbR^{4})$, the standard energy conservation
\eqref{eq:energyConsv} follows by integrating
\eqref{eq:energyConsv:local} over $(t_{0}, t_{1}) \times \bbR^{4}$ and
applying the divergence theorem. The case of an admissible solution
then easily follows by approximation.

We conclude this subsection with a proof of
Proposition~\ref{prop:energy-H-rho}.
\begin{proof} [Proof of Proposition~\ref{prop:energy-H-rho}]
  Note that $X_{0} = \rd_{\rho}$ and $T = \cosh y \rd_{\rho} - \sinh y
  (\rho^{-1} \rd_{y})$ in the hyperbolic coordinates $(\rho, y,
  \Tht)$.
  In the following computation, we use the orthonormal frame
  $\set{\rd_{\rho}, \rho^{-1} \rd_{y}, e_{\mfa}}$ at each point, where
  $\set{e_{\mfa}}_{\mfa=1,2,3}$ is an orthonormal frame tangent to the
  constant $\rho, y$ sphere as before. Then we compute
  \begin{align*}
    {}^{(\KG)}\EM[A, \phi](\rd_{\rho}, \rd_{\rho})
    =& \frac{1}{2} \bb( \abs{\covD_{\rho} \phi}^{2} + \abs{\rho^{-1} \covD_{y} \phi}^{2} + \abs{\scovD \phi}^{2} \bb) \\
    {}^{(\KG)}\EM[A, \phi](\rd_{\rho}, \rho^{-1} \rd_{y})
    =& \Re(\covD_{\rho} \phi \overline{\rho^{-1} \covD_{y} \phi}), \\
    {}^{(\M)} \EM[A, \phi](\rd_{\rho}, \rd_{\rho})
    =& \frac{1}{2} F(\rd_{\rho}, \rho^{-1} \rd_{y})^{2}
    + \frac{1}{2} \sum_{\mfa = 1, \ldots, 3} F(\rd_{\rho}, e_{\mfa})^{2} \\
    & + \frac{1}{2} \sum_{\mfa = 1, \ldots, 3} \rho^{-2} F(\rd_{y},
    e_{\mfa})^{2}
    + \frac{1}{2} \sum_{1 \leq \mfa < \mfb \leq 3} F(e_{\mfa}, e_{\mfb})^{2}, \\
    {}^{(\M)} \EM[A, \phi](\rd_{\rho}, \rho^{-1} \rd_{y})
    =& \sum_{\mfa = 1, \ldots, 3} F(\rd_{\rho}, e_{\mfa}) F(\rho^{-1}
    \rd_{y}, e_{\mfa}).
  \end{align*}
  By the self-similarity conditions $\iota_{\rd_{\rho}} F =
  F(\rd_{\rho}, \cdot) = 0$ and $(\covD_{\rho} + \frac{1}{\rho}) \phi
  = 0$, we have
  \begin{align*}
    \vC{T}_{\rho}[A, \phi]
    =& \cosh y \EM[A, \phi](\rd_{\rho}, \rd_{\rho}) - \sinh y \EM[A, \phi](\rho^{-1} \rd_{y}, \rd_{\rho}) \\
    =& \frac{1}{2} \bb( \frac{\cosh y}{\rho^{2}}\abs{\phi}^{2} + 2
    \frac{\sinh y}{\rho^{2}} \Re(\phi \overline{\covD_{y} \phi}) +
    \cosh y (\abs{\covD \phi}_{\calH_{\rho}}^{2} +
    \abs{F}_{\calH_{\rho}}^{2}) \bb)
  \end{align*}
  where
  \begin{equation} \label{eq:norms-H-rho} \abs{\covD
      \phi}_{\calH_{\rho}}^{2} := (g_{\calH_{\rho}}^{-1})^{\bfa \bfb}
    \covD_{\bfa} \phi \overline{\covD_{\bfb} \phi}, \quad
    \abs{F}_{\calH_{\rho}}^{2} := \frac{1}{2}
    (g_{\calH_{\rho}}^{-1})^{\bfa \bfc} (g_{\calH_{\rho}}^{-1})^{\bfb
      \bfd} F_{\bfa \bfb} F_{\bfc \bfd},
  \end{equation}
  and $g_{\calH_{\rho}}^{-1} = \rho^{-2} \rd_{y} \cdot \rd_{y} +
  \sum_{\mfa=1,2,3} e_{\mfa} \cdot e_{\mfa}$ is the induced metric on
  $\calH_{\rho}$.

  We are ready to complete the proof. Denote by $\calH_{>\rho}$ the
  region $\set{(\rho', y', \Tht') : \rho' > \rho}$. Integrate
  \eqref{eq:energyConsv:local} over the region $C_{(0, t)} \cap
  \calH_{> \rho}$, whose boundary is $S_{t} \cup (\calH_{\rho} \cap
  C_{(0, t)})$, and apply the divergence theorem. Then taking $t \to
  \infty$, the desired estimate \eqref{eq:energy-H-rho} on
  $\calH_{\rho}$ follows. \qedhere
\end{proof}

\subsection{A localized Hardy's inequality and proof of Lemma~\ref{lem:flux-decay-prelim}} \label{subsec:hardy}
We begin by stating a very general identity (valid for any dimension $d \geq 3$), which can be thought of as Hardy's inequality with all the errors terms explicit.
\begin{lemma} \label{lem:local-hardy}
Let $\phi$ be a smooth $\bbC$-valued function and $A$ be a smooth 1-form on $\bbR^{d}$ $(d \geq 3)$. Then for $0 < r_{1} < r_{2}$, we have
\begin{equation} \label{eq:local-hardy:general}
\begin{aligned}
& \hskip-2em
	\int_{r_{1}}^{r_{2}} \int \frac{1}{r^{2}} \abs{\phi}^{2} r^{d-1} \, \ud \sgm_{\bbS^{d-1}} \, \ud r 
	+ \int_{r_{1}}^{r_{2}} \int \abs{\frac{2}{d-2} \covD_{r} \phi + \frac{1}{r} \phi}^{2} r^{d-1} \, \ud \sgm_{\bbS^{d-1}} \, \ud r \\
	= & \bb( \frac{2}{d-2} \bb)^{2} \int_{r_{1}}^{r_{2}} \int \abs{\covD_{r} \phi}^{2} r^{d-1} \, \ud \sgm_{\bbS^{d-1}} \, \ud r + \frac{2}{d-2} \int \abs{\phi}^{2} r^{d-2} \, \ud \sgm_{\bbS^{d-1}} \bb\vert_{r = r_{1}}^{r_{2}} .
\end{aligned}
\end{equation}
\end{lemma}
We omit the proof, which is a simple algebra plus an application of the fundamental theorem of calculus in $r$. Specializing to $d = 4$ and rearranging some terms, we obtain
\begin{equation} \label{eq:local-hardy}
\begin{aligned}
& \hskip-2em
	\int_{\set{r=r_{1}}} \frac{\abs{\phi}^{2}}{r} r^{3} \ud \sgm_{\bbS^{3}} + \int_{r_{1}}^{r_{2}} \int \frac{1}{r^{2}} \abs{\phi}^{2} r^{3} \, \ud \sgm_{\bbS^{3}} \, \ud r 
	+ \int_{r_{1}}^{r_{2}} \int \abs{r^{-1} \covD_{r} (r \phi)}^{2} r^{3} \, \ud \sgm_{\bbS^{3}} \, \ud r \\
	= & 	\int_{\set{r=r_{2}}} \frac{\abs{\phi}^{2}}{r} r^{3} \ud \sgm_{\bbS^{3}} + \int_{r_{1}}^{r_{2}} \int \abs{\covD_{r} \phi}^{2} r^{3} \, \ud \sgm_{\bbS^{3}} \, \ud r.
\end{aligned}
\end{equation}

The last term on the left-hand side of \eqref{eq:local-hardy} is always non-negative; moreover, for $\phi \in \calS(\bbR^{4})$, the first term on the right-hand side vanishes as $r_{2} \to \infty$. By approximation, the following gauge invariant version of Hardy's inequality on $\bbR^{4}$ follows.
\begin{corollary} \label{cor:Hardy:const-t}
Let $\phi, A \in \dot{H}^{1}(\bbR^{4})$. Then $r^{-1} \phi \in L^{2}(\bbR^{4})$ and $\phi \rst_{\rd B_{r}} \in L^{2}(\rd B_{r})$ for every $r >0$. Moreover, we have
\begin{equation} \label{eq:Hardy:const-t}
	\nrm{\frac{\phi}{r}}_{L^{2}(\bbR^{4})}^{2}
	+ \sup_{r > 0} \frac{1}{r} \nrm{\phi}_{L^{2}(\rd B_{r})}^{2}
	\leq \nrm{\covD_{r} \phi}_{L^{2}(\bbR^{4})}^{2}.
\end{equation}
\end{corollary}

We are ready to establish Lemma~\ref{lem:flux-decay-prelim}.
\begin{proof} [Proof of Lemma~\ref{lem:flux-decay-prelim}]
We first consider the case when $(A, \phi)$ is smooth. Then by local conservation of energy, we have
\begin{equation*}
	\EFlux_{\rd C_{[t_{0}, t_{1}]}} = \frac{1}{2} \int_{\rd C_{[t_{0}, t_{1}]}} \vC{T}_{L}[A, \phi] r^{3} \, \ud v \ud \sgm_{\bbS^{3}}
\end{equation*}
and hence the non-negativity and additivity are obvious. The first local Hardy's inequality \eqref{eq:local-hardy:G} is a consequence of \eqref{eq:local-hardy} applied to the hypersurface $\rd C_{[t_{0}, t_{1}]} = \set{u = 0, \, r \in [t_{0}, t_{1}]}$ in the coordinate system $(u, r, \Tht)$, whereas the second local Hardy's inequality \eqref{eq:bound4G} follows from a similar argument used to derive Corollary~\ref{cor:Hardy:const-t}.

Now we turn to the general case. Since $(A, \phi)$ is an admissible $C_{t} \calH^{1}$ solution, there exists a sequence of smooth solutions converging to $(A, \phi)$ in $C_{t} \calH^{1}(I \times \bbR^{4})$. Since all quantities in the conclusions of the lemma are continuous with respect to the $C_{t} \calH^{1}(I \times \bbR^{4})$ topology, the general case follows from the smooth case by approximation. \qedhere
\end{proof}

\subsection{Monotonicity formulae and proofs of Propositions~\ref{prop:monotonicity}, \ref{prop:monotonicity:t-like}} \label{subsec:monotonicity}
Here we derive monotonicity formulae associated with the vector fields $X_{\veps}$, which are defined in the polar coordinates as
\begin{equation} \label{eq:X-rho-eps}
	X_{\veps} = \frac{1}{\rho_{\veps}} ((t+\veps) \rd_{t} + r \rd_{r}), \quad
	\rho_{\veps} = \sqrt{(t+\veps)^{2} - r^{2}},
\end{equation}
where $\veps \geq 0$, $t > -\veps$. 

The starting point for derivation of the monotonicity formula \eqref{eq:monotonicity:exact}, as well as Propositions~\ref{prop:monotonicity} and~\ref{prop:monotonicity:t-like}, is to contract the energy-momentum tensor $\EM$ with one of the vector fields $X_{\veps}$. Due to the unfavorable contribution of ${}^{(\KG)} \EM$, however, several additional modifications are necessary. To simplify the discussion, we first restrict to the case $\veps = 0$. 
The reader should keep in mind that the general case follows simply by translating in time by $\veps$.

Using the formula \eqref{eq:defTcoord}, we compute
\begin{equation*}
	\frac{1}{2} \defT{X_{0}}^{\sharp} 
	= \frac{1}{\rho^{3}} \bb(  \rd_{y} \cdot \rd_{y} + \frac{1}{\sinh^{2} y} (g^{-1}_{\bbS^{3}}) \bb) 
	= \frac{1}{\rho} (\met^{-1} + X_{0} \cdot X_{0}).
\end{equation*}
Hence we have
\begin{align}
	\sC{X_{0}} 
=& 	{}^{(\M)}\EM_{\bfa \bfb} (\frac{1}{2} \defT{X_{0}}^{\sharp})^{\bfa \bfb} 
	+ {}^{(\KG)}\EM_{\bfa \bfb} (\frac{1}{2} \defT{X_{0}}^{\sharp})^{\bfa \bfb} \notag \\
=&	 \frac{1}{\rho} \abs{\iota_{X_{0}} F}^{2} + \frac{1}{\rho} \abs{\covD_{X_{0}} \phi}^{2} - \frac{1}{\rho} \covD_{\bfa} \phi \overline{\covD^{\bfa} \phi}. \label{eq:monotonicity:deriv:1}
\end{align}
where $\abs{\iota_{X_{0}} F}^{2} = \met(\iota_{X_{0}} F, \iota_{X_{0}} F) \geq 0$, since $X_{0}$ is time-like. 
The first term on \eqref{eq:monotonicity:deriv:1} is satisfactory in view of our goal \eqref{eq:monotonicity:exact}, but the rest is not. To remove the last term, we use the currents $\vC{w_{0}}$ and $\sC{w_{0}}$ with $w_{0} = \frac{1}{\rho}$ and compute
\begin{align}
	\sC{X_{0}} + \sC{w_{0}}
	=& \frac{1}{\rho} \abs{\iota_{X_{0}} F}^{2} + \frac{1}{\rho} \abs{\covD_{X_{0}} \phi}^{2} 
		- \frac{1}{\rho^{3}}  \abs{\phi}^{2}. \label{eq:monotonicity:deriv:2}
\end{align}

Now we introduce an auxiliary divergence identity, which is related to Hardy's inequality in the $\rho$ variable. Define $\vC{\Hardy_{0}}[\phi]$ in the hyperbolic coordinates $(\rho, y, \Tht)$ by
\begin{equation} \label{eq:vC-Hardy}
	\vC{\Hardy_{0}}_{\rho}[\phi] := - \frac{\abs{\phi}^{2}}{\rho^{2}}, 
\end{equation}
where the remaining components are set to be zero. Define also
\begin{equation} \label{eq:sC-Hardy}
	\sC{\Hardy_{0}}[\phi] := \frac{2}{\rho^{3}} \abs{\phi}^{2} + \frac{1}{\rho^{2}} \rd_{\rho} \abs{\phi}^{2}.
\end{equation}
Then a simple computation shows that
\begin{equation} \label{eq:div-Hardy}
	\nb^{\bfa} (\vC{\Hardy_{0}}_{\bfa}[\phi]) = \sC{\Hardy_{0}}[\phi].
\end{equation}
Since $\rd_{\rho} \abs{\phi}^{2} = 2 \Re(\phi \overline{\covD_{\rho} \phi})$ and $X_{0} = \rd_{\rho}$, we arrive at
\begin{equation} \label{eq:monotonicity:deriv:3}
	\sC{X_{0}} + \sC{w_{0}} + \sC{\Hardy_{0}}
=	\frac{1}{\rho} \abs{\iota_{X_{0}} F}^{2} + \frac{1}{\rho} \abs{(\covD_{X_{0}} + \frac{1}{\rho}) \phi}^{2},
\end{equation}
which is precisely the integrand in the space-time integral in \eqref{eq:monotonicity:exact}. 

The preceding computation suggests that we should define a new 1- and 0-currents by $\vC{X_{0}} + \vC{w_{0}} + \vC{\Hardy_{0}}$ and $\sC{X_{0}} + \sC{w_{0}} + \sC{\Hardy_{0}}$, respectively. To make the $L$ and $\uL$ components of the 1-current look more favorable, however, it turns out to be convenient to add in an auxiliary current $\vC{\Null_{0}}$ defined by 
\begin{equation} \label{eq:vC-Null}
	\vC{\Null_{0}}_{L}[\phi] = \frac{1}{2 r^{3}} L( r^{3} \frac{t}{\rho r} \abs{\phi}^{2} ), \quad
	\vC{\Null_{0}}_{\uL}[\phi] = - \frac{1}{2 r^{3}} \uL( r^{3} \frac{t}{\rho r} \abs{\phi}^{2} ), 
\end{equation}
where the remaining components are set to be zero. By equality of mixed partials $L \uL = 4 \rd_{v} \rd_{u} = 4 \rd_{u} \rd_{v} = \uL L$, it follows that
\begin{equation} \label{eq:div-Null}
	\nb^{\bfa} (\vC{\Null_{0}}_{\bfa}[\phi]) = 0.
\end{equation}

For $\mvC{X_{0}} := \vC{X_{0}} + \vC{w_{0}} + \vC{\Hardy_{0}} + \vC{\Null_{0}}$, we claim that
\begin{align} 
	\mvC{X_{0}}_{L} 
=&	\frac{1}{2} \bb(\frac{v}{u}\bb)^{\frac{1}{2}} (\abs{r^{-1}\covD_{L}(r \phi)}^{2} + \abs{\alp}^{2})
	+ \frac{1}{2} \bb(\frac{u}{v}\bb)^{\frac{1}{2}} \bb( \abs{\scovD \phi}^{2} + \frac{\abs{\phi}^{2}}{r^{2}} + \abs{\varrho}^{2} + \abs{\sgm}^{2} \bb), 
						\label{eq:monotonicity:mvC-0-L} \\
	\mvC{X_{0}}_{\uL}
=&	\frac{1}{2} \bb(\frac{u}{v}\bb)^{\frac{1}{2}} (\abs{r^{-1}\covD_{\uL}(r \phi)}^{2} + \abs{\ualp}^{2})
	+ \frac{1}{2} \bb(\frac{v}{u}\bb)^{\frac{1}{2}} \bb( \abs{\scovD \phi}^{2} + \frac{\abs{\phi}^{2}}{r^{2}} + \abs{\varrho}^{2} + \abs{\sgm}^{2} \bb). 
						\label{eq:monotonicity:mvC-0-uL} 
\end{align}
We will prove \eqref{eq:monotonicity:mvC-0-L}, leaving the task of verifying \eqref{eq:monotonicity:mvC-0-uL} to the reader. Using the relations
\begin{equation*}
\rho^{2} = uv, \quad X_{0} = \frac{1}{2} (\frac{v}{\rho} L + \frac{u}{\rho} \uL), 
\end{equation*}
and the null decomposition formulae \eqref{eq:null-decomp:Dphi}, \eqref{eq:null-decomp:F}, we have
\begin{align*}
	\vC{X_{0}}_{L}[A, \phi]
=&	\frac{1}{2} \bb( \frac{v}{\rho} \abs{\covD_{L} \phi}^{2} + \frac{u}{\rho} \abs{\scovD \phi}^{2} \bb)
	+ \frac{1}{2} \bb( \frac{v}{\rho} \abs{\alp}^{2} + \frac{u}{\rho} (\abs{\varrho}^{2} + \abs{\sgm}^{2}) \bb).
\end{align*}
On the other hand, we compute
\begin{align*}
	\vC{w_{0}}_{L}[A, \phi]
=	\frac{1}{\rho} \Re(\phi \overline{\covD_{L} \phi}) + \frac{1}{2} \frac{1}{\rho v} \abs{\phi}^{2}, \quad
	\vC{\Hardy_{0}}_{L}[\phi] 
=	- \frac{1}{\rho v} \abs{\phi}^{2}.
\end{align*}
To prove \eqref{eq:monotonicity:mvC-0-L}, it suffices to verify
\begin{equation} \label{eq:monotonicity:mvC-0-L:key}
	\frac{1}{2} \frac{v}{\rho} \abs{\covD_{L} \phi}^{2}
	+ \vC{w_{0}}_{L}[A, \phi] + \vC{\Hardy_{0}}_{L}[\phi] + \vC{\Null_{0}}_{L}[\phi]
= \frac{1}{2} \frac{v}{\rho} \abs{r^{-1} \covD_{L} (r \phi)}^{2} + \frac{1}{2} \frac{u}{\rho} \frac{\abs{\phi}^{2}}{r^{2}}.
\end{equation}
For this purpose, it is convenient to work with $\psi = r \phi$. We have
\begin{align*}
	\hbox{LHS of }\eqref{eq:monotonicity:mvC-0-L:key}
	=& \frac{1}{2} \frac{v}{\rho} \abs{\covD_{L}(\psi/r)}^{2} 
		+ \frac{1}{\rho r} \Re(\psi \overline{\covD_{L}(\psi/r)}) + \frac{1}{2} \frac{1}{\rho v} \frac{\abs{\psi}^{2}}{r^{2}} 
		- \frac{1}{\rho v} \frac{\abs{\psi}}{r^{2}}
		+ \frac{1}{2 r^{3}} L(\frac{t}{\rho} \abs{\psi}^{2}) \\
	=& \frac{1}{2} \frac{v}{\rho} \abs{r^{-1} \covD_{L} \psi}^{2} + \frac{1}{2} \bb(  \frac{v}{\rho r^{2}} - \frac{2}{\rho r} - \frac{1}{\rho v} + \frac{1}{r} L(t/\rho) \bb) \frac{\abs{\psi}^{2}}{r^{2}}
\end{align*}
Since $r^{-1} L(t/\rho) = 1/(\rho r) - t/(\rho r v) = 1/(\rho v)$, we see that
\begin{align*}
\frac{v}{\rho r^{2}} - \frac{2}{\rho r} - \frac{1}{\rho v} + \frac{1}{r} L(t/\rho)
= \frac{v}{\rho r^{2}} - \frac{2}{\rho r} 
= \frac{u}{\rho r^{2}},
\end{align*}
which establishes \eqref{eq:monotonicity:mvC-0-L:key}, and hence \eqref{eq:monotonicity:mvC-0-L}.

We now return to the general case $\veps \geq 0$. Define $\vC{X_{\veps}}$, $\vC{w_{\veps}}$, $\vC{\Hardy_{\veps}}$, $\vC{\Null_{\veps}}$ and their 0-current counterparts by pulling back the $\veps = 0$ versions defined above along the map $(t, r, \Tht) \mapsto (t+\veps, r, \Tht)$. For $\vC{X_{\veps}}$, $\vC{w_{\veps}}$, $\sC{X_{\veps}}$ and $\sC{w_{\veps}}$, note that this definition agrees with that from Section~\ref{subsec:noether} using $X_{\veps}$ as in \eqref{eq:X-rho-eps} and $w_{\veps} := 1/\rho_{\veps}$.
Let
\begin{equation} \label{eq:monotonicity:mvC-msC}
\begin{aligned}
	\mvC{X_{\veps}}[A, \phi] :=& \vC{X_{\veps}}[A, \phi] + \vC{w_{\veps}}[A, \phi] + \vC{\Hardy_{\veps}}[\phi] + \vC{\Null_{\veps}}[\phi], \\
	\msC{X_{\veps}}[A, \phi] :=& \sC{X_{\veps}}[A, \phi] + \sC{w_{\veps}}[A, \phi] + \sC{\Hardy_{\veps}}[\phi].
\end{aligned}
\end{equation}
We summarize the discussion so far in the following lemma, which follows easily by pulling back the above computations along $(t, r, \Tht) \mapsto (t+\veps, r, \Tht)$.
\begin{lemma} \label{lem:monotonicity}
Let $(A ,\phi)$ be a smooth solution to \eqref{eq:MKG} on an open subset $\calO \subseteq C_{(0, \infty)}$.
The 1- and 0-currents $\mvC{X_{\veps}}[A, \phi]$ and $\msC{X_{\veps}}$ obeys the divergence identity
\begin{equation} \label{eq:monotonicity:local}
	\nb^{\bfa} (\mvC{X_{\veps}}_{\bfa}[A, \phi]) = \msC{X_{\veps}}[A, \phi],
\end{equation}
where $\msC{X_{\veps}} = \msC{X_{\veps}}[A, \phi]$ takes the form
\begin{equation} \label{eq:monotonicity:msC}
	\msC{X_{\veps}} = \frac{1}{\rho_{\veps}} \abs{\iota_{X_{\veps}} F}^{2} + \frac{1}{\rho_{\veps}} \abs{(\covD_{X_{\veps}} + \frac{1}{\rho_{\veps}}) \phi}^{2}.
\end{equation}
Here, $\abs{\iota_{X_{\veps}} F}^{2} = \met(\iota_{X_{\veps}} F, \iota_{X_{\veps}} F) \geq 0$. Moreover, the $L$ and $\uL$ components of $\mvC{X_{\veps}} = \mvC{X_{\veps}}[A, \phi]$ take the form
\begin{align} 
	\mvC{X_{\veps}}_{L}
=&	\frac{1}{2} \bb(\frac{v_{\veps}}{u_{\veps}}\bb)^{\frac{1}{2}} (\abs{r^{-1}\covD_{L}(r \phi)}^{2} + \abs{\alp}^{2})
	+ \frac{1}{2} \bb(\frac{u_{\veps}}{v_{\veps}}\bb)^{\frac{1}{2}} \bb( \abs{\scovD \phi}^{2} + \frac{\abs{\phi}^{2}}{r^{2}} + \abs{\varrho}^{2} + \abs{\sgm}^{2} \bb),  	\label{eq:monotonicity:mvC-L}\\
	\mvC{X_{\veps}}_{\uL}
=&	\frac{1}{2} \bb(\frac{u_{\veps}}{v_{\veps}}\bb)^{\frac{1}{2}} (\abs{r^{-1}\covD_{\uL}(r \phi)}^{2} + \abs{\ualp}^{2})
	+ \frac{1}{2} \bb(\frac{v_{\veps}}{u_{\veps}}\bb)^{\frac{1}{2}} \bb( \abs{\scovD \phi}^{2} + \frac{\abs{\phi}^{2}}{r^{2}} + \abs{\varrho}^{2} + \abs{\sgm}^{2} \bb),		\label{eq:monotonicity:mvC-uL}
\end{align}
where $v_{\veps} := (t+\veps) + r$ and $u_{\veps} := (t+\veps) - r$.
\end{lemma}

Here we give a quick proof of \eqref{eq:monotonicity:exact} for a smooth solution $(A, \phi)$ on $\bbR^{1+4}$. By $\calF_{\rd C_{[t_{0}, t_{1}]}} = 0$, $\G_{\rd S_{t_{1}}} = 0$ and Lemma~\ref{lem:flux-decay-prelim}, note that $F = 0$ and $\phi = 0$ on the boundary $\rd C_{[t_{0}, t_{1}]}$. Integrate \eqref{eq:monotonicity:local} with $\veps = 0$ over $C_{[t_{0}, t_{1}]}$ and apply the divergence theorem. The boundary term on $\rd C_{[t_{0}, t_{1}]}$ vanishes thanks to $F, \phi = 0$, and thus \eqref{eq:monotonicity:exact} follows. 

In the preceding proof, however, note from \eqref{eq:monotonicity:mvC-0-L} that there is a weight $(\frac{v}{u})^{1/2}$ in the boundary term, which would blow up if $\covD_{L}(r\phi)$ and $\alp_{A}$ were not exactly zero on $\rd C_{[t_{0}, t_{1}]}$. We now turn to the proof of Proposition~\ref{prop:monotonicity}, whose goal is exactly to deal with this issue.
\begin{proof} [Proof of Proposition~\ref{prop:monotonicity}]
As the hypothesis \eqref{eq:monotonicity:hyp} and the conclusion \eqref{eq:monotonicity} only involve quantities which are continuous with respect to the $C_{t} \calH^{1}(I \times \bbR^{4})$ topology, it suffices to consider the case when $(A, \phi)$ is smooth. Integrating \eqref{eq:monotonicity:local} with $\veps > 0$ over $C_{[\veps, 1]}$ and integrating by parts, we obtain
\begin{equation} \label{eq:monotonicity:pf:1}
\begin{aligned}
& \hskip-2em
 \int_{S_{1}} \mvC{X_{\veps}}_{T}[A, \phi] \, \ud x 
+ \iint_{C_{[\veps, 1]}} \frac{1}{\rho_{\veps}} \abs{\iota_{X_{\veps}} F}^{2} + \frac{1}{\rho_{\veps}} \abs{(\covD_{X_{\veps}}  + \frac{1}{\rho_{\veps}}) \phi}^{2} \, \ud t \ud x \\
= & \int_{S_{\veps}} \mvC{X_{\veps}}_{T}[A, \phi] \, \ud x 
		+ \frac{1}{2} \int_{\rd C_{[\veps, 1]}} \mvC{X_{\veps}}_{L} [A, \phi] r^{3} \, \ud v \ud \sgm_{\bbS^{3}}.
\end{aligned}
\end{equation}
We claim that the right-hand side is bounded from above by $\aleq E$. We begin with the first term. On $S_{\veps}$, we have the pointwise bound
\begin{equation*}
	\mvC{X_{\veps}}_{T}[A, \phi]
	\aleq \mvC{T}_{T}[A, \phi] + \frac{1}{r^{2}} \abs{\phi}^{2},
\end{equation*}
since $(u_{\veps}, v_{\veps}) \aeq 1$ and $(v_{\veps}, u_{\veps}) \aeq 1$ on $S_{\veps}$. By \eqref{eq:monotonicity:hyp}, Lemma~\ref{lem:flux-decay-prelim} and \eqref{eq:local-hardy} applied to $\phi$ on $S_{\veps}$ with $r_{1} = 0$, $r_{2} = \veps$, it follows that the first term on the right-hand side of \eqref{eq:monotonicity:pf:1} is bounded by $\aleq E$.

We now consider the last term in \eqref{eq:monotonicity:pf:1}. On $\rd C_{[\veps, 1]}$, we have
\begin{equation*}
\mvC{X_{\veps}}_{L}[A, \phi]
\aleq	\veps^{-\frac{1}{2}} \bb( \abs{\covD_{L}\phi}^{2} + \frac{1}{r^{2}}\abs{\phi}^{2} + \abs{\alp}^{2} \bb) + \vC{T}_{L}[A, \phi],
\end{equation*}
Then by \eqref{eq:monotonicity:hyp}, Lemma~\ref{lem:flux-decay-prelim} and the fact that $t = r$ on $\rd C$, the last term in \eqref{eq:monotonicity:pf:1} is bounded by $\aleq E$ as desired.
\end{proof}

We end this section with a proof of Proposition~\ref{prop:monotonicity:t-like}.
\begin{proof} [Proof of Proposition~\ref{prop:monotonicity:t-like}]
As before, by approximation, it suffices to consider the case when $(A, \phi)$ is smooth. Let $\dlt \in [\dlt_{0}, \dlt_{1}]$ be a number to be determined below. Integrating \eqref{eq:monotonicity:local} with $\veps = 0$ over $C^{\dlt}_{[t_{0}, 1]}$ and using the divergence theorem, we see that \eqref{eq:monotonicity:t-like} would follow if there exists $\dlt \in [\dlt_{0}, \dlt_{1}]$ such that
\begin{equation} \label{eq:monotonicity:t-like-a}
	\int_{\rd C^{\dlt}_{[t_{0}, 1]}} \mvC{X_{0}}_{L}[A, \phi] \, r^{3} \, \ud v \ud \sgm_{\bbS^{3}} \aleq \bb( (\dlt_{1}/t_{0})^{\frac{1}{2}} + \abs{\log(\dlt_{1}/\dlt_{0})}^{-1} \bb) E.
\end{equation}
The contribution of the term with the weight $(u_{0} / v_{0})^{1/2}$ in \eqref{eq:monotonicity:mvC-L} is easy to treat; indeed, using localized Hardy's inequality and local conservation of energy, we have
\begin{align*}
& \hskip-2em
	\int_{\rd C^{\dlt}_{[t_{0}, 1]}} \frac{1}{2} \bb(\frac{u}{v}\bb)^{\frac{1}{2}} \bb( \abs{\scovD \phi}^{2} + \frac{\abs{\phi}^{2}}{r^{2}} + \abs{\varrho}^{2} + \abs{\sgm}^{2} \bb) \, r^{3} \, \ud v \ud \sgm_{\bbS^{3}} \\
	\aleq & \bb(\frac{\dlt_{1}}{t_{0}} \bb)^{1/2}\bb( \int_{\rd C^{\dlt}_{[t_{0}, 1]}} \vC{T}_{L}[A, \phi] \, r^{3} \, \ud v \ud \sgm_{\bbS^{3}} + \calE_{S_{1} \setminus S_{1}^{\dlt}}[A, \phi] + \G_{S_{1}}[\phi] \bb) \aleq \bb( \frac{\dlt_{1}}{t_{0}} \bb)^{1/2} E.
\end{align*}

It remains to treat the term with the weight $(v_{0} / u_{0})^{1/2}$ in \eqref{eq:monotonicity:mvC-L}. Note that
\begin{align*}
	r^{-1} \covD_{L}(r \phi) 
	= & (\covD_{L} + \frac{1}{r}) \phi
	= 2 \bb( \frac{u_{\veps}}{v_{\veps}} \bb)^{\frac{1}{2}} (\covD_{X_{\veps}} + \frac{1}{\rho_{\veps}}) \phi
		 - \bb( \frac{u_{\veps}}{v_{\veps}} \bb) \covD_{\uL} \phi
		 + \bb( \frac{u_{\veps}}{v_{\veps}} \bb) \frac{1}{r} \phi,\\
	\alp_{\mathfrak{a}} 
	= & F(L, e_{\mathfrak{a}}) 
	= 2 \bb(\frac{u_{\veps}}{v_{\veps}} \bb)^{\frac{1}{2}} F(X_{\veps}, e_{\mathfrak{a}}) - \bb( \frac{u_{\veps}}{v_{\veps}} \bb) F(\uL, e_{\mathfrak{a}}).
\end{align*}
Note that $u \leq u_{\veps}$ and $v \leq v_{\veps}$. Furthermore $u_{\veps} \leq 2 u$ on $\rd C^{\dlt}_{[t_{0}, 1]}$ since $2 \veps \leq \dlt_{0}$. Hence,
\begin{align}
& \hskip-2em
	\int_{\rd C^{\dlt}_{[t_{0}, 1]}} \frac{1}{2} \bb(\frac{v}{u}\bb)^{\frac{1}{2}} ( \abs{r^{-1} \covD_{L} (r \phi)}^{2} + \abs{\alp}^{2} ) \, r^{3} \, \ud v \ud \sgm_{\bbS^{3}}  \label{eq:monotonicity:t-like:pf:1} \\
	\aleq & \int_{\rd C^{\dlt}_{[t_{0}, 1]}} \frac{u}{\rho_{\veps}} \bb( \abs{(\covD_{X_{\veps}} + \frac{1}{\rho_{\veps}}) \phi}^{2} + \abs{\iota_{X_{\veps}} F}^{2} \bb) 
				+  \frac{u^{\frac{3}{2}}}{v^{\frac{3}{2}}}  \bb( \abs{\covD_{\uL} \phi}^{2} + \frac{1}{r^{2}} \abs{\phi}^{2} + \abs{\ualp}^{2} \bb) 
			\, r^{3} \, \ud v \ud \sgm_{\bbS^{3}}. \notag
\end{align}
We claim that the integral of the right-hand side over $\dlt_{0} \leq u \leq \dlt_{1}$ with respect to $u^{-1} \ud u$ is bounded by $E$. Then by the pigeonhole principle, there would exist $\dlt \in [\dlt_{0}, \dlt_{1}]$ such that the left-hand side of \eqref{eq:monotonicity:t-like:pf:1} is bounded by $\aleq \abs{\log(\dlt_{1} / \dlt_{0})}^{-1} E$, as desired.

For the contribution of the first term, the claim follows directly from Proposition~\ref{prop:monotonicity}. For the second term, we have
\begin{equation*}
\iint_{C^{\dlt_{0}}_{[t_{0}, 1]} \setminus C^{\dlt_{1}}_{[t_{0}, 1]}} \frac{u^{\frac{1}{2}}}{v^{\frac{3}{2}}}  \bb( \abs{\covD_{\uL} \phi}^{2} + \frac{1}{r^{2}} \abs{\phi}^{2} + \abs{\ualp}^{2} \bb) \, \ud t \ud x
\aleq \iint_{C^{\dlt_{0}}_{[t_{0}, 1]} \setminus C^{\dlt_{1}}_{[t_{0}, 1]}} \frac{\dlt_{1}^{\frac{1}{2}}}{t^{\frac{3}{2}}} \vC{T}_{T}[A, \phi] \, \ud t \ud x
\aleq \bb(\frac{\dlt_{1}}{t_{0}} \bb)^{\frac{1}{2}} E,
\end{equation*}
which is sufficient to prove the claim. \qedhere
\end{proof}

\section{Local strong compactness and weak solutions to
  \eqref{eq:MKG}} \label{sec:cpt} The first goal of this section is to
establish the following local strong compactness result for
asymptotically stationary (see \eqref{eq:cpt:vanishingX} below)
sequences of solutions to \eqref{eq:MKG} with small energy.
\begin{proposition} \label{prop:cpt} There exists a universal constant
  $\eps_{0} > 0$ such that the following holds.  Let $B = B_{1}(x_{0}) \subseteq
  \bbR^{4}$ be an open ball of unit radius centered at $x_{0}$, and
  let $(A^{(n)}, \phi^{(n)})$ be a sequence of admissible $C_{t}
  \calH^{1}$ solutions to \eqref{eq:MKG} in $(-2, 2) \times 8B$ such that
  \begin{equation} \label{eq:cpt:hyp} 
\calE_{\set{0} \times 8B}[A^{(n)}, \phi^{(n)}] +   \nrm{ \phi^{(n)}(0, x)}_{L^{2}_{x}(8 B)}^{2}
 \leq \eps_{0}^{2}.
  \end{equation}
  Suppose furthermore that $(A^{(n)}, \phi^{(n)})$ is
  \emph{asymptotically stationary} in the sense that
  \begin{equation} \label{eq:cpt:vanishingX} \iint_{(-2, 2) \times 2
      B} \abs{\iota_{X} F^{(n)}}^{2} + \abs{(\covD^{(n)}_{X} + b)
      \phi^{(n)}}^{2} \, \ud t \ud x \to 0 \quad \hbox{ as } n \to
    \infty,
  \end{equation}
  where $X$ is a smooth time-like vector field and $b$ is a smooth
  real-valued function. Then there exists a pair $(A, \phi)$ in
  $L^{2}_{t,x}((- 1, 1) \times B)$ such that the following statements
  hold:
  \begin{enumerate}
  \item There exists a sequence of gauge transforms $\chi^{(n)} \in
    C_{t} \calG^{2}((-1, 1) \times B)$ such that, after passing to a
    subsequence, we have
    \begin{align}
      (A_{\mu}^{(n)} - \rd_{\mu} \chi^{(n)}, e^{i \chi^{(n)}}
      \phi^{(n)} )
      \to & (A_{\mu}, \phi) \quad \hbox{ strongly in } L^{2}_{t,x}((-1, 1) \times B),	 \label{eq:cpt:converge:non-cov} \\
      (F_{\mu \nu}^{(n)}, e^{i \chi^{(n)}} \covD_{\mu}^{(n)}
      \phi^{(n)}) \to & (F_{\mu \nu}, \covD_{\mu} \phi) \quad \hbox{
        strongly in } L^{2}_{t,x}((-1, 1) \times
      B), \label{eq:cpt:converge:cov}
    \end{align}
    where $F_{\mu \nu} = \rd_{\mu} A_{\nu} - \rd_{\nu} A_{\mu}$ and
    $\covD_{\mu} \phi = \rd_{\mu} \phi + i A_{\mu} \phi$ are defined
    in the sense of distributions.
  \item The limiting pair $(A, \phi)$ is a \emph{weak solution} to
    \eqref{eq:MKG} on $(-1, 1) \times B$, in the sense of
    Definition~\ref{def:weakSol} below. The connection 1-form $A$
    obeys, in the sense of distributions, the Coulomb gauge condition
    \begin{equation} \label{eq:cpt:coulomb} \rd^{\ell} A_{\ell} = 0
      \quad \hbox{ on } (-1, 1) \times B.
    \end{equation}
  \item The pair $(A ,\phi)$ possesses the following additional
    regularity:
    \begin{equation} \label{eq:cpt:reg} A \in H^{1}_{t,x}((-1, 1)
      \times B), \quad F_{\mu \nu}, \in H^{\frac{1}{2}}_{t,x}((-1, 1)
      \times B), \quad \phi \in H^{\frac{3}{2}}_{t,x}((- 1, 1) \times
      B).
    \end{equation}
  \item Moreover, the pair $(A, \phi)$ is \emph{stationary with
      respect to $X$}, in the sense that
    \begin{equation} \label{eq:cpt:stationary} \iota_{X} F = 0, \quad
      (\covD_{X} + b) \phi = 0 \quad \hbox{ on } (-1, 1) \times B.
    \end{equation}
  \end{enumerate}
\end{proposition}

As a result of taking limits, the notion of \emph{weak solutions} to
\eqref{eq:MKG} arises naturally from Proposition~\ref{prop:cpt}.  For
our application in Section~\ref{sec:proof}, we also need to formulate
the notion of locally defined weak solutions $(A_{[\alp]},
\phi_{[\alp]})$ that can be pieced together to form a global pair
(weak compatible pairs). Developing a theory of these objects is
another goal of this section.

\begin{remark} 
  We remark that weak solutions and their gauge structure play only an
  auxiliary role in our work. Indeed, the stationarity equation
  \eqref{eq:cpt:stationary}, combined with \eqref{eq:MKG} and the
  additional regularity \eqref{eq:cpt:reg} of $(A, \phi)$, allow us to
  infer \emph{smoothness} of $(A, \phi)$ via elliptic regularity. This
  issue is considered in Section~\ref{sec:stationary-self-sim}, where
  we study stationary and self-similar solutions to \eqref{eq:MKG}.
\end{remark}

\begin{remark} 
It is in fact possible to obtain stronger convergence than \eqref{eq:cpt:converge:non-cov}
namely $A_{\mu}^{(n)} - \rd_{\mu} \chi^{(n)} \to A_{\mu}$ and $e^{i \chi^{(n)}} \phi^{(n)} \to \phi$
in $H^{1}_{t,x}((-1, 1) \times B)$. Moreover, the limit $A_{\mu}$ obeys the additional 
regularity $H^{3/2-\veps}_{t,x}((-1, 1) \times B)$ for any $\veps > 0$.  As these facts are not
necessary for the proof of our main theorem, we omit their proofs to avoid lengthening the paper.
\end{remark}

The rest of this section is structured as follows. We first give a
proof of Proposition~\ref{prop:cpt} in Section~\ref{subsec:cpt-proof},
except the statement that the limit $(A, \phi)$ is a \emph{weak}
solution to \eqref{eq:MKG}. In Section~\ref{subsec:weak-sol}, we
formulate a notion of weak solutions to \eqref{eq:MKG} that will be
used in our proof. Finally, in Section~\ref{subsec:cp}, we introduce
and discuss the notions of smooth and weak compatible pairs, which are
local descriptions of smooth and weak solutions to \eqref{eq:MKG},
respectively.

\subsection{Proof of
  Proposition~\ref{prop:cpt}} \label{subsec:cpt-proof} Here we prove
Proposition~\ref{prop:cpt} modulo the assertion that the limit $(A,
\phi)$ is a weak solution to \eqref{eq:MKG}, which would be clear once
we define the notion of a weak solution in
Definition~\ref{def:weakSol} below.
\begin{proof}
  The basic idea behind proof is as in
  \cite[Proposition~5.1]{MR2657818}: Small energy \eqref{eq:cpt:hyp}
  implies local uniform $S^{1}$ bound on $(-2, 2) \times 2B$, which can be
  combined with asymptotic stationarity \eqref{eq:cpt:vanishingX} via
  a microlocal decomposition to conclude strong convergence in $(-1,
  1) \times B$. In implementing this strategy, we need to take into
  account the presence of the constraint equation and the system
  nature of \eqref{eq:MKG} (especially the Maxwell part). Our proof
  proceeds in several steps.


  \pfstep{Step 1} In this step, we use the excision and gluing
  technique to produce gauge equivalent Coulomb solutions on the
  smaller region $(-2, 2) \times 2B$, which enjoy a uniform $S^{1}$ bound.

  Let $(a_{j}^{(n)}, e_{j}^{(n)}, f^{(n)}, g^{(n)}) = (A_{j}^{(n)},
  F_{0j}, \phi^{(n)}, \covD_{t}^{(n)} \phi^{(n)}) \rst_{\set{t=0}}$ be
  the data for $(A, \phi)$ on $\set{t = 0}$. Applying
  Theorem~\ref{thm:gluing} to $8 B \setminus 4\overline{B}$, we
  obtain an initial data set $(\widetilde{a}^{(n)},
  \widetilde{e}^{(n)}, \widetilde{f}^{(n)}, \widetilde{g}^{(n)}) \in
  \calH^{1}(\bbR^{4})$ such that $(\widetilde{a}^{(n)},
  \widetilde{e}^{(n)}, \widetilde{f}^{(n)}, \widetilde{g}^{(n)}) =
  (a^{(n)}, e^{(n)}, f^{(n)}, g^{(n)})$ on $4B$ and
  \begin{equation*}
    \calE[\widetilde{a}^{(n)}, \widetilde{e}^{(n)}, \widetilde{f}^{(n)}, \widetilde{g}^{(n)}]
    \lesssim \eps_{0}^{2}.
  \end{equation*}
  by \eqref{eq:gluing:energy} and \eqref{eq:cpt:hyp}. Choosing
  $\eps_{0}$ appropriately, we may
  ensure that the left-hand side is smaller than $\thE^{2}$, which is
  the threshold for Theorem~\ref{thm:smallEnergy}.

  To pass to the global Coulomb gauge, consider the gauge
  transformation $\underline{\chi}^{(n)} \in \calG^{2}(\bbR^{4})$
  defined by $\underline{\chi}^{(n)} = \lap^{-1} \rd^{\ell}
  \widetilde{a}_{\ell}^{(n)}$ and let
  \begin{equation*}
    (\check{a}^{(n)}, \check{e}^{(n)}, \check{f}^{(n)}, \check{g}^{(n)})
    := (\widetilde{a}^{(n)} - \ud \underline{\chi}^{(n)}, \widetilde{e}^{(n)}, e^{i \underline{\chi}^{(n)}} \widetilde{f}^{(n)}, e^{i \underline{\chi}^{(n)}} \widetilde{g}^{(n)}).
  \end{equation*}
  This initial data set agrees with $(a^{(n)}, e^{(n)}, f^{(n)},
  g^{(n)})$ on $4 B$ up to a gauge transformation, i.e.,
  \begin{equation}
    (\check{a}^{(n)}, \check{e}^{(n)}, \check{f}^{(n)}, \check{g}^{(n)}) 
    = (a^{(n)} - \ud \underline{\chi}^{(n)}, e^{(n)}, e^{i \underline{\chi}^{(n)}} f^{(n)}, e^{i \underline{\chi}^{(n)}} g^{(n)}) \quad \hbox{ on } 4B,
  \end{equation}
  and furthermore obeys the small energy condition
  \begin{equation}
    \calE[\check{a}^{(n)}, \check{e}^{(n)}, \check{f}^{(n)}, \check{g}^{(n)}] < \thE^{2}.
  \end{equation}
  By small energy global well-posedness
  (Theorem~\ref{thm:smallEnergy}), it follows that there exists a
  unique $C_{t} \calH^{1}$ admissible solution $(\check{A}^{(n)},
  \check{\phi}^{(n)})$ on $\bbR^{1+4}$ with initial data
  $(\check{a}^{(n)}, \check{e}^{(n)}, \check{f}^{(n)},
  \check{g}^{(n)})$, which obeys
  \begin{equation} \label{eq:cpt:cA-SY}
    \nrm{\check{A}^{(n)}_{0}}_{Y^{1}(\bbR^{1+4})} +
    \nrm{\check{A}^{(n)}_{x}}_{S^{1}(\bbR^{1+4})} +
    \nrm{\check{\phi}^{(n)}}_{S^{1}(\bbR^{1+4})} \aleq \thE.
  \end{equation}
  Moreover, by geometric uniqueness (Proposition~\ref{prop:geomUni})
  and the simple fact that
  \begin{equation*}
    (-2, 2) \times 2 B \subseteq \calD^{+}(\set{0} \times 4B) \cup \calD^{-}(\set{0} \times 4B), 
  \end{equation*}
  there exists $\chi^{(n)} \in C_{t} \calG^{2}((-2, 2) \times 2B)$
  such that
  \begin{equation}
    (\check{A}^{(n)}, \check{\phi}^{(n)}) = (A^{(n)} - \ud \chi^{(n)}, e^{i \chi^{(n)}} \phi^{(n)}) \quad \hbox{ on } (-2, 2) \times 2B.
  \end{equation}

  Let $\eta_{0}, \ldots , \eta_{3} \in C^{\infty}_{0}(\bbR^{1+4})$ be
  such that
  \begin{equation*}
    \eta_{j} = 1 \hbox{ on } (-1, 1) \times B, \quad 
    \supp \, \eta_{j} \subseteq (-2, 2) \times 2B, \quad
    \eta_{j} \eta_{j+1} = \eta_{j}.
  \end{equation*}
  for $j=0, 1, 2, 3$ (except for the last property, for which $j = 0,
  1, 2$), which will be fixed for the rest of the proof. We will also
  often write $\eta = \eta_{0}$ and $\widetilde{\eta} = \eta_{3}$. By
  \eqref{eq:cpt:cA-SY} and Remark~\ref{rem:SY-norms}, the solution
  $(\check{A}^{(n)}, \check{\phi}^{(n)})$ satisfies
  \begin{align}
    \nrm{\rd_{t,x} (\eta_{j} \check{A}^{(n)})}_{L^{\infty}_{t}
      L^{2}_{x}} + \nrm{\rd_{t,x} (\eta_{j}
      \check{\phi}^{(n)})}_{L^{\infty}_{t} L^{2}_{x}}
    & \aleq_{\eta_{j}} \eps_{0}, \label{eq:cpt:cA-E} \\
    \nrm{\rd_{t,x} (\eta_{j} \check{A}^{(n)}_{0})}_{L^{2}_{t}
      \dot{H}^{\frac{1}{2}}_{x}} + \nrm{\Box (\eta_{j}
      \check{A}^{(n)}_{x})}_{L^{2}_{t} \dot{H}^{-\frac{1}{2}}_{x}} +
    \nrm{\Box (\eta_{j} \check{\phi}^{(n)})}_{L^{2}_{t}
      \dot{H}^{-\frac{1}{2}}_{x}} & \aleq_{\eta_{j}}
    \eps_{0}. \label{eq:cpt:cA-X}
  \end{align}
  for any $j=0,1,2,3$. In particular, in view of \eqref{eq:cpt:cA-E} and
  H\"older's inequality, the sequence $(\widetilde{\eta}
  \check{A}^{(n)}, \widetilde{\eta} \check{\phi}^{(n)})$ is uniformly
  bounded in $H^{1}_{t,x}$. By the Rellich-Kondrachov theorem, there exists a
  subsequence, which we still denote by $(\widetilde{\eta}
  \check{A}^{(n)}, \widetilde{\eta} \check{\phi}^{(n)})$, and a pair
  $(A, \phi) \in H^{1}_{t,x}$ such that
  \begin{equation} \label{eq:cpt:converge:weak} (\widetilde{\eta}
    \check{A}^{(n)}, \widetilde{\eta} \check{\phi}^{(n)}) \weakto (A,
    \phi) \quad \hbox{ in } H^{1}_{t,x}, \quad (\widetilde{\eta}
    \check{A}^{(n)}, \widetilde{\eta} \check{\phi}^{(n)}) \to (A,
    \phi)\quad \hbox{ in } L^{2}_{t,x},
  \end{equation}
  as $n \to \infty$, where the notation $\weakto$ refers to weak
  convergence.

  \pfstep{Step 2} In this preparatory step, we make a microlocal
  decomposition of $\eta$ that will allows us to combine
  \eqref{eq:cpt:vanishingX} with the bound \eqref{eq:cpt:cA-X} on the
  sequence; see \eqref{eq:decomp4eta}.

  We use the classical pseudo-differential calculus. Let $q_{0}(\tau,
  \xi) \in S^{0}$ be a smooth cutoff such that $q_{0}=1$ to the region
  $\set{(\tau, \xi) : \abs{\tau} \leq (1-\dlt) \abs{\xi}}$ in Fourier
  space and $\supp \, q_{0} \subseteq \set{(\tau, \xi) : \abs{\tau}
    \leq (1-\dlt/2) \abs{\xi}}$, where $\dlt > 0$ is to be chosen
  shortly.  On the support of $q_{0}$, the norm on the left-hand side
  of \eqref{eq:cpt:cA-X} is effective. On the other hand, since $X =
  X^{\mu} \rd_{\mu}$ is a time-like vector field, we have
  $\abs{X^{0}(t,x)}^{2} > \sum_{j=1}^{4} \abs{X^{j}(t,x)}^{2}$
  everywhere. As $\supp \, \eta$ is compact, we may choose $\dlt > 0$
  sufficiently small so that
  \begin{equation*}
    \abs{X^{0}(t,x)} \geq (1-\dlt)^{2} \bb( \sum_{j=1}^{4} \abs{X^{j}(t,x)}^{2} \bb)^{\frac{1}{2}} \quad \hbox{ for } (t,x) \in \supp \, \eta.
  \end{equation*}
  With such a choice of $\dlt > 0$, the symbol $X^{0}(t,x) \tau +
  X^{\ell}(t,x) \xi_{\ell} \in S^{1}$ is \emph{elliptic} on the phase
  space support of $\eta (t,x) (1-q_{0})(\tau, \xi)$, in the sense
  that
  \begin{equation*}
    \abs{X^{0}(t,x) \tau + X^{\ell}(t,x) \xi_{\ell}} 
    \geq \abs{X^{0}(t,x) \tau} - \abs{X^{\ell}(t,x) \xi_{\ell}} 
    \geq c_{\dlt, \eta, X^{0}} (\abs{\tau} + \abs{\xi})
  \end{equation*}
  for $(t,x) \in \supp \, \eta$ and $(\tau, \xi) \in \supp \,
  (1-q_{0})$, where we may take
  \begin{equation*}
    c_{\dlt, \eta, X^{0}} = \frac{\dlt (1-\dlt)}{2} \inf_{\supp \, \eta} \abs{X^{0}} > 0.
  \end{equation*}
  Using the standard construction of a pseudo-differential elliptic
  parametrix, we may write
  \begin{equation*}
    \eta (1-q_{0})(D_{t,x}) = q_{-1} (t,x, D_{t,x}) \, \eta X^{\mu} \rd_{\mu} + \widetilde{r}_{-1}(t,x,D_{t,x})
  \end{equation*}
  where $q_{-1}, \widetilde{r}_{-1} \in S^{-1}$. Rearranging the
  terms, commuting $\eta(t,x)$ with $q_{0}$ and applying
  multiplication by $\eta_{1}$ on the right, we arrive at the
  decomposition
  \begin{equation} \label{eq:decomp4eta} \eta = q_{-1}(t,x, D_{t,x})
    \eta X^{\mu} \rd_{\mu} + q_{0} \eta + r_{-1}(t,x,D_{t,x})
    \eta_{1},
  \end{equation}
  where $r_{-1} \in S^{-1}$ is the sum of $\widetilde{r}_{-1}$ and the
  commutator between $\eta$ and $q_{0}$.

  \pfstep{Step 3} Here we show the strong convergence $\eta
  F^{(n)}_{\mu \nu} \to \eta F_{\mu \nu}$ in $L^{2}_{t,x}$, where we
  remind the reader that $F_{\mu \nu} = \hat{F}_{\mu \nu}$ by gauge
  invariance of the curvature 2-form. By \eqref{eq:decomp4eta}, we may
  write
  \begin{equation*}
    \eta F^{(n)}_{\mu \nu} = q_{-1}(t,x, D_{t,x}) \eta X^{\lmb} \rd_{\lmb} F^{(n)}_{\mu \nu} + q_{0}(D_{t,x}) \eta F^{(n)}_{\mu \nu} + r_{-1}(t,x,D_{t,x}) \eta_{1} F^{(n)}_{\mu \nu}.
  \end{equation*}
  Using $\ud F^{(n)} = 0$, we rewrite $\eta X^{\lmb} \rd_{\lmb}
  F^{(n)}_{\mu \nu}$ as
  \begin{equation*}
    \eta X^{\lmb} \rd_{\lmb} F_{\mu \nu}^{(n)} = \rd_{\mu} (\eta X^{\lmb} F^{(n)}_{\lmb \nu}) - \rd_{\nu} (\eta X^{\lmb} F^{(n)}_{\lmb \mu})
    - \rd_{\mu} (\eta X^{\lmb}) F^{(n)}_{\lmb \nu} + \rd_{\nu} (\eta X^{\lmb}) F^{(n)}_{\lmb \mu},
  \end{equation*}
  and hence we arrive at
  \begin{equation}
    \begin{aligned}
      \eta F^{(n)}_{\mu \nu} = & q_{-1} (t,x, D_{t,x}) \big[ \rd_{\mu}
      (\eta (\iota_{X} F^{(n)})_{\nu}) - \rd_{\nu} (\eta (\iota_{X}
      F^{(n)})_{\mu}) \big] + R_{\M}[F^{(n)}]_{\mu \nu}
    \end{aligned}
  \end{equation}
  where
  \begin{equation*}
\begin{split}
    R_{\M}[F^{(n)}]_{\mu \nu}
    = & \ q_{0} (D_{t,x}) \eta F^{(n)}_{\mu \nu} 
    - q_{-1}(t,x, D_{t,x}) \big[ \rd_{\mu} (\eta X^{\lmb}) F^{(n)}_{\lmb \nu} - \rd_{\nu}(\eta X^{\lmb}) F^{(n)}_{\lmb \mu} \big]
   \\ & \  + r_{-1}(t, x, D_{t,x}) \eta_{1} F^{(n)}_{\mu \nu} .
\end{split}
  \end{equation*}
  By \eqref{eq:cpt:vanishingX}, it follows that
  \begin{equation*}
    \nrm{q_{-1} (t,x, D_{t,x}) \big[ \rd_{\mu} (\eta (\iota_{X} F^{(n)})_{\nu}) - \rd_{\nu} (\eta (\iota_{X} F^{(n)})_{\mu}) \big]}_{L^{2}_{t,x}} \to 0.
  \end{equation*}
  Moreover, we claim that $R_{\M}[F^{(n)}]_{\mu \nu}$ enjoys improved
  regularity, i.e.,
  \begin{equation} \label{eq:cpt:converge:F-claim}
    \nrm{R_{\M}[F^{(n)}]_{\mu \nu}}_{H^{\frac{1}{2}}_{t,x}} \aleq
    \eps_{0} \quad \hbox{ uniformly in } n.
  \end{equation}
  By the Rellich-Kondrachov theorem, after passing to a subsequence of
  $(\check{A}^{(n)}, \check{\phi}^{(n)})$, the sequence
  $\widetilde{\eta} R_{\M}[F^{(n)}]_{\mu \nu}$ is strongly convergent
  in $L^{2}_{t,x}$; moreover, we can also ensure that the limit
  belongs to $H^{\frac{1}{2}}_{t,x}$. Combining these facts, as well
  as the identity $\eta \widetilde{\eta} = \eta$, we see that $\eta
  F^{(n)}_{\mu \nu}$ is strongly convergent in $L^{2}_{t,x}$ to a
  limit that belongs to $H^{\frac{1}{2}}_{t,x}$. Since
  $\widetilde{\eta} \check{A}_{\mu} \to A_{\mu}$ in $L^{2}_{t,x}$, the
  limit is equal to $\eta F_{\mu \nu}$. Hence the statements regarding
  $F$ in \eqref{eq:cpt:converge:cov} and \eqref{eq:cpt:reg} follow.

  It remains to verify the claim \eqref{eq:cpt:converge:F-claim}; it
  is at this point we use the uniform bound
  \eqref{eq:cpt:cA-X}. Expanding $F^{(n)} = \ud \check{A}^{(n)}$, it
  follows from \eqref{eq:cpt:cA-E} that $\nrm{\eta_{2}
    F^{(n)}}_{L^{2}_{t,x}} \aleq \eps_{0}$. Then by
  \eqref{eq:cpt:cA-X} and the support property of the symbol $q_{0}$,
  we have
  \begin{equation*}
    \nrm{q_{0}(D_{t,x}) \eta F^{(n)}}_{H^{\frac{1}{2}}_{t,x}} 
    \aleq \nrm{\eta_{2} F^{(n)}}_{L^{2}_{t,x}} 
    +\nrm{\rd_{t,x} (\eta_{2} \check{A}_{0}^{(n)})}_{L^{2}_{t} \dot{H}^{\frac{1}{2}}_{x}} 
    + \nrm{\Box (\eta_{2} \check{A}_{x}^{(n)})}_{L^{2}_{t} \dot{H}^{-\frac{1}{2}}_{x}}
    \aleq \eps_{0},
  \end{equation*}
  and for the remainder, we have
  \begin{equation*}
    \nrm{R_{\M}[F^{(n)}]_{\mu \nu} - q_{0}(D_{t,x}) \eta F^{(n)}_{\mu \nu}}_{H^{1}_{t,x}} \aleq \nrm{\eta_{2} F^{(n)}}_{L^{2}_{t,x}} \aleq \eps_{0},
  \end{equation*}
  which proves the claim.

  \pfstep{Step 4} In this intermediate step, we use strong
  $L^{2}_{t,x}$ convergence of $F^{(n)}_{\mu \nu}$ to prove
  \begin{equation} \label{eq:cpt:converge:A-L2H1} \eta
    \check{A}^{(n)}_{\mu} \to \eta A_{\mu} \quad \hbox{ strongly in }
    L^{2}_{t} H^{1}_{x} .
  \end{equation}
  as $n \to \infty$, up to a subsequence. We also prove improved
  regularity for the limit $A_{\mu}$, i.e.,
  \begin{equation} \label{eq:cpt:reg:A} \rd_{x} (\eta A_{\mu}) \in
    H^{\frac{1}{2}}_{t,x}.
  \end{equation}

  To begin with, observe that $\lap \check{A}^{(n)}_{\mu} = \rd^{\ell}
  F^{(n)}_{\ell \mu}$ by the Coulomb gauge condition. Therefore, for
  each spatial component $\mu = k \in \set{1,2,3,4}$, we have
  \begin{equation} \label{eq:cpt:A-from-F} \eta \check{A}^{(n)}_{k} =
    \lap^{-1} \bb( \rd^{\ell} (\eta F^{(n)}_{\ell k}) + [\lap, \eta]
    \check{A}^{(n)}_{k} + [\eta, \rd^{\ell}] F^{(n)}_{\ell k} \bb).
  \end{equation}
  For any $j \in \set{1,2,3,4}$, note that $\rd_{j} \lap^{-1}
  \rd^{\ell} (\eta F^{(n)}_{\ell k})$ is strongly convergent in
  $L^{2}_{t,x}$, thanks to the previous step. Writing out $F^{(n)} =
  \ud \check{A}^{(n)}$ and using the strong $L^{2}_{t,x}$ convergence
  of $\widetilde{\eta} \check{A}^{(n)}_{k}$, it follows that the
  remainder $\rd_{j} \lap^{-1} ([\lap, \eta] \check{A}^{(n)}_{k} +
  [\eta, \rd^{\ell}] F^{(n)}_{\ell k} )$ is strongly convergent in
  $L^{2}_{t,x}$ as well. Hence \eqref{eq:cpt:converge:A-L2H1} holds
  for $\mu \in \set{1,2,3,4}$.

  In the case $\mu = 0$, note that \eqref{eq:cpt:cA-E} and
  \eqref{eq:cpt:cA-X} already imply
  \begin{equation} \label{eq:cpt:A0-n} \nrm{\rd_{x} (\widetilde{\eta}
      \check{A}^{(n)}_{0})}_{H^{\frac{1}{2}}_{t,x}} \aleq \eps_{0}
    \quad \hbox{ uniformly in } n.
  \end{equation}
  Therefore, after taking a suitable subsequence, the desired
  convergence \eqref{eq:cpt:converge:A-L2H1} (by the Rellich-Kondrachov theorem)
  as well as the improved regularity \eqref{eq:cpt:reg:A} follow.

  It only remains to prove the improved regularity
  \eqref{eq:cpt:reg:A} for $\mu = k \in \set{1,2,3,4}$.  First, by
  \eqref{eq:cpt:A-from-F} and the improved regularity $\eta F \in
  H^{\frac{1}{2}}_{t,x}$, $\widetilde{\eta} \check{A} \in
  H^{1}_{t,x}$, it follows that $\eta \check{A}^{(n)}_{k} \in
  L^{2}_{t} H^{\frac{3}{2}}_{x}$. Then using the identity
  \begin{equation*}
    \rd_{t} (\eta A_{k}) - \rd_{k} ( \eta A_{0})= \eta F_{0 k} + [\rd_{j}, \eta] A_{k} - [\rd_{k}, \eta] A_{0},
  \end{equation*}
  and the improved regularity $\rd_{x} (\eta A_{0}) \in
  H^{\frac{1}{2}}_{t,x}$, as well as $\eta F \in
  H^{\frac{1}{2}}_{t,x}$, $\widetilde{\eta} \check{A} \in
  H^{1}_{t,x}$, we have $\rd_{t} (\eta \check{A}_{k}^{(n)}) \in
  H^{\frac{1}{2}}_{t,x}$. It follows that $\eta \check{A}^{(n)}_{k}
  \in H^{\frac{3}{2}}_{t,x}$, which is better than what we need.

  \pfstep{Step 5} In this step, we show that $\eta \check{\covD}^{(n)}
  \check{\phi}^{(n)} \to \eta \covD \phi$ in $L^{2}_{t,x}$ and $\eta
  \phi \in H^{\frac{3}{2}}_{t,x}$. For the former, from the
  decomposition
  \begin{equation*}
    \eta \check{\covD}^{(n)}_{\mu} \check{\phi}^{(n)} = \eta \rd_{\mu} \check{\phi}^{(n)} + i \eta \check{A}^{(n)}_{\mu} \check{\phi}^{(n)},
  \end{equation*}
  the convergence $\eta \check{A}^{(n)}_{\mu} \to \eta A$ in
  $L^{2}_{t} H^{1}_{x}$ and \eqref{eq:cpt:cA-E}, we see that it
  suffices to prove
  \begin{equation} \label{eq:cpt:rd-phi} \eta \rd_{\mu}
    \check{\phi}^{(n)} \to \eta \rd_{\mu} \phi \quad \hbox{ in }
    L^{2}_{t,x}.
  \end{equation}

  By \eqref{eq:decomp4eta}, we have
  \begin{equation*}
    \eta \check{\phi}^{(n)} = q_{-1}(t,x, D_{t,x}) \eta X^{\mu} \rd_{\mu} \check{\phi}^{(n)} + q_{0} \eta \check{\phi}^{(n)} + r_{-1}(t,x,D_{t,x}) \eta_{1} \check{\phi}^{(n)}
  \end{equation*}
  To use \eqref{eq:cpt:vanishingX}, we rewrite $\eta X^{\mu} \rd_{\mu}
  \check{\phi}^{(n)}$ as
  \begin{equation*}
    \eta X^{\mu} \rd_{\mu} \check{\phi}^{(n)} = \eta (\check{\covD}^{(n)}_{X} +b)\check{\phi}^{(n)} - i X^{\nu} \check{A}^{(n)}_{\nu} \eta \check{\phi}^{(n)} - \eta b \check{\phi}^{(n)}.
  \end{equation*}
  where $\check{\covD}^{(n)} = \ud + i \check{A}^{(n)}$. Expanding
  $\eta \check{A}^{(n)} = \eta (\check{A}^{(n)} - A^{(n)}) + \eta
  A^{(n)}$, we arrive at
  \begin{equation} \label{eq:cpt:pf:decomp4phi}
    \begin{aligned}
      \eta \check{\phi}^{(n)} = & q_{-1}(t, x, D_{t,x}) \eta
      (\check{\covD}^{(n)}_{X} + b) \check{\phi}^{(n)}
      - i q_{-1} (t,x, D_{t,x}) X^{\nu} \eta (\check{A}^{(n)}_{\nu} - A_{\nu}) \check{\phi}^{(n)} \\
      & - i q_{-1} (t,x, D_{t,x}) X^{\nu} \eta A_{\nu}
      \check{\phi}^{(n)} + R_{\KG}[\check{\phi}^{(n)}]
    \end{aligned}
  \end{equation}
  where
  \begin{equation*}
    R_{\KG}[\check{\phi}^{(n)}] := q_{0} \eta \check{\phi}^{(n)}  + r_{-1}(t, x, D_{t,x}) \eta_{1} \check{\phi}^{(n)} - b q_{-1}(t, x, D_{t,x}) \eta \check{\phi}^{(n)}.
  \end{equation*}
  As in Step 2, for the first term we have
  \begin{equation*}
    \nrm{q_{-1}(t, x, D_{t,x}) \eta (\check{\covD}^{(n)}_{X} + b) \check{\phi}^{(n)}}_{H^{1}_{t,x}} \to 0
  \end{equation*}
  as $n \to \infty$, thanks to \eqref{eq:cpt:vanishingX}. For the
  second term, we have
  \begin{equation*}
    \nrm{q_{-1} (t,x, D_{t,x}) X^{\nu} \eta (\check{A}^{(n)}_{\nu} - A_{\nu}) \check{\phi}^{(n)} }_{H^{1}_{t,x}}
    \aleq \nrm{\eta (\check{A}^{(n)}_{\nu} - A_{\nu})}_{L^{2}_{t} L^{4}_{x}} \nrm{\check{\phi}^{(n)}}_{L^{\infty}_{t} L^{4}_{x}} \to 0 
  \end{equation*}
  as $n \to \infty$, by H\"older, Sobolev in $x$, $L^{2}_{t}
  H^{1}_{x}$ convergence of $\eta \check{A}^{(n)}_{\nu}$ to $\eta
  A_{\nu}$ and \eqref{eq:cpt:cA-E}.  On the other hand, for the third
  term, we have
  \begin{equation*}
    \nrm{q_{-1} (t,x, D_{t,x}) X^{\nu} \eta A_{\nu} \check{\phi}^{(n)}}_{H^{\frac{3}{2}}_{t,x}}
    \aleq \eps_{0} \nrm{\brk{D_{x}} \brk{D_{t,x}}^{\frac{1}{2}} (\eta A)}_{L^{2}_{t,x}} \quad \hbox{ uniformly in } n.
  \end{equation*}
  where we used Lemma~\ref{lem:paraproduct} below with $f = \eta
  A_{\nu}$ and $g = \check{\phi}^{(n)}$. We also used the obvious
  bound $\nrm{\eta A_{\nu} \check{\phi}^{(n)}}_{L^{2}_{t,x}} \aleq
  \eps_{0} \nrm{\brk{D_{x}} (\eta A)}_{L^{2}_{t,x}}$, which follows
  from H\"older, Sobolev in $x$ and \eqref{eq:cpt:cA-E}, to control
  the $L^{2}_{t,x}$ norm of the left-hand side.  Finally, for
  $R_{\KG}[\check{\phi}^{(n)}]$ we have, as in Step 3,
  \begin{equation*}
    \nrm{R_{\KG}[\check{\phi}^{(n)}]}_{H^{\frac{1}{2}}_{t,x}} \aleq \eps_{0} \quad \hbox{ uniformly in } n.
  \end{equation*}
  By the Rellich-Kondrachov theorem, there exists a subsequence (which we still
  denote by $\check{\phi}^{(n)}$) such that
  \begin{equation*}
    \widetilde{\eta} (- i q_{-1} (t,x, D_{t,x}) X^{\nu} \eta A_{\nu} \check{\phi}^{(n)} 
    + R_{\KG}[\check{\phi}^{(n)}])
  \end{equation*}
  is strongly convergent in $H^{1}_{t,x}$ to a limit that belongs to
  $H^{\frac{3}{2}}_{t,x}$. As a consequence of these facts, as well as
  the identity $\eta \widetilde{\eta} = \eta$, it follows that $\eta
  \check{\phi}^{(n)}$ is strongly convergent in $H^{1}_{t,x}$ to a
  limit in $H^{\frac{3}{2}}_{t,x}$. Finally, since $\widetilde{\eta}
  \check{\phi}^{(n)} \to \phi$ in $L^{2}_{t,x}$, the limit is equal to
  $\eta \phi$.  \qedhere
\end{proof}

\begin{lemma} \label{lem:paraproduct} For $f, g \in
  \calS(\bbR^{1+4})$, we have
  \begin{equation} \label{eq:paraproduct}
    \nrm{fg}_{\dot{H}^{\frac{1}{2}}_{t,x}} \aleq
    \nrm{\abs{D_{t,x}}^{\frac{1}{2}} f}_{L^{2}_{t} \dot{H}^{1}_{x}}
    \nrm{D_{t,x} g}_{L^{\infty}_{t} L^{2}_{x}} .
  \end{equation}
\end{lemma}
\begin{proof}
  We use the Littlewood-Paley projections $\set{S_{j}}$ in
  $\bbR^{1+4}$. For every $j \in \bbZ$, we decompose
  \begin{align*}
    S_{j} (fg) = S_{j}((S_{> j - 10} f) g) + S_{j}(S_{\leq j - 10} f
    S_{[j-5, j+5]} g)
  \end{align*}
  Using Sobolev and H\"older, we estimate each term on the right-hand
  side as follows:
  \begin{align*}
    \nrm{S_{j}((S_{> j - 10} f) g)}_{\dot{H}^{\frac{1}{2}}_{t,x}}
    \aleq & \sum_{j_{1} > j - 10} 2^{\frac{1}{2} j}  \nrm{S_{j_{1}} f}_{L^{2}_{t} L^{4}_{x}} \nrm{g}_{L^{\infty}_{t} L^{4}_{x}} \\
    \aleq &\nrm{D_{t,x} g}_{L^{\infty}_{t} L^{2}_{x}}\sum_{j_{1} > j - 10} 2^{\frac{1}{2} (j - j_{1})} \nrm{\abs{D_{t,x}}^{\frac{1}{2}} S_{j_{1}} f}_{L^{2}_{t} \dot{H}^{1}_{x}}, \\
    \nrm{S_{j}(S_{\leq j - 10} f S_{[j-5, j+5]}
      g)}_{\dot{H}^{\frac{1}{2}}_{t,x}}
    \aleq & \sum_{j_{1} \leq j-10} 2^{\frac{1}{2} j} \nrm{S_{j_{1}} f}_{L^{2}_{t} L^{\infty}_{x}} \nrm{S_{[j-5, j+5]} g}_{L^{\infty}_{t} L^{2}_{x}} \\
    \aleq & \nrm{D_{t,x} g}_{L^{\infty}_{t} L^{2}_{x}} \sum_{j_{1}
      \leq j-10} 2^{\frac{1}{2}(j_{1} - j)}
    \nrm{\abs{D_{t,x}}^{\frac{1}{2}} S_{j_{1}} f}_{L^{2}_{t}
      \dot{H}^{1}_{x}} .
  \end{align*}
  Thanks to the exponential gain $2^{- \frac{1}{2}\abs{j - j_{1}}}$,
  we have
  \begin{equation*}
    \sum_{j} \nrm{S_{j}(fg)}_{\dot{H}^{\frac{1}{2}}_{t,x}}^{2}
    \aleq \nrm{D_{t,x} g}_{L^{\infty}_{t} L^{2}_{x}}^{2} \sum_{j_{1}} \nrm{\abs{D_{t,x}}^{\frac{1}{2}} S_{j_{1}} f}_{L^{2}_{t} \dot{H}^{1}_{x}}^{2}.
  \end{equation*}
  The desired estimate is now a consequence of almost orthogonality of
  $\set{S_{j}}_{j \in \bbZ}$ in $L^{2}_{t,x}$. \qedhere
\end{proof}

\subsection{Weak solutions to \eqref{eq:MKG}} \label{subsec:weak-sol}
We first define a function space that is suitable for a weak
formulation of \eqref{eq:MKG}.
\begin{definition} \label{def:weakSolSp} Let $\calO \subseteq
  \bbR^{1+4}$ be an open set. We define $\Xw(\calO)$ to be the linear
  space of pairs $(A, \phi)$, where $A$ is a real-valued 1-form and
  $\phi$ is a $\bbC$-valued function on $\calO$, such that
  \begin{equation} \label{eq:Xw:def} A_{\mu}, \phi \in
    L^{2}_{t,x}(\calO), \ F_{\mu \nu}, \covD_{\mu} \phi \in
    L^{2}_{t,x}(\calO) \quad \hbox{ for all } \mu, \nu = 0, 1, \ldots,
    4,
  \end{equation} 
  where $F_{\mu \nu} = \rd_{\mu} A_{\nu} - \rd_{\nu} A_{\mu}$ and
  $\covD_{\mu} \phi = \rd_{\mu} \phi + i A_{\mu} \phi$ in the sense of
  distributions.
\end{definition}

We may now define a notion of weak solutions to \eqref{eq:MKG} as
follows.
\begin{definition}[Weak solutions to
  \eqref{eq:MKG}] \label{def:weakSol} Let $\calO \subseteq \bbR^{1+4}$
  be an open set, and let $(A, \phi) \in \Xw(\calO)$. We say that $(A,
  \phi)$ is a \emph{weak solution} to \eqref{eq:MKG} on $\calO$ if for
  every real-valued 1-form $\omg \in C^{\infty}_{0}(\calO)$ and
  complex-valued function $\varphi \in C^{\infty}_{0}(\calO)$, we have
  \begin{align}
    \iint_{\calO} F_{\nu \mu} \rd^{\mu} \omg^{\nu} + \Im(\phi \overline{\covD_{\nu} \phi}) \omg^{\nu} \, \ud t \ud x =& 0,  \label{eq:weakSol:M} \\
    \iint_{\calO} \Re (\covD_{\mu} \phi \overline{\rd^{\mu} \varphi})
    + \Im(A^{\mu} \covD_{\mu} \phi \overline{\varphi}) \, \ud t \ud x
    =& 0. \label{eq:weakSol:KG}
  \end{align}
\end{definition}

By an integration by parts argument, it may be readily verified that
admissible and classical solutions to \eqref{eq:MKG} are indeed weak
solutions. In the converse direction, if $(A, \phi)$ is a weak
solution to \eqref{eq:MKG} that is furthermore smooth, then $(A,
\phi)$ solves \eqref{eq:MKG} in the usual, classical sense.

Next, we discuss the gauge structure of weak solutions to
\eqref{eq:MKG}. We first define the space of gauge transformations
between pairs in $\Xw$.
\begin{definition} \label{def:Yw:def} Given an open set $\calO
  \subseteq \bbR^{1+4}$, let $\Yw(\calO)$ be the space of real-valued
  functions $\chi$ on $\calO$ such that $\chi \in H^{1}_{t,x}(\calO)$.
\end{definition}
Indeed, note that if $(A, \phi) \in \Xw$ and $\chi \in \Yw$, then the
gauge transform $(\widetilde{A}, \widetilde{\phi}) := (A - \ud \chi,
e^{i \chi})$ also belongs to $\Xw$. Moreover, if $(A, \phi)$ is a weak
solution to \eqref{eq:MKG} then so is $(\widetilde{A},
\widetilde{\phi})$, as the next lemma demonstrates.

\begin{lemma} \label{lem:gt4weakSol} Let $\calO \subseteq \bbR^{1+4}$
  be an open set, and let $(A, \phi) \in \Xw(\calO)$ be a weak
  solution to \eqref{eq:MKG}. Then for every $\chi \in \Yw(\calO)$,
  the gauge transform $(\widetilde{A}, \widetilde{\phi}) := (A - \ud
  \chi, e^{i \chi} \phi)$ also belongs to $\Xw(\calO)$ and is a weak
  solution to \eqref{eq:MKG}.
\end{lemma}
\begin{proof}
  We need to verify \eqref{eq:weakSol:M} and \eqref{eq:weakSol:KG} for
  $(\widetilde{A}, \widetilde{\phi})$. For \eqref{eq:weakSol:M} there
  is nothing to verify, as both $F$ and $\Im(\phi \overline{\covD
    \phi})$ are invariant under gauge transformation. For
  \eqref{eq:weakSol:KG}, we have
  \begin{align*}
    & \hskip-2em
    \iint_{\calO} \Re(\widetilde{\covD}_{\mu} \widetilde{\phi} \, \overline{\rd^{\mu} \varphi}) + \Im(\widetilde{A}^{\mu} \widetilde{\covD_{\mu}} \widetilde{\phi} \, \overline{\varphi}) \, \ud t \ud x \\
    = & \iint_{\calO} \Re(\covD_{\mu} \phi \overline{\rd^{\mu} (e^{-i
        \chi} \varphi)}) + \Im(A^{\mu} \covD_{\mu} \phi \,
    \overline{e^{-i \chi} \varphi}) \, \ud t \ud x.
  \end{align*}
  Observe that if $\chi \in C^{\infty}(\calO)$, then the last line
  would be equal to zero by \eqref{eq:weakSol:KG} for $(A,
  \phi)$. Considering a sequence $\chi^{(n)} \in C^{\infty}(\calO)$
  such that $\chi^{(n)} \to \chi$ in the $H^{1}_{t,x}(\calO)$ topology
  and also pointwise almost everywhere, it can be seen that the last
  line is indeed zero, by the dominated convergence theorem, Leibniz's
  rule and H\"older's inequality. \qedhere
%
\end{proof}

\subsection{Local description of solutions to
  \eqref{eq:MKG}} \label{subsec:cp} Here we discuss how to describe a
solution to \eqref{eq:MKG} by local data. More precisely, given an
open cover $\calQ = \set{Q_{\alp}}$ of an open set $\calO \subseteq
\bbR^{1+4}$, we would like to describe a solution to \eqref{eq:MKG} on
$\calO$ by local solutions on $Q_{\alp}$ satisfying certain
compatibility conditions, which ensure that the local solutions
combine to form a single solution on $\calO$. This idea is made
precise by the ensuing definition.
\begin{definition}[Smooth compatible pairs] \label{def:smth-cp} Let
  $\calO \subseteq \bbR^{1+4}$ be an open set and let $\calQ =
  \set{Q_{\alp}}$ be a locally finite open covering of $\calO$. For
  each index $\alp$, consider a pair $(A_{[\alp]}, \phi_{[\alp]}) \in
  C^{\infty}_{t,x}(Q_{\alp})$, where $A_{[\alp]}$ is a real-valued
  1-form and $\phi_{[\alp]}$ is a $\bbC$-valued function on
  $Q_{\alp}$. We say that $(A_{[\alp]}, \phi_{[\alp]})$ are
  \emph{smooth compatible pairs} if for every $\alp, \bt$, there
  exists a gauge transformation $\chi_{[\alp \bt]} \in
  C^{\infty}_{t,x}(Q_{\alp} \cap Q_{\bt})$ such that the following
  properties hold:
  \begin{enumerate}
  \item \label{item:cp:1} For every $\alp$, we have $\chi_{[\alp
      \alp]} = 0$.
  \item \label{item:cp:2} For every $\alp, \bt$, we have
    \begin{equation}
      (A_{[\bt]}, \phi_{[\bt]}) = (A_{[\alp]} - \ud \chi_{[\alp \bt]}, e^{i \chi_{[\alp \bt]}} \phi_{[\alp]}) 
      \quad \hbox{ on } Q_{\alp} \cap Q_{\bt}.
    \end{equation}
  \item \label{item:cp:3} For every $\alp, \bt, \gmm$, the following
    \emph{cocycle condition} is satisfied:
    \begin{equation}
      \chi_{[\alp \bt]} + \chi_{[\bt \gmm]} + \chi_{[\gmm \alp]} \in 2 \pi \bbZ 
      \quad \hbox{ on } Q_{\alp} \cap Q_{\bt} \cap Q_{\gmm}.
    \end{equation}
  \end{enumerate}
\end{definition}

The notion of (gauge-)\emph{equivalence} of compatible pairs is
defined as follows.
\begin{definition}[Equivalence of smooth compatible
  pairs] \label{def:smth-cp:equiv} Let $\calO \subseteq \bbR^{1+4}$ be
  an open set, and let $\calQ = \set{Q_{\alp}}$, $\calQ' =
  \set{Q_{\bt}'}$ be locally finite open coverings of
  $\calO$. Consider two sets of smooth compatible pairs $(A_{[\alp]},
  \phi_{[\alp]})$ and $(A'_{[\bt]}, \phi'_{[\bt]})$ on $\calQ$ and
  $\calQ'$, respectively.  When $\calQ'$ is a refinement of $\calQ$
  (i.e., for every $\bt$ there exists $\alp(\bt)$ such that $Q'_{\bt}
  \subseteq Q_{\alp}$), we say that $(A_{[\alp]}, \phi_{[\alp]})$ and
  $(A'_{[\bt]}, \phi'_{[\bt]})$ are (gauge-)\emph{equivalent} if for
  every $\bt$ there exists $\chi_{[\bt]} \in
  C^{\infty}_{t,x}(Q'_{\bt})$ such that $(A'_{[\bt]}, \phi'_{[\bt]})
  =(A_{[\alp]} - \ud \chi_{[\bt]}, \phi_{[\alp]} e^{i \chi_{[\bt]}})$.
  In the general case, we say that $(A_{[\alp]}, \phi_{[\alp]})$ and
  $(A'_{[\bt]}, \phi'_{[\bt]})$ are (gauge-)\emph{equivalent} if there
  exists a common refinement $\calQ''$ of $\calQ$, $\calQ'$ and a set
  of smooth compatible pairs $(A''_{[\gmm]}, \phi''_{[\gmm]})$ on
  $\calQ''$ which is equivalent to both $(A_{[\alp]}, \phi_{[\alp]})$
  and $(A'_{[\bt]}, \phi'_{[\bt]})$.
\end{definition}

\begin{remark} 
  In more geometric terms, compatible pairs $(A_{[\alp]},
  \phi_{[\alp]})$ on $Q_{\alp}$ are precisely expressions of a
  connection $A$ and a section $\phi$ of a complex line bundle $L$ in
  local trivializations $L \rst_{Q_{\alp}} \simeq Q_{\alp} \times
  \bbC$.  Moreover, equivalent sets of compatible pairs are
  alternative expressions of the same global pair $(A, \phi)$.
\end{remark}

In fact, expression of connections and sections in local
trivializations in the fashion of Definition~\ref{def:smth-cp} is
necessary if the complex line bundle $L$ under consideration is
topologically nontrivial (i.e., $L$ is \emph{not} homeomorphic to the
product of $\bbC$ and the base space). In our setting, however, there
is no loss of generality in simply identifying connections and
sections of $L$ with real-valued 1-forms and complex-valued functions,
respectively, as all base spaces we consider (e.g., $\calO = I \times
\bbR^{4}$ or $C^{T}_{[T, \infty)}$ for some $T > 0$) are contractible
and hence all complex line bundles over such spaces are topologically
trivial. In this case, every smooth compatible pairs on $\calO$ is
equivalent to a global smooth pair $(A, \phi)$ on $\calO$.

\begin{remark} 
We emphasize that no delicate patching is needed for smooth compatible pairs
in this paper, since all we need is merely the soft fact that the energy argument in 
Section~\ref{sec:energy} and the stress tensor argument in Section~\ref{sec:stationary-self-sim} 
(which are both gauge invariant) can be justified. 
In contrast, in \cite{OT1} an elaborate patching argument had to be developed 
in order to control the $S^{1}$ norm of the equivalent global pair in the Coulomb gauge.
\end{remark}

Based on the spaces introduced for the weak formulation of
\eqref{eq:MKG} discussed above, we can also formulate the notion of
\emph{weak compatible pairs}.
\begin{definition}[Weak compatible pairs] \label{def:weak-cp} Let
  $\calO \subseteq \bbR^{1+4}$ be an open set and let $\calQ =
  \set{Q_{\alp}}$ be a locally finite covering of $\calO$. For each
  index $\alp$, consider a pair $(A_{[\alp]}, \phi_{[\alp]}) \in
  \Xw(Q_{\alp})$. We say that $(A_{[\alp]}, \phi_{[\alp]})$ are
  \emph{weak compatible pairs} if for every $\alp, \bt$, there exists
  a gauge transformation $\chi_{[\alp \bt]} \in \Yw(Q_{\alp} \cap
  Q_{\bt})$ such that the properties
  (\ref{item:cp:1})--(\ref{item:cp:3}) in Definition~\ref{def:smth-cp}
  hold almost everywhere.
\end{definition}

The notion of equivalent sets of weak compatible pairs is defined as
in Definition~\ref{def:smth-cp:equiv}, where the space
$C_{t,x}(Q'_{\bt})$ is replaced by $\Yw(Q'_{\bt})$.

Geometrically, weak compatible pairs $(A_{[\alp]}, \phi_{[\alp]})$ may
be thought of as local descriptions of a connection and a section
defined on a \emph{rough} complex line bundle $L$.  A simple but
crucial observation is that smoothness of the pairs $(A_{[\alp]},
\phi_{[\alp]})$ implies smoothness of the gauge transformations
$\chi_{[\alp \bt]}$. Indeed, simply note that $\ud \chi_{[\alp \bt]} =
A_{[\alp]} - A_{[\bt]}$ by the property (\ref{item:cp:2}) in
Definition~\ref{def:smth-cp}. As this fact will play an important role
in our argument (see Proposition~\ref{prop:reg}), we record it as a
separate lemma.
\begin{lemma} \label{lem:smooth-weak-cp} Let $\calQ = \set{Q_{\alp}}$
  be an open cover of $\calO \subseteq \bbR^{1+4}$, and let
  $(A_{[\alp]}, \phi_{[\alp]})$ on $Q_{\alp}$ be weak compatible
  pairs. If $A_{[\alp]}, \phi_{[\alp]} \in C^{\infty}(Q_{\alp})$ for
  every $\alp$, then $(A_{[\alp]}, \phi_{[\alp]})$ form \emph{smooth
    compatible pairs} in the sense of Definition~\ref{def:smth-cp}.
\end{lemma}

We end this subsection with another simple lemma, which will be used
later to show that the local solutions obtained from
Proposition~\ref{prop:cpt} in the limit form weak compatible pairs.
\begin{lemma} \label{lem:weak-cp-limit} Let $Q_{1}, Q_{2} \subseteq
  \bbR^{1+4}$ be open sets such that $Q_{1} \cap Q_{2} \neq \0$ is an
  open bounded set with a piecewise smooth boundary. Consider
  sequences $(A_{[\alp]}^{(n)}, \phi_{[\alp]}^{(n)}) \in
  \Xw(Q_{\alp})$ ($\alp = 1,2$) and $\chi_{[12]}^{(n)} \in \Yw(Q_{1}
  \cap Q_{2})$ such that
  \begin{equation} \label{eq:weak-cp-seq} (A_{[2]}^{(n)},
    \phi_{[2]}^{(n)}) = (A_{[1]}^{(n)} - \ud \chi_{[12]}^{(n)},
    \phi_{[1]}^{(n)} e^{i \chi_{[12]}^{(n)}}) \quad \hbox{ a.e. on }
    Q_{1} \cap Q_{2}.
  \end{equation}
  In other words, $(A_{[\alp]}^{(n)}, \phi_{[\alp]}^{(n)})$ are weak
  compatible pairs for each $n$. Suppose furthermore that each
  sequence $(A_{[\alp]}^{(n)}, \phi_{[\alp]}^{(n)})$ has a limit
  $(A_{[\alp]}, \phi_{[\alp]})$ in $\Xw(Q_{\alp})$ as $n \to \infty$.
  Then the limits $(A_{[\alp]}, \phi_{[\alp]})$ ($\alp=1,2$) also form
  weak compatible pairs, i.e., there exists $\chi_{[12]} \in \Yw(Q_{1}
  \cap Q_{2})$ such that
  \begin{equation} \label{eq:weak-cp-limit} (A_{[2]}, \phi_{[2]} ) =
    (A_{[1]} - \ud \chi_{[12]}, \phi_{[1]} e^{i \chi_{[12]}}) \quad
    \hbox{ a.e. on } Q_{1} \cap Q_{2}.
  \end{equation}
  Moreover, there exists a subsequence of $\chi_{[12]}^{(n)}$ that
  converges\footnote{That is, there exists $k_{m} \in \bbZ$ such that
    $\chi_{[12]}^{(n_{m})} + 2 \pi k_{m} \to \chi_{[12]}$ in
    $\Yw(Q_{1} \cap Q_{2})$.} to $\chi_{[12]}$ in $\Yw(Q_{1} \cap
  Q_{2})$ up to integer multiples of $2 \pi$.
\end{lemma}

\begin{proof} 
  Let $\overline{\chi}_{[12]}^{(n)} := \int_{Q_{1} \cap Q_{2}}
  \chi_{[12]}^{(n)}$ denote the mean of $\chi_{[12]}^{(n)}$. By
  Poincar\'e's inequality, the identity $\ud \chi_{[12]}^{(n)} =
  A_{[1]}^{(n)} - A_{[2]}^{(n)}$ and the $L^{2}_{t,x}$ convergence of
  $A_{[\alp]}^{(n)}$ ($\alp = 1, 2$), the mean-zero part
  $\hat{\chi}_{[12]}^{(n)} := \chi_{[12]}^{(n)} -
  \overline{\chi}_{[12]}^{(n)}$ converges to a limit
  $\hat{\chi}_{[12]}$ in $\Yw(Q_{1} \cap Q_{2}) = H^{1}_{t,x}(Q_{1}
  \cap Q_{2})$. On the other hand, we can easily extract a convergent
  subsequence from the bounded sequence $e^{i
    \overline{\chi}_{[12]}^{(n)}}$; abusing the notation a bit, we
  denote the subsequence still by $e^{i
    \overline{\chi}_{[12]}^{(n)}}$, and the limit by $e^{i
    \overline{\chi}_{[12]}}$ for some $\overline{\chi}_{[12]} \in
  \bbR$. It follows that $\chi_{[12]}^{(n)}$ converges to $\chi_{[12]}
  := \hat{\chi}_{[12]} + \overline{\chi}_{[12]}$ in $\Yw(Q_{1} \cap
  Q_{2})$ as $n \to \infty$ up to integer multiples of $2 \pi$. The
  desired gauge equivalence in the limit \eqref{eq:weak-cp-limit} is
  now an easy consequence of \eqref{eq:weak-cp-seq} and the above
  convergences.
%
\end{proof}

\section{Stationary / self-similar solutions with finite energy} \label{sec:stationary-self-sim}
In the context of the blow-up analysis to be performed in Section~\ref{sec:proof}, the local strong compactness result (Proposition~\ref{prop:cpt}) will give rise to two types of solutions to \eqref{eq:MKG}: 
\begin{itemize}
\item A \emph{stationary solution} $(A, \phi)$, which is defined by the property 
\begin{equation} \label{eq:stationary}
	\iota_{Y} F = 0, \quad \covD_{Y} \phi = 0 
\end{equation}
for some constant time-like vector field $Y$; or
\item A \emph{self-similar solution} $(A, \phi)$, defined by the property
\begin{equation} \label{eq:self-sim}
	\iota_{X_{0}} F = 0, \quad (\covD_{X_{0}} + \frac{1}{\rho}) \phi = 0.
\end{equation}
\end{itemize}

In Sections~\ref{subsec:st} and \ref{subsec:ss}, we show that such solutions must be \emph{trivial} under the finite energy assumption. We use the method of stress tensor, which is the elliptic version of the energy-momentum-stress tensor considered in Section~\ref{sec:energy}. In Section~\ref{subsec:reg}, we establish an elliptic regularity result for these solutions under the improved regularity assumption \eqref{eq:cpt:reg} ensured by Proposition~\ref{prop:cpt}.

\subsection{Triviality of finite energy stationary solutions} \label{subsec:st}
As any unit constant time-like vector field $Y$ can be transformed to the vector field $T = \rd_{t}$ in the rectilinear coordinates, we may assume that $Y = T$.  Our main result in this case is as follows.
\begin{proposition} \label{prop:trivial:st}
Let $(A, \phi)$ be a smooth solution to \eqref{eq:MKG} on $\bbR^{1+4}$ with $\iota_{T} F = 0$ and $\covD_{T} \phi = 0$. Suppose furthermore that $(A, \phi)$ has finite energy, i.e., $\calE_{\set{0} \times \bbR^{4}}[A, \phi] < \infty$. Then $\calE_{\set{0} \times \bbR^{4}}[A, \phi] = 0$.
\end{proposition}

\begin{proof} 
We use the rectilinear coordinates $(t= x^{0}, x^{1}, \ldots, x^{4})$, in which $T = \rd_{t}$. By the stationarity assumptions $(\iota_{T} F)(\rd_{j}) = F_{0j}= 0$ and $\covD_{T} \phi = \covD_{0} \phi = 0$, \eqref{eq:MKG} reduces to the following elliptic system on each constant $t$ hypersurface:
\begin{equation} \label{eq:stMKG}
\left\{
\begin{aligned}
	\rd^{\ell} F_{j \ell} =& \Im (\phi \overline{\covD_{j} \phi}),  \\
	\covD^{\ell} \covD_{\ell} \phi =& 0.
\end{aligned}
\right.
\end{equation}
Henceforth, we work with $F, \phi$ restricted to the hypersurface $\set{t = 0}$. 

For the purpose of showing $\calE[A, \phi] = 0$, consider the following stress tensor associated to \eqref{eq:stMKG}:
\begin{equation}
	\EM_{j k} [A, \phi] := \Re(\covD_{j} \phi \overline{\covD_{k} \phi}) - \frac{1}{2} \dlt_{jk} \Re(\covD_{k} \phi \overline{\covD^{k} \phi})
					+ F_{j \ell} \tensor{F}{_{k}^{\ell}} - \frac{1}{4} \dlt_{jk} F_{\ell m} F^{\ell m}.
\end{equation}
Given a vector field $S$ on $\bbR^{4}$, we define as before the associated $1$-and $0$-currents
\begin{equation*}
	\vC{S}_{j}[A, \phi] := \EM_{jk}[A,\phi] S^{k}, \quad
	\sC{S}[A, \phi] := \EM_{jk}[A, \phi] \defT{S}^{jk}
\end{equation*}
which, thanks to \eqref{eq:stMKG}, satisfy the divergence identity
\begin{equation} \label{eq:stMKG:i-by-p}
	\nb^{\bfa} (\vC{S}_{\bf}[A, \phi]_{\bfa}) = \sC{S}[A, \phi].
\end{equation}
Choosing $S$ to be the scaling vector field on $\bbR^{4}$ so that, in the rectilinear coordinates
\begin{equation*}
S^{k} = x^{k}, \quad \defT{S}^{jk} = 2 \dlt^{jk},
\end{equation*}
we have
\begin{align*}
	\sC{S}[A, \phi] = - 2 \abs{\covD \phi}^{2}, \quad
	\abs{\vC{S}_{j}[A, \phi]} \aleq \abs{x} \abs{\covD \phi}^{2} + \abs{x} \abs{F}^{2}.
\end{align*}
where $\abs{\covD \phi}^{2} = \sum_{j=1}^{4} \abs{\covD_{j} \phi}^{2}$ and $\abs{F}^{2} = \sum_{1 \leq j < k \leq 4} \abs{F_{jk}}^{2}$. 

We now integrate \eqref{eq:stMKG:i-by-p} by parts on a ball $B_{R} \subseteq \bbR^{4}$ of radius $R > 1$ centered at $0$. Then we see that
\begin{equation} \label{}
	- 2 \int_{B_{R}} \abs{\covD u}^{2} \, \ud x
	= \int_{\rd B_{R}} \vC{S}[A, u]_{\bfa} \bfn^{\bfa}, \quad  \hbox{ where } \bfn = \frac{x^{\ell}}{\abs{x}} \rd_{\ell}.
\end{equation}
By the finite energy condition, we have $\abs{\covD \phi}, \abs{F} \in L^{2}(\bbR^{4})$; this fact is enough to deduce the existence of a sequence of radii $R_{n} \to \infty$ along which the boundary integral vanishes. 
Hence it follows that $\covD_{x} \phi = 0$.

It only remains to show that $F = 0$. Note that $F$ is now a harmonic 2-form in $L^{2}(\bbR^{4})$, as $\ud F = \ud^{2} A = 0$ and the right-hand side of the first equation in \eqref{eq:stMKG} vanishes. Therefore, each component $F_{jk}$ is a harmonic function. By the non-existence\footnote{This fact can be proved using the monotonicity \eqref{eq:harmonic-monotonicity}, which holds for all $0 < r_{1} < r_{2}$ for harmonic functions on $\bbR^{4}$.} of nontrivial harmonic functions in $L^{2}(\bbR^{4})$, it follows that $F =0$, which completes the proof.   \qedhere

\end{proof}

\subsection{Triviality of finite energy self-similar solutions} \label{subsec:ss}
In the case of a self-similar solution with finite energy, our main result is as follows.
\begin{proposition}  \label{prop:trivial:ss}
Let $(A, \phi)$ be a smooth solution to \eqref{eq:MKG} on the forward light cone $C_{(0, \infty)}$ with $\iota_{X_{0}} F = 0$ and $\covD_{X_{0}} \phi + \frac{1}{\rho} \phi = 0$. Suppose furthermore that $(A, \phi)$ has finite energy, i.e., $\sup_{t \in (0, \infty)} \calE_{S_{t}}[A, \phi] < \infty$.
Then $\calE_{S_{t}}[A, \phi] = 0$ for all $t > 0$.
\end{proposition}

\begin{proof} 
We use the hyperbolic coordinates $(\rho, y, \Tht)$, in which $X_{0} = \rd_{\rho}$. By the self-similarity assumption $\iota_{X_{0}} F (\cdot) =  F(\rd_{\rho}, \cdot) = 0$ and $\covD_{\rd_{\rho}} \phi = - \frac{1}{\rho} \phi$, it follows that the pullback of $(A, \phi)$ to $\calH_{1} = \set{\rho = 1} = \bbH^{4}$, which we still denote by $(A, \phi)$, solves the system
\begin{equation} \label{eq:ssMKG}
	\left\{
\begin{aligned}
	- \mathrm{div}_{\bbH^{4}} F =& \Im ( \phi \overline{\covD_{\bbH^{4}} \phi}), \\
	(-\lap_{\bbH^{4}, A} - 2) \phi =& 0,
\end{aligned}
	\right.
\end{equation} 
where $F = \ud A$, $(\mathrm{div}_{\bbH^{4}} F)_{\bfa} = \nb_{\bbH^{4}}^{\bfb} F_{\bfb \bfa}$, $\covD_{\bbH^{4}} = \nb_{\bbH^{4}} + i A$ and $\lap_{\bbH^{4}, A} = \covD_{\bbH^{4}}^{\bfa} \covD_{\bbH^{4}, \bfa}$.
Furthermore, by Proposition~\ref{prop:energy-H-rho} applied to $\calH_{1} = \bbH^{4}$, we have
\begin{gather} 
	\int_{\bbH^{4}} \frac{1}{2} \cosh y \, \abs{F}_{\bbH^{4}}^{2} \, \ud \sgm_{\bbH^{4}} < \infty, \label{eq:ssEnergy:M} \\
	\int_{\bbH^{4}} \frac{1}{2} \bb[ \cosh y \abs{\phi}^{2} 
	+ 2 \sinh y \Re [ \phi \overline{\covD_{y} \phi} ] 
	+ \cosh y \abs{\covD \phi}_{\bbH^{4}}^{2} \bb] 
	\ud \sgm_{\bbH^{4}} < \infty. \label{eq:ssEnergy:KG}
\end{gather}
where $\abs{F}_{\bbH^{4}}^{2} = \frac{1}{2} (g_{\bbH^{4}}^{-1})^{\bfa \bfc} (g_{\bbH^{4}}^{-1})^{\bfb \bfd} F_{\bfa \bfb} F_{\bfc \bfd}$ and $\abs{\covD \phi}_{\bbH^{4}}^{2} = (g_{\bbH^{4}}^{-1})^{\bfa \bfb} \covD_{\bfa} \phi \overline{\covD_{\bfb} \phi}$.

In order to proceed, we reformulate the system on $\bbD^{4}$ using the conformal equivalence of $\bbD^{4}$ and $\bbH^{4}$. 
Consider the following map from $\bbD^{4}$ to $\bbH^{4}$:
\begin{equation*}
	\Phi : \bbD^{4} \to \bbH^{4}, \quad
	(r, \Tht) \mapsto (y, \Tht) = (2 \tanh^{-1} r, \Tht)
\end{equation*}
The map $\Phi$ is a \emph{conformal isometry}, i.e.,
\begin{equation*}
	\Phi^{\ast} g_{\bbH^{4}} 
	= \Phi^{\ast} (\ud y^{2} + \sinh^{2} y \, g_{\bbS^{3}})
	= \Omg^{2} (\ud r^{2} + r^{2} \, g_{\bbS^{3}})
	= \Omg^{2} g_{\bbD^{4}},
\end{equation*}
where $\Phi^{\ast}$ denotes the pullback along $\Phi$ to $\bbD^{4}$, and $\Omg := \frac{2}{1-r^{2}}$. For the pulled-back pair $(\Phi^{\ast} A, \Omg \, \Phi^{\ast} \phi)$ on $\bbD^{4}$, which (slightly abusing the notation) we will denote by $(A, u)$,
we have
\begin{equation} \label{eq:ssMKG:D}
	\left\{
\begin{aligned}
		\rd^{\ell} F_{j \ell} =& \Im(u \overline{\covD_{j} u}) \\
		\covD^{\ell} \covD_{\ell} u =& 0.
\end{aligned}	
	\right.
\end{equation}
where $F = \ud A$ and $\covD = \nb + i A$. Moreover, the bounds \eqref{eq:ssEnergy:M} and \eqref{eq:ssEnergy:KG} then translate to
\begin{gather}
\int_{\bbD^{4}}
\frac{1}{2} \frac{1+r^{2}}{1-r^{2}} \abs{F}_{\bbD^{4}}^{2} \, \ud \sgm_{\bbD^{4}} < \infty, \label{eq:ssEnergy:M:D}\\
\int_{\bbD^{4}}
	\frac{1}{2} \bb[ 
	\frac{1}{1-r^{2}} \abs{r \covD_{r} u + 2 u}^{2}
	+ \frac{1}{1-r^{2}} \abs{\covD_{r} u}^{2} 
	+ \frac{1+r^{2}}{(1-r^{2})r^{2}} \abs{\scovD u}^{2}  \bb]
	\ud \sgm_{\bbD^{4}} < \infty. \label{eq:ssEnergy:KG:D}
\end{gather}
where $\abs{\scovD u}^{2} = (g_{\bbS^{3}}^{-1})^{\bfa \bfb} \covD_{\bfa} u \overline{\covD_{\bfb} u}$.
Indeed, note that
\begin{equation*}
	\Phi^{\ast} \ud \sgm_{\bbH^{4}} = \Omg^{4} \ud \sgm_{\bbD^{4}}, \quad 
	\Phi^{\ast} (\cosh y) = \frac{1 + r^{2}}{1 - r^{2}}, \quad
	\Phi^{\ast} (\sinh y) = \frac{2 r}{1 - r^{2}}.
\end{equation*}
From these identities and \eqref{eq:ssEnergy:M}, we immediately see that \eqref{eq:ssEnergy:M:D} holds. Moreover, \eqref{eq:ssEnergy:KG:D} follows from \eqref{eq:ssEnergy:KG} and the following computation:
\begin{align*}
&\hskip-1em
	\int_{\bbH^{4}} \frac{1}{2} \bb[ \cosh y \abs{\phi}^{2} 
	+ 2 \sinh y \Re [ \phi \overline{\covD_{y} \phi} ] 
	+ \cosh y \abs{\covD \phi}_{\bbH^{4}}^{2} \bb] 
	\ud \sgm_{\bbH^{4}} \\
=&
\int_{\bbD^{4}}
	\frac{1}{2} \bb[ \frac{1+r^{2}}{1-r^{2}} \Omg^{2} \abs{u}^{2}
	+ \frac{4 r}{1-r^{2}} \Re [\Omg u \overline{\Omg \covD_{r} (\Omg^{-1} u)}]
	+ \frac{1+r^{2}}{1-r^{2}} \bb( \abs{\Omg \covD_{r} (\Omg^{-1} u)}^{2} + \frac{1}{r^{2}} \abs{\scovD u}^{2} \bb) \bb]
	\ud \sgm_{\bbD^{4}} \\	
=&
\int_{\bbD^{4}}
	\frac{1}{2} \bb[ \frac{4}{1-r^{2}} \abs{u}^{2}
	+ \frac{4r}{1-r^{2}} \Re [u \overline{\covD_{r} u}] 
	+ \frac{r^{2}+1}{1-r^{2}} \abs{\covD_{r} u}^{2}
	+ \frac{1+r^{2}}{(1-r^{2})r^{2}} \abs{\scovD u}^{2}  \bb]
	\ud \sgm_{\bbD^{4}} \\
=&
\int_{\bbD^{4}}
	\frac{1}{2} \bb[ 
	\frac{1}{1-r^{2}} \abs{r \covD_{r} u + 2 u}^{2}
	+ \frac{1}{1-r^{2}} \abs{\covD_{r} u}^{2} 
	+ \frac{1+r^{2}}{(1-r^{2})r^{2}} \abs{\scovD u}^{2}  \bb]
	\ud \sgm_{\bbD^{4}}.
\end{align*}
%

We will now show that \eqref{eq:ssMKG:D}, \eqref{eq:ssEnergy:M:D} and \eqref{eq:ssEnergy:KG:D} imply $u = 0$ on $\bbD^{4}$. 
Since the system \eqref{eq:ssMKG:D} coincides with \eqref{eq:stMKG} restricted to $\bbD^{4}$, the divergence identity \eqref{eq:stMKG:i-by-p} can be used in the present context as well.
Integrating \eqref{eq:stMKG:i-by-p} by parts on a ball $B_{R} \subseteq \bbD^{4}$ of radius $R < 1$ centered at $0$, we see that
\begin{equation} \label{}
	- 2 \int_{B_{R}} \abs{\covD u}^{2} \, \ud \sgm_{\bbD^{4}}
	= \int_{\rd B_{R}} \vC{S}[A, u]_{\bfa} \bfn^{\bfa}, \quad  \hbox{ where } \bfn = \frac{x^{\ell}}{\abs{x}} \rd_{\ell}.
\end{equation}
Observe that \eqref{eq:ssEnergy:M:D} and \eqref{eq:ssEnergy:KG:D} imply the existence of a sequence $R_{n} \to 1$ such that
\begin{equation*}
	\int_{\rd B_{R_{n}}} \abs{\vC{S}[A, u]_{\bfa} \bfn^{\bfa} } \to 0,
\end{equation*}
which shows that $\covD u = 0$ on $\bbD^{4}$. Plugging this information into \eqref{eq:ssEnergy:KG:D}, it follows that $u = 0$ on $\bbD^{4}$, as desired.

To complete the proof, it only remains to show that $F = 0$. As before, $F$ is now a harmonic 2-form in $L^{2}(\bbD^{4})$ by \eqref{eq:ssMKG}; hence each component $F_{j k}$ is a harmonic function on $\bbD^{4}$. Fix $j, k \in \set{1,2,3,4}$ and observe that $\varphi := F_{j k}$, viewed as a real-valued function, obeys the following monotonicity property:
\begin{equation} \label{eq:harmonic-monotonicity}
	\frac{1}{r_{1}^{3}} \int_{\rd B_{r_{1}}} \abs{\varphi}^{2} \leq \frac{1}{r_{2}^{3}} \int_{\rd B_{r_{2}}} \abs{\varphi}^{2} \quad \hbox{ where } 0 < r_{1} < r_{2} < 1.
\end{equation}
Indeed, \eqref{eq:harmonic-monotonicity} is a consequence of interpolating the inequalities
\begin{equation*}
	\frac{1}{r_{1}^{3}} \int_{\rd B_{r_{1}}} \abs{\varphi} \leq \frac{1}{r_{2}^{3}} \int_{\rd B_{r_{2}}} \abs{\varphi} , \qquad
	\sup_{\rd B_{r_{1}}} \abs{\varphi} \leq \sup_{\rd B_{r_{2}}} \abs{\varphi} \qquad \hbox{ where } 0 < r_{1} < r_{2} < 1,
\end{equation*}
which follow from the mean-value property and the weak maximum principle for the subharmonic function $\abs{\varphi}$ on $\bbD^{4}$, respectively.
By \eqref{eq:ssEnergy:M:D}, it follows that $F_{jk} = \varphi = 0$ on $\bbD^{4}$. \qedhere
\end{proof}

\subsection{Regularity of stationary and self-similar weak solutions to \eqref{eq:MKG}} \label{subsec:reg}
We end this section with a regularity result, which applies to weak solutions obtained by Proposition~\ref{prop:cpt}.
\begin{proposition} \label{prop:reg}
Let $(A, \phi)$ be a weak solution to \eqref{eq:MKG} on an open set $\calO \subseteq \bbR^{1+4}$ such that
\begin{equation} \label{eq:reg:hyp}
\begin{aligned}
	A_{\mu} \in H^{1}_{t,x}(\calO), \quad 
	\phi \in H^{\frac{3}{2}}_{t,x}(\calO).
\end{aligned}
\end{equation}
Suppose furthermore that one of the following holds: 
\begin{enumerate}
\item \label{item:reg:st} Either $(A, \phi)$ is \emph{stationary} on $\calO$ in the sense of \eqref{eq:stationary}; or 
\item \label{item:reg:ss} The set $\calO$ is a subset of the cone $C_{(0, \infty)} = \set{0 \leq r < t}$ and $(A, \phi)$ is \emph{self-similar} on $\calO$ in the sense of \eqref{eq:self-sim}.
\end{enumerate}
Then for every $p \in \calO$, there exists an open neighborhood $p \in Q_{p} \subseteq \calO$ and a gauge transformation $\chi_{[p]} \in \Yw(Q_{p})$ such that $(A_{[p]}, \phi_{[p]}) = (A - \ud \chi_{[p]}, \phi e^{i \chi_{[p]}})$ is smooth on $Q_{p}$. 
\end{proposition}

\begin{proof} 
The idea is to derive an elliptic system as in \eqref{eq:stMKG} [resp. \eqref{eq:ssMKG}] using stationarity [resp. self-similarity], and then use its regularity theory. To get rid of the non-local operator $\brk{D_{t,x}}^{\frac{3}{2}}$ in the norm, we begin with the following simple maneuver: For any open bounded subset $Q \subseteq \calO$ with smooth boundary, by Sobolev  and \eqref{eq:reg:hyp}, we have
\begin{equation} \label{eq:reg:hyp:W1q}
	A_{\mu} \in H^{1}_{t,x}(Q), \quad
	\phi \in W^{1, q}_{t,x}(Q)
\end{equation} 
where $q = \frac{5}{2}$. The important point is that $q > 2$, which will make this bound \emph{subcritical}. Hence we would be able to conclude regularity via a simple elliptic bootstrap argument.

We first treat Case~\ref{item:reg:st}. Applying a suitable Lorentz transformation, it suffices to consider the case $Y = \rd_{t}$ in the rectilinear coordinates $(t = x^{0}, x^{1}, \ldots, x^{4})$. Moreover, applying an appropriate space-time translation, we may assume that $p$ is the origin. Let $Q_{p} := (-\dlt, \dlt) \times \dlt B$, where $\dlt B$ is the open ball of radius $\dlt$ centered at the origin. Choosing $\dlt > 0$ small enough, we have $Q_{p} \subseteq \calO$. By \eqref{eq:reg:hyp:W1q} and Fubini, there exists $\overline{t} \in (-\dlt, \dlt)$ such that 
\begin{equation} \label{eq:reg:fixed-t}
	A \rst_{\overline{t} \times \dlt B} \in H^{1}(\dlt B), \quad
	\phi \rst_{\overline{t} \times \dlt B} \in W^{1, q}(\dlt B),
\end{equation}
where the shorthand $\overline{t} = \set{\overline{t}}$ is used for simplicity.
We claim that there exists $\chi_{[p]} \in \Yw((-\dlt, \dlt) \times \dlt B)$ so that $\chi_{[p]} \rst_{\overline{t} \times \dlt B} \in H^{2}(\dlt B)$ and
\begin{equation} \label{eq:reg:chi-st}
	\rd_{t} \chi_{[p]} = A_{0} \hbox{ in } (-\dlt, \dlt) \times \dlt B, \quad 
	\lap \chi_{[p]} \rst_{\overline{t} \times \dlt B}= \rd^{\ell} (A \rst_{\overline{t} \times \dlt B})_{\ell}.
\end{equation} 
Indeed, we may simply define $\underline{\chi}_{[p]} = \lap^{-1} \rd^{\ell} (\eta A \rst_{\set{t = \overline{t} }})_{\ell}$, where $\eta \in C^{\infty}_{0}(\bbR^{4})$ satisfies $\eta = 1$ on $\dlt B$ and $\supp \, \eta \subseteq \calO$, then solve the transport equation $\rd_{t} \chi_{[p]} = A_{0} \hbox{ in } (-\dlt, \dlt) \times \dlt B$ with initial data $\chi_{[p]} \rst_{\overline{t} \times \dlt B} = \underline{\chi}_{[p]}$. That this $\chi_{[p]}$ belongs to $\Yw((-\dlt, \dlt) \times \dlt B)$ and $\chi_{[p]} \rst_{\overline{t} \times \dlt B} \in H^{2}(\dlt B)$ easily follow from the bounds for $A$ in \eqref{eq:reg:hyp:W1q} and \eqref{eq:reg:fixed-t}.

Consider now the gauge transform $(A_{[p]}, \phi_{[p]}) = (A - \ud \chi_{[p]}, \phi e^{i \chi_{[p]}})$. By \eqref{eq:reg:chi-st}, we have
\begin{equation} \label{eq:reg:Ap:st}
	A_{[p] 0} = 0 \hbox{ in } (-\dlt, \dlt) \times \dlt B, \quad 
	\rd^{\ell} (A_{[p]} \rst_{\overline{t} \times \dlt B})_{\ell} = 0  \hbox{ in } \dlt B.
\end{equation} 
By the stationarity assumption $\iota_{\rd_{t}} F = 0$ and $\covD_{\rd_{t}} \phi = 0$, it follows that 
\begin{equation*}
	\rd_{t} A_{[p] j} = F_{0j} = 0, \quad \rd_{t} \phi_{[p]} = 0 \hbox{ in } (-\dlt, \dlt) \times \dlt B.
\end{equation*} 
Hence to prove that $(A_{[p]}, \phi_{[p]})$ is smooth in $Q_{p}$, it suffices to show that $(A_{[p]}, \phi_{[p]}) \rst_{\overline{t} \times \dlt B}$ is smooth. 
Abusing the notation slightly for simplicity, we will henceforth write $A = A_{[p]} \rst_{\overline{t} \times \dlt B}$ and $\phi = \phi_{[p]} \rst_{\overline{t} \times \dlt B}$. By \eqref{eq:stMKG} and \eqref{eq:reg:Ap:st} (in particular, the Coulomb condition for $A$), $(A, \phi)$ satisfies an elliptic system on $\dlt B$ of the schematic form
\begin{align*}
	\lap A =& \phi \rd \phi + \phi A \phi, \\
	\lap \phi = & A \rd \phi + A A \phi.
\end{align*}
Moreover, $(A, \phi)$ belongs to $A \in H^{1}(\dlt B)$ and $\phi \in W^{1, q}(\dlt B)$, thanks to \eqref{eq:reg:fixed-t} and $\chi_{[p]} \rst_{\overline{t} \times \dlt B}\in H^{2}(\dlt B)$. As this system is $H^{1}$-critical and every nonlinear term has at least one factor of $\phi$, which obeys a \emph{subcritical} bound $\phi \in W^{1, q}(\dlt B)$, we can perform a standard elliptic bootstrap argument to conclude that $(A, \phi)$ is smooth on $\dlt B$ with uniform bounds on compact subsets. This concludes the proof in Case~\ref{item:reg:st}.

The proof in Case~\ref{item:reg:ss} is entirely analogous to Case~\ref{item:reg:st}, so we only give a brief outline. Here, instead of the rectilinear coordinates, we use the hyperbolic coordinates $(\rho, y, \Tht)$, in which $X = \rd_{\rho}$. Applying a suitable Lorentz transformation and scaling transformation, we may assume that $p$ coincides with the point $\rho = 1$, $y = 0$. Let $Q_{p} = (-\dlt, \dlt) \times D_{\dlt}$, where $D_{\dlt} := \set{(y, \Tht) : \abs{y} < \dlt}$, which is contained in $\calO$ if $\dlt > 0$ is sufficiently small. By \eqref{eq:reg:hyp:W1q} and Fubini, there exists $\overline{\rho} \in (-\dlt, \dlt)$ such that
\begin{equation} \label{eq:reg:const-rho}
	A \rst_{\overline{\rho} \times D_{\dlt}} \in H^{1}(D_{\dlt}), \quad 
	\phi \rst_{\overline{\rho} \times D_{\dlt}} \in W^{1, p}(D_{\dlt}).
\end{equation}
Proceeding as before, we can find $\chi_{[p]} \in \Yw ((-\dlt, \dlt) \times D_{\dlt})$ so that $\chi_{[p]} \rst_{\overline{\rho} \times D_{\dlt}} \in H^{2}(D_{\dlt})$ and
\begin{equation*}
	\rd_{\rho} \chi_{[p]} = 0 \hbox{ in } (-\dlt, \dlt) \times D_{\dlt}, \quad
	\lap_{\calH_{\overline{\rho}}} \, \chi_{[p]} \rst_{\overline{\rho} \times D_{\dlt}} = \nb_{\calH_{\overline{\rho}}}^{\bfa} (A \rst_{\overline{\rho} \times D_{\dlt}})_{\bfa}  \, .
\end{equation*}
Then the gauge transform $(A_{[p]}, \phi_{[p]}) = (A - \ud \chi_{[p]}, \phi e^{i \chi_{[p]}})$ obeys
\begin{equation*}
	A_{[p]}(\rd_{\rho}) = 0 \hbox{ in } (-\dlt, \dlt) \times D_{\dlt}, \quad
	\nb_{\calH_{\overline{\rho}}}^{\bfa} (A_{[p]} \rst_{\overline{\rho} \times D_{\dlt}})_{\bfa} = 0 \hbox{ in } D_{\dlt}.
\end{equation*}
By self-similarity, we have $\calL_{\rd_{\rho}} A_{[p]} = 0$ and $\rd_{\rho} (\rho \phi_{[p]}) = 0$, so it only remains to prove that the pullback of $(A_{[p]}, \phi_{[p]})$ on $\overline{\rho} \times D_{\dlt}$, which we will refer to as $(A, \phi)$, is smooth. As in the previous case, this is a consequence of the fact that $(A, \phi)$ obeys an elliptic system (thanks to \eqref{eq:ssMKG} and the Coulomb gauge condition on $\calH_{\overline{\rho}}$), the bounds $A \in H^{1}(D_{\dlt})$ and $\phi \in W^{1, q}(D_{\dlt})$ with $q > 2$ (by \eqref{eq:reg:const-rho} and $\chi_{[p]} \rst_{\overline{\rho} \times D_{\dlt}} \in H^{2}(D_{\dlt})$), and a standard elliptic bootstrap argument.  \qedhere
\end{proof}

\section{Proof of global well-posedness and scattering} \label{sec:proof}
Here we carry out the proof of Theorem~\ref{thm:main} using the tools developed in the earlier parts.
\subsection{Finite time blow-up/non-scattering scenarios and initial
  reduction} \label{subsec:ini-red} Our overall strategy for proving
Theorem~\ref{thm:main} is by contradiction.  Suppose that
Theorem~\ref{thm:main} fails for an initial data set $(a, e, f, g) \in
\calH^{1}$ in the global Coulomb gauge. By time reversal symmetry, it
suffices to consider the forward evolution. Let $(A, \phi)$ be the
admissible $C_{t} \calH^{1}$ solution to the Cauchy problem in the
global Coulomb gauge defined on the maximal forward time interval $I =
[0, T_{+})$ for some $T_{+} > 0$ constructed by
Theorem~\ref{thm:lwp4MKG}. By Theorem~\ref{thm:finite-S}, the solution
$(A, \phi)$ exhibits one of the following behaviors:
\begin{enumerate}
\item ({\bf Finite time blow-up}) We have $T_{+} < \infty$ and
  \begin{equation} \label{eq:dich:blow-up} \nrm{A_{0}}_{Y^{1}[0,
      T_{+})} + \nrm{A_{x}}_{S^{1}[0, T_{+})} + \nrm{\phi}_{S^{1}[0,
      T_{+})} = \infty.
  \end{equation}
\item ({\bf Non-scattering}) We have $T_{+} = \infty$, but
  \begin{equation} \label{eq:dich:non-scatter} \nrm{A_{0}}_{Y^{1}[0,
      \infty)} + \nrm{A_{x}}_{S^{1}[0, \infty)} + \nrm{\phi}_{S^{1}[0,
      \infty)} = \infty.
  \end{equation}
\end{enumerate}

In the case of finite time blow-up, we may use the energy
concentration scale $r_{c}$ in Theorem~\ref{thm:lwp4MKG} to show that
the energy must concentrate at a point.
\begin{lemma} \label{lem:EC-continue} Let $(A, \phi)$ be an admissible
  $C_{t} \calH^{1}$ solution to \eqref{eq:MKG} on $[0, T_{+}) \times
  \bbR^{4}$ with $T_{+} < \infty$ in the global Coulomb gauge. Then
  either $(A, \phi)$ can be continued past $T_{+}$ as an admissible
  $C_{t} \calH^{1}$ solution in the global Coulomb gauge (as in
  Theorem~\ref{thm:lwp4MKG}), or there exists a point $x_{0} \in
  \bbR^{4}$ such that
  \begin{equation} \label{eq:EC-at-x0} \limsup_{t \to T_{+}}
    \calE_{\set{t} \times B_{(T_{+} - t)}(x_{0})} [A, \phi] > 0.
  \end{equation}
\end{lemma}
\begin{proof}
For $ t < T_+$ and $x \in \bbR^4$ we define the function
\[
E(t,x) = \calE_{\set{t_{0}} \times B_{(T_{+} - t)}(x)} [A, \phi] 
\]
This is continuous in $x$, and, by the nonnegativity  of the flux in the energy relation \eqref{eq:F-G},
it is nonincreasing in $t$. Further, by the same relation, we have 
\[
\lim_{x \to \infty} E(t,x) = 0, \qquad \text{uniformly in } t \in [0,T_+).
\]
Then we have two alternatives:

(i) Either $\lim_{t \to T_+} \sup_{x\in \bbR^4} E(t,x) < \dlt_{0}(E, \thE^{2})$,
which implies that there exists $t_0$ so that  energy concentration scale $r_{c}$ at $t =
  t_{0}$ as in \eqref{eq:EC:def} is greater than $T_{+} - t_{0}$.  By
  Theorem~\ref{thm:lwp4MKG} we can then extend $(A, \phi)$ past
  $T_{+}$, as claimed. 

(ii) Or,  $\lim_{t \to T_+} \sup_{x\in \bbR^4} E(t,x) \geq \dlt_{0}(E, \thE^{2})$. Then the sets 
$D_t = \{ x \in \bbR^4; E(t,x) \geq \dlt_{0}(E, \thE^{2})\}$ are nonempty, compact, and decreasing in $t$.
Thus they must intersect. Any  $x_0$ in the intersection will provide the second alternative in the lemma.  
 \qedhere
\end{proof}

Theorem~\ref{thm:ED} provides additional information about the nature
of the singularity in both scenarios, which is crucial to our proof of
Theorem~\ref{thm:main}.  To utilize this information, we introduce a
smooth function $\zt$ satisfying the following properties:
\begin{itemize}
\item $\displaystyle{\supp \, \zt \subseteq B_{1}(0)}$ and $\int \zt =
  1$.

\item There exists a function $\widetilde{\zt} \in
  C^{\infty}_{0}(\bbR^{4})$ with $\widetilde{\zt} \geq 0$ such that
  $\zt = \widetilde{\zt} \ast \widetilde{\zt}$.
\end{itemize}
Then we define the physical space version of energy dispersion as
follows:
\begin{equation}
  \covED[A, \phi](I) := \sup_{k \in \bbZ} \bb( 2^{-k} \nrm{\zt_{2^{-k}} \ast \phi (t,x)}_{L^{\infty}_{t,x}(I \times \bbR^{4})} + 2^{-2k} \nrm{\zt_{2^{-k}} \ast \covD_{t} \phi (t,x)}_{L^{\infty}_{t,x}(I \times \bbR^{4})} \bb)
\end{equation}
where $\zt_{2^{-k}} := 2^{4k} \zt(2^{k} \cdot)$.  The first property
makes $\covED[A, \phi]$ simpler to use in physical space arguments; on
the other hand, the second property is helpful in connection with the
\emph{diamagnetic inequality}, which we state here.
\begin{lemma}[Diamagnetic inequality] \label{lem:diamagnetic} Let $O
  \subseteq \bbR^{4}$ be an open set and $\phi, A \in H^{1}(O)$. Then
  for any smooth vector $X$, $\abs{\rd_{X} \abs{\phi}} \leq
  \abs{\covD_{X} \phi}$ in the sense of distributions. More precisely,
  for any smooth $\eta \geq 0$ with $\supp \, \eta \, \subseteq O$, we
  have
  \begin{equation*}
    \int \eta \abs{\rd_{X} \abs{\phi}} \, \ud x \leq \int \eta \abs{\covD_{X} \phi} \, \ud x.
  \end{equation*}
\end{lemma}
The key to the proof is the formal computation $\abs{\rd_{X}
  \abs{\phi}} = \abs{\abs{\phi}^{-1} \brk{\phi, \covD_{X} \phi}} \leq
\abs{\covD_{X} \phi}$; we omit the standard details.  We fix the
choice of functions $\zt, \widetilde{\zt}$ here, and henceforth we
will suppress the dependence of constants on these functions for
simplicity.

The physical space version $\covED[A, \phi]$ is related to the earlier
Littlewood-Paley version $\ED[\phi]$ defined in \eqref{eq:EDC:def} as
follows.
\begin{lemma} \label{lem:EDC-covED} Let $(A, \phi)$ be an admissible
  $C_{t} \calH^{1}$ solution to \eqref{eq:MKG} on $I \times \bbR^{4}$
  in the global Coulomb gauge with $\calE_{\set{t} \times \bbR^{4}}[A,
  \phi] \leq E$. Then there exists $C = C(E)$ such that
  \begin{equation*}
    \ED[\phi](I) \leq C \, \covED[A, \phi](I) + \frac{1}{100} \thED(E),
  \end{equation*}
  where $\thED(E)$ is as in Theorem~\ref{thm:ED}.
\end{lemma}
\begin{proof}
  All norms in this proof will be taken over $I \times \bbR^{4}$.  The
  following estimates are straightforward to establish:
  \begin{align}
    \sup_{k} 2^{-k} \nrm{P_{k} \phi}_{L^{\infty}_{t,x}} \aleq & \sup_{k} 2^{-k} \nrm{\zt_{2^{-k}} \ast \phi}_{L^{\infty}_{t,x}}, \label{eq:EDC-covED:pf:1} \\
    \sup_{k} 2^{-2k} \nrm{P_{k} (\covD_{t} \phi)}_{L^{\infty}_{t,x}}
    \aleq & \sup_{k} 2^{-2k} \nrm{\zt_{2^{-k}} \ast (\covD_{t}
      \phi)}_{L^{\infty}_{t,x}}.\label{eq:EDC-covED:pf:2}
  \end{align}
  In view of \eqref{eq:EDC-covED:pf:1} and \eqref{eq:EDC-covED:pf:2},
  the lemma would follow once we prove that, for any $m_{1} > 10$,
  \begin{equation*}
    \sup_{k} 2^{-2k} \nrm{P_{k} \rd_{t} \phi}_{L^{\infty}_{t,x}}
    \aleq_E 2^{ m_{1}} \sup_{k} \bb( 2^{-2k} \nrm{P_{k} (\covD_{t} \phi)}_{L^{\infty}_{t,x}} + 2^{-k} \nrm{P_{k} \phi}_{L^{\infty}_{t,x}}\bb) + 2^{-m_{1}}.
  \end{equation*}
  By the relation $\rd_{t} = \covD_{t} - iA_{0}$, it suffices to show
  that
  \begin{equation} \label{eq:EDC-covED:pf:3} \sup_{k} 2^{-2k}
    \nrm{P_{k} (A_{0} \phi)}_{L^{\infty}_{t,x}} \aleq 2^{m_{1}}
    E^{\frac{1}{2}} \sup_{k} 2^{-k} \nrm{P_{k}
      \phi}_{L^{\infty}_{t,x}} + 2^{-m_{1}} (E+E^{\frac{3}{2}}).
  \end{equation}
  Thanks to the global Coulomb condition, we have
  \begin{equation*}
    \nrm{A_{0}}_{L^{\infty}_{t} \dot{H}^{1}_{x}} \aleq E^{1/2}, \quad 
    \nrm{\phi}_{L^{\infty}_{t} \dot{H}^{1}_{x}} \aleq E^{1/2} + E.
  \end{equation*}
  For each $k \in \bbZ$, we split $\phi = P_{\leq k + m_{1}} \phi +
  P_{> k + m_{1}} \phi$. For the former, we have
  \begin{align*}
    2^{-2k} \nrm{P_{k} (A_{0} P_{\leq
        k+m_{1}}\phi)}_{L^{\infty}_{t,x}} \aleq \sum_{\ell \leq k +
      m_{1}} 2^{\ell - k} \nrm{A_{0}}_{L^{\infty}_{t} L^{4}_{x}}
    2^{-\ell} \nrm{P_{\ell}\phi}_{L^{\infty}_{t,x}} \aleq 2^{m_{1}}
    E^{\frac{1}{2}} \sup_{\ell} 2^{-\ell}
    \nrm{P_{\ell}\phi}_{L^{\infty}_{t,x}}.
  \end{align*}
  For the latter, by the properties of frequency supports, note that
  \begin{equation*}
    P_{k} (A_{0} P_{>k + m_{1}} \phi) = \sum_{\ell > k+m_{1}} (P_{[\ell-3, \ell+3]} A_{0} P_{\ell} \phi).
  \end{equation*} 
  Hence \eqref{eq:EDC-covED:pf:3} follows from the estimate
  \begin{align*}
    2^{-2k} \nrm{P_{k} (A_{0} P_{>k + m_{1}}\phi)}_{L^{\infty}_{t,x}}
    \aleq & \sum_{\ell > k + m_{1}} 2^{2k} \nrm{P_{[\ell-3, \ell+3]}A_{0}}_{L^{\infty}_{t} L^{2}_{x}} \nrm{P_{\ell}\phi}_{L^{\infty}_{t} L^{2}_{x}} \\
    \aleq & 2^{-2m_{1}} (E + E^{3/2}).  \qedhere
  \end{align*}

\end{proof}

As a result, there exists a function $\thcovED = \thcovED(E) > 0$ such
that Theorem~\ref{thm:ED} holds with the condition
\eqref{eq:ED:small-EDC} replaced by
\begin{equation} \label{eq:ED:small-covED}
  \tag{$\ref{eq:ED:small-EDC}'$} \covED[A, \phi](I) \leq \thcovED(E).
\end{equation}

Let $\veps > 0$ be a small parameter to be chosen below. We have the
following result, which unifies the proof of Theorem~\ref{thm:main} in
both finite time blow-up and non-scattering scenarios from here on.
\begin{lemma} \label{lem:ini-seq} Suppose that Theorem~\ref{thm:main}
  fails for some initial data $(a,e,f,g)$ of energy $E$.  Then for
  every $\veps > 0$ there exists a sequence $\veps_{n} \to 0$
  and a sequence of admissible $C_{t} \calH^{1}$ solutions $(A^{(n)},
  \phi^{(n)})$ on $[\veps_{n}, 1] \times \bbR^{4}$ in the global
  Coulomb gauge that satisfy the following properties:
  \begin{enumerate}
  \item Bounded energy in the cone
    \begin{equation} \label{eq:ini-seq:energy} \calE_{S_{t}} [A^{(n)},
      \phi^{(n)}] \leq 2E \quad \hbox{ for every } t \in [\veps_{n},
      1],
    \end{equation}
  \item Small energy outside the cone
    \begin{equation} \label{eq:ini-seq:extr} \calE_{(\set{t} \times
        \bbR^{4}) \setminus S_{t}} [A^{(n)}, \phi^{(n)}] \leq
      \veps^{8} E \quad \hbox{ for every } t \in [\veps_{n}, 1],
    \end{equation}
  \item Decaying flux on $\rd C$
    \begin{equation} \label{eq:ini-seq:flux} \EFlux_{[\veps_{n}, 1]}
      [A^{(n)}, \phi^{(n)}] + \G_{S_{1}}[\phi^{(n)}] \leq
      \veps_{n}^{\frac{1}{2}} E,
    \end{equation}
  \item Pointwise concentration at $t = 1$
    \begin{equation} \label{eq:ini-seq:conc} 2^{-k_{n}}
      \abs{\zt_{2^{-k_{n}}} \ast \phi^{(n)}(1,x_{n})} + 2^{-2k_{n}}
      \abs{\zt_{2^{-k_{n}}} \ast \covD_{t}^{(n)} \phi^{(n)}(1, x_{n})}
      > \thcovED(E)
    \end{equation}
    for some $k_{n} \in \bbZ$ and $x_{n} \in \bbR^{4}$.
  \end{enumerate}

\end{lemma}

\begin{remark} 
  The small parameter $\veps > 0$ will be specified near the end of
  the proof of Theorem~\ref{thm:main}, precisely in
  Lemma~\ref{lem:final-rescale}, depending only on $E$.
\end{remark}
\begin{remark} \label{rem:ini-seq:g-dep-bnd} By the global Coulomb
  gauge condition $\rd^{\ell} A^{(n)}_{\ell} = 0$, the following gauge
  dependent uniform bounds for $A^{(n)}$ and $\phi^{(n)}$ hold:
  \begin{equation} \label{eq:ini-seq:g-dep-bnd} \nrm{\rd_{t,x}
      A^{(n)}}_{L^{\infty}_{t} ([\veps_{n}, 1]; L^{2}_{x})} \aleq
    E^{\frac{1}{2}}, \quad \nrm{\rd_{t,x} \phi^{(n)}}_{L^{\infty}_{t}
      ([\veps_{n}, 1]; L^{2}_{x})} \aleq
    (1+E^{\frac{1}{2}})E^{\frac{1}{2}}.
  \end{equation}
\end{remark}

\begin{proof} 
  Suppose that Theorem~\ref{thm:main} fails. Then by the discussion at
  the beginning of the section, there exists an admissible $C_{t}
  \calH^{1}$ solution $(A, \phi)$ of energy $E$ to \eqref{eq:MKG} on
  $[0, T_{+}) \times \bbR^{4}$ which satisfies either $0 < T_{+} <
  \infty$ and \eqref{eq:dich:blow-up} (finite time blow-up) or $T_{+}
  = \infty$ and \eqref{eq:dich:non-scatter} (non-scattering). We treat these two
  cases separately.
 
  \pfstep{Case 1: Finite time blow-up} By Lemma~\ref{lem:EC-continue},
  there exists a point $x_{0} \in \bbR^{4}$ such that
  \eqref{eq:EC-at-x0} holds with $T = T_{+}$.  By translation in
  space-time and reversing time, we may assume that $x_{0} = 0$ and we
  have energy concentration at the space-time origin as $t \to 0$,
  i.e.,
  \begin{equation} \label{eq:ini-seq:EC-at-0} \limsup_{t \to 0}
    \calE_{S_{t}}[A, \phi] > 0.
  \end{equation}

  Our next course of action is to use the excision and gluing
  technique (Theorem~\ref{thm:gluing}) to cut away the part of $(A,
  \phi)$ outside the cone of influence of $(0, 0)$.  In what follows,
  we denote the ball $B_{1}(0)$ by $B$, so that $rB = B_{r}(0)$ for
  any $r > 0$.

By Corollary~\ref{cor:flux-decay} there exists $t_{0} > 0$ such that
  \begin{equation*}
    \EFlux_{\rd C_{(0, t_{0}]}}[A, \phi] \ll \min \set{ \dlt_{0}(E, \thE^{2}), \veps^{8} E}
  \end{equation*}
  where $\dlt_{0}(E, \thE^{2})$ is as in \eqref{eq:EC:def}.  Furthermore, we
  can find a collar of radius $r_{0} > 0$ around $S_{t_{0}} =
  \set{t_{0}} \times t_{0} B$ with small energy, i.e.,
  \begin{equation*}
    \calE_{\set{t_{0}} \times ((t_{0} + r_{0}) B \setminus {t_{0} B})}[A, \phi] \ll 
\min \set{ \dlt_{0}(E, \thE^{2}), \veps^{8} E}.
  \end{equation*}
  By local conservation of energy, we then have
  \begin{equation*}
    \calE_{\set{t} \times ((t + r_{0}) B \setminus {t B})}[A, \phi] \ll  \min \set{ \dlt_{0}(E, \thE^{2}), \veps^{8} E} 
\quad  \hbox{ for every } t \in (0, t_{0}].
  \end{equation*}
Observe that the ratio $(t+r_0)/t$ goes to $\infty$ as $t \to 0$. 
Hence, by the improved Hardy estimate in Lemma~\ref{l:hardy+}, for sufficiently
small $0 < \bar t < r_0$ we also obtain
  \begin{equation*}
  \nrm{\frac{1}{\abs{x}} \phi(\bar{t}, \cdot)}_{L^{2}_{x} (2\bar t B \setminus \bar t B)}^{2}    \ll  \min \set{ \dlt_{0}(E, \thE^{2}), \veps^{8} E} 
\quad  \hbox{ for every } t \in (0, t_{0}].
  \end{equation*} 
 We may now apply Theorem~\ref{thm:gluing} to $(a, e, f, g) = (A_{j},
  F_{0j}, \phi, \covD_{t} \phi) \rst_{\set{t = \bar{t}}}$   to
  obtain a new data set $(\widetilde{a}, \widetilde{e}, \widetilde{f},
  \widetilde{g})$ that coincides with $(a, e, f, g)$ on $\bar{t}
  B$ and obeys
  \begin{equation*}
    \calE_{\bbR^{4} \setminus \bar{t} B}[\widetilde{a}, \widetilde{e}, \widetilde{f}, \widetilde{g}] 
    \leq  \frac{1}{2} \min \set{ \dlt_{0}(E, \thE^{2}), \veps^{8} E}.
  \end{equation*}
  
  To pass to the global Coulomb gauge, we define the gauge
  transformation $\underline{\chi} \in \calG^{2}(\bbR^{4})$ by
  $\underline{\chi} = \lap^{-1} \rd^{\ell} \widetilde{a}_{\ell}$ and
  let $(\check{a}, \check{e}, \check{f}, \check{g})$ be the gauge
  transform of $(\widetilde{a}, \widetilde{e}, \widetilde{f},
  \widetilde{g})$ by $\underline{\chi}$. Let $(\check{A},
  \check{\phi})$ be the admissible $C_{t} \calH^{1}$ solution to the
  Cauchy problem in the global Coulomb gauge given by
  Theorem~\ref{thm:lwp4MKG}, defined on the maximal time interval $I
  \ni \bar{t}$.

  As a consequence of the construction and local conservation of
  energy, the energy outside the cone $C$ is always tiny, i.e.,
  \begin{equation}
    \calE_{(\set{t} \times \bbR^{4}) \setminus S_{t}}[\check{A}, \check{\phi}] \leq \frac{1}{2} \min \set{\dlt_{0}(E, \thE^{2}), \veps^{8} E} \quad \hbox{ for every } t \in I.
  \end{equation}
  Then by an argument similar to the proof of
  Lemma~\ref{lem:EC-continue}, it follows that $(\check{A},
  \check{\phi})$ can be always continued to the past until $0$, i.e.,
  $(0, \bar{t}] \subseteq I$. Furthermore, there exist sequences
  $(t_{n}, x_{n}) \in I \times \bbR^{4}$ and $k_{n} \in \bbZ$ with
  $t_{n} \to 0$ such that
  \begin{equation}
    2^{-k_{n}} \abs{\zt_{2^{-k_{n}}} \ast \check{\phi}(t_{n}, x_{n})}
    + 2^{-2k_{n}} \abs{\zt_{2^{-k_{n}}} \ast \check{\covD}_{t} \check{\phi}(t_{n}, x_{n})} > \thcovED(E).
  \end{equation}
  For otherwise, there exists $\dlt > 0$ such that
  \eqref{eq:ED:small-covED} holds on $(0, \dlt)$. Then by
  Theorem~\ref{thm:ED} (with \eqref{eq:ED:small-EDC} replaced by
  \eqref{eq:ED:small-covED}) and Theorem~\ref{thm:finite-S}, the
  solution $(\check{A}, \check{\phi})$ can be extended past
  $t=0$. Hence $\limsup_{t \to 0} \calE_{S_{t}}[\check{A},
  \check{\phi}] = 0$, but this fact contradicts
  \eqref{eq:ini-seq:EC-at-0} as $\calE_{S_{t}}[\check{A},
  \check{\phi}] = \calE_{S_{t}}[A, \phi]$ for every $t \in I$.

  Applying Corollary~\ref{cor:flux-decay} to $(\check{A},
  \check{\phi})$, we may choose a sequence $\veps_{n} \to 0$ such that
  \begin{equation*}
    \EFlux_{[\veps_{n} t_{n}, t_{n}]} [A, \phi] + \G_{S_{t_{n}}}[\phi] \leq \veps_{n}^{\frac{1}{2}} E.
  \end{equation*}
  Then it follows that the sequence of rescaled solutions
  \begin{equation*}
    (A^{(n)}, \phi^{(n)})(t,x) := t_{n}^{-1} (\check{A}, \check{\phi})(t_{n}^{-1} t, t_{n}^{-1} x)
  \end{equation*}
  obeys the desired properties.

  \pfstep{Case 2: Non-scattering} This case follows by a simple
  rescaling argument. Let $R_{0} > 0$ be a large radius such that
  $\calE_{\set{0} \times (\bbR^{4} \setminus B_{R_{0}}(0))} [A, \phi]
  \leq \veps^{8} E$. Translating in time by $R_{0}$ and using the
  local conservation of energy, we may assume that $(A, \phi)$ obeys
  \begin{equation*}
    \calE_{(\set{t} \times \bbR^{4}) \setminus S_{t}} [A, \phi] \leq \veps^{8} E \quad \hbox{ for every } t \in [R_{0}, \infty).
  \end{equation*}
  By Theorem~\ref{thm:ED} with \eqref{eq:ED:small-EDC} replaced by
  \eqref{eq:ED:small-covED} and \eqref{eq:dich:non-scatter}, there
  exist sequences $(t_{n}, x_{n}) \in [R_{0}, \infty) \times \bbR^{4}$
  and $k_{n} \in \bbZ$ with $t_{n} \to \infty$ such that
  \begin{equation*}
    2^{-k_{n}} \abs{\zt_{2^{-k_{n}}} \ast \phi(t_{n}, x_{n})}
    + 2^{-2k_{n}} \abs{\zt_{2^{-k_{n}}} \ast \covD_{t} \phi(t_{n}, x_{n})} > \thcovED(E)
  \end{equation*}
  By Corollary~\ref{cor:flux-decay}, we may then choose a sequence
  $\veps_{n} \to 0$ such that $\veps_{n} t_{n} \to \infty$ and
  \begin{equation*}
    \EFlux_{[\veps_{n} t_{n}, t_{n}]} [A, \phi] + \G_{S_{t_{n}}}[\phi] \leq \veps_{n}^{\frac{1}{2}} E.
  \end{equation*}
  Defining $(A^{(n)}, \phi^{(n)})(t,x) := t_{n}^{-1} (A,
  \phi)(t_{n}^{-1} t, t_{n}^{-1} x)$, we obtain a desired sequence.
\end{proof}

\subsection{Elimination of the null concentration scenario} \label{subsec:no-null}
Using Proposition~\ref{prop:monotonicity}, in particular the weighted energy estimate on $S_{1}$, we show that \emph{null concentration} cannot happen. The precise statement is as follows.
\begin{lemma}[No null concentration] \label{lem:no-null}
Let $(A^{(n)}, \phi^{(n)})$ be a sequence of admissible $C_{t} \calH^{1}$ solutions to \eqref{eq:MKG} satisfying the conclusions of Lemma~\ref{lem:ini-seq} with the sequences $\veps_{n}$, $k_{n}$ and $x_{n}$. There exist $K = K(E) > 0$ and $\gmm = \gmm(E) \in (0, 1)$ such that 
if $k_{n} > K(E)$ and $\abs{x_{n}} > \gmm(E)$ for all sufficiently large $n$, and $\veps > 0$ is sufficiently small depending on $E$, then 
\begin{equation} \label{eq:no-null}
	\limsup_{n \to \infty} \, 
	2^{-k_{n}} \abs{\zt_{2^{-k_{n}}} \ast \phi(1, x_{n})}
	+ 2^{-2k_{n}} \abs{\zt_{2^{-k_{n}}} \ast \covD_{t}^{(n)} \phi^{(n)} (1,x_{n})} \leq \thcovED(E).
\end{equation}
\end{lemma}

\begin{remark}  \label{rem:no-null:K-gmm}
As $K(E)$ in Lemma~\ref{lem:no-null} can be replaced {\it a posteriori} by any number greater than $K(E)$. Hence given any $m = m(E)$ depending only on $E$, we may assume in addition to the statement of Lemma~\ref{lem:no-null} that
\begin{equation} \label{eq:no-null:K-gmm}
	2^{-K} \leq \frac{1}{100 \, m(E)} (1 - \gmm).
\end{equation}
This observation will be useful in the proof of Lemma~\ref{lem:t-like-e:t=1} below.
\end{remark}
\begin{proof} 
The idea of the proof is similar to that of \cite[Lemma~6.2]{MR2657818} with additional ideas to deal with the presence of covariant derivatives.

\pfstep{Step 1}
The starting point is Proposition~\ref{prop:monotonicity} applied to $(A, \phi) = (A^{(n)}, \phi^{(n)})$ with $\veps = \veps_{n}$, more precisely the first term on the left-hand side of \eqref{eq:monotonicity}. Using Lemma~\ref{lem:monotonicity} to write out $\mvC{X_{\veps_{n}}}_{T}$, we see that the following a-priori estimate holds on $S_{1}$:
\begin{equation} 
	\int_{S_{1}} \frac{1}{(1 - \abs{x} +\veps_{n})^{\frac{1}{2}}} \bb( \abs{\covD_{L}^{(n)} \phi^{(n)}}^{2} + \abs{\scovD^{(n)} \phi^{(n)}}^{2} \bb) \, \ud x \aleq E.
\end{equation}
By the smallness of the energy outside $S_{1}$, as well as conservation energy, we then obtain the global bound
\begin{equation} \label{eq:no-null:key-w-est}
	\int_{\set{t = 1}} \frac{1}{((1 - \abs{x})_{+} + \veps_{n})^{\frac{1}{2}} + \veps^{8}} \bb( \abs{\covD_{L}^{(n)} \phi^{(n)}}^{2} + \abs{\scovD^{(n)} \phi^{(n)}}^{2} \bb) \, \ud x \aleq E,
\end{equation}
where $(\cdot)_{+} := \max \set{\cdot, 0}$.

\pfstep{Step 2} We claim that for any $k \in \bbZ$ the following estimate holds:
\begin{equation} \label{eq:no-null:g-inv}
	\limsup_{n \to \infty} \, 2^{-k} \abs{\zt_{2^{-k}} \ast \abs{\phi^{(n)}}(1,x)}
	\aleq \bb( 2^{-\frac{3}{8} k} + \big( (1-\abs{x})_{+} + 2^{-k} \big)^{\frac{1}{4}} + \veps^{4} \bb) E^{\frac{1}{2}}.
\end{equation}
The point of \eqref{eq:no-null:g-inv} is that $\abs{\phi^{(n)}}$ is \emph{gauge invariant}, and hence we can avoid estimating $A$. 
Henceforth, we will denote $\psi^{(n)} := \abs{\phi^{(n)}}(1, \cdot)$. We use the rotational symmetry to bring $x$ to the $x^{1}$-axis, so that $x = (\abs{x}, 0, 0, 0)$. Henceforth we will write $x = (x^{1}, x')$ where $x' = (x^{2}, x^{3}, x^{4})$.

By the diamagnetic inequality (Lemma~\ref{lem:diamagnetic}), conservation of energy implies
\begin{equation} \label{eq:no-null:energy:g-inv}
	\int \abs{\nb \psi^{(n)}}^{2} \, \ud x \aleq E.
\end{equation}
where $\abs{\nb \psi}^{2} := \sum_{\ell=1}^{4} \abs{\rd_{\ell} \psi}^{2}$. Note that \eqref{eq:no-null:energy:g-inv} and Young's inequality implies the trivial bound $2^{-k} \nrm{\zt_{2^{-k}} \ast \psi^{(n)}}_{L^{\infty}_{x}} \aleq E^{1/2}$, which allows us to restrict our attention to $x = (x^{1}, 0, 0, 0)$ with $1/2 < x^{1} < 2$.

We claim that for $n$ sufficiently large so that $\veps_{n}^{1/2} \leq \frac{1}{100} \veps^{8}$, the directional derivatives other than $\rd_{1}$ obey an improved estimate
\begin{equation} \label{eq:no-null:key-w-est:g-inv}
	\sum_{j=1}^{4} \int w_{k} \abs{\rd_{j} \psi^{(n)}}^{2} \, \ud x \aleq E,
\end{equation}
where $w_{k} > 0$ is defined as
\begin{equation} \label{eq:no-null:key-w}
	w_{k}(x) := \frac{1}{(\abs{1-x^{1}} + \abs{x'}^{2} + 2^{k})^{\frac{1}{2}} + \veps^{8}}.
\end{equation}

The estimate \eqref{eq:no-null:key-w-est:g-inv} is a consequence of \eqref{eq:no-null:key-w-est}. Indeed, the latter estimate combined with the diamagnetic inequality implies
\begin{equation} \label{eq:no-null:key-w-est:g-inv:pre}
	\int \frac{1}{((1 - \abs{x})_{+} + \veps_{n})^{\frac{1}{2}} + \veps^{8}} \abs{\snb \psi^{(n)}}^{2}  \, \ud x \aleq E.
\end{equation}
At $x = (1, 0, 0, 0)$ we have $\frac{1}{r^{2}} g_{\bbS^{3}}^{-1} = \sum_{j=2}^{4} \rd_{j} \cdot \rd_{j}$. Therefore, by smoothness, we have
\begin{equation*}
	\abs{\abs{\snb \psi}^{2} - \sum_{j=2}^{4}\abs{\rd_{j} \psi}^{2}} \aleq (\abs{1 - x^{1}} + \abs{x'}) \abs{\nb \psi}^{2}.
\end{equation*}
On the other hand, $(1-\abs{x})_{+} \aleq \abs{1 - x^{1}} + \abs{x'}^{2}$.
Therefore, combined with \eqref{eq:no-null:energy:g-inv} (to control $\nb \psi$ in the error), \eqref{eq:no-null:key-w-est:g-inv:pre} implies
\begin{equation*} 
	\sum_{j=2}^{4} \int \frac{1}{(\abs{1 - x^{1}} + \abs{x'}^{2} + \veps_{n})^{\frac{1}{2}} + \veps^{8}} \abs{\rd_{j} \psi^{(n)}}^{2} \, \ud x \aleq E.
\end{equation*}
Then under the assumption that $\veps_{n} \leq \frac{1}{100} \veps^{8}$, the desired estimate \eqref{eq:no-null:key-w-est:g-inv} follows.

Observe that we have put in an extra $2^{k}$ in the weight $w_{k}$. This maneuver ensures that $w$ is \emph{slowly varying} at scale $2^{k} \times 2^{k/2} \times \cdots \times 2^{k/2}$, i.e., for any $x, y \in \bbR^{4}$ we have
\begin{equation} \label{eq:no-null:w-slow-vary}
	\abs{\frac{w_{k}(x)}{w_{k}(x-y)}}
	\aleq e^{\sum_{j=1}^{4} \abs{y^{j}} \nrm{\rd_{j} \log w}_{L^{\infty}}}
	\aleq e^{2^{k} \abs{y^{1}} + 2^{k/2} \abs{y'}}.
\end{equation}

We now turn to the task of deriving \eqref{eq:no-null:g-inv} from \eqref{eq:no-null:energy:g-inv} and \eqref{eq:no-null:key-w-est:g-inv}. We introduce the notation $Z_{k} \psi := \zt_{2^{-k}} \ast \psi$ and write $z_{k}(\xi)$ for the symbol of the integral operator $Z_{k}$; of course, $z_{k}$ is nothing but the Fourier transform of $\zt_{2^{-k}}$. We furthermore decompose
\begin{equation*}
	Z_{k} = Z_{k}^{1} \rd_{1} + Z_{k}^{2} \rd_{2} + \cdots + Z_{k}^{4} \rd_{4}
\end{equation*}
where the symbols $z_{k}^{j}(\xi)$ of $Z_{k}^{j}$ are given by 
\begin{align*}
	z_{k}^{1} (\xi) =& z_{k}(\xi) \eta(2^{-\frac{k}{2}} \xi') \frac{1}{i \xi_{1}},  \\
	z_{k}^{j} (\xi) =& z_{k}(\xi) (1 - \eta(2^{-\frac{k}{2}} \xi')) \frac{\xi_{j}}{i \abs{\xi'}^{2}} \quad \hbox{ for } j = 2, 3, 4.
\end{align*}

The contribution of $Z_{k}^{1} \rd_{1}$ to \eqref{eq:no-null:g-inv} is easy to treat. Observe that $z_{k}^{1} (\xi) i \xi_{1}$ is a smooth symbol which is rapidly decaying at scale $2^{k}$ in the $\xi_{1}$-direction and compactly supported in the set $\set{\abs{\xi'} \aleq 2^{k/2}}$ in the other directions. By Bernstein's inequality, we have
\begin{equation*}
2^{-k} \abs{Z_{k}^{1} \rd_{1} \psi^{(n)}(x)} \aleq 2^{-\frac{3}{8} k} \nrm{\psi^{(n)}}_{\dot{H}^{1}_{x}} \aleq 2^{-\frac{3}{8} k} E^{\frac{1}{2}},
\end{equation*}
which is acceptable. 

It remains to treat the contribution of $Z_{k}^{j} \rd_{j}$ for $j = 2,3, 4$.
Denote by $\zt_{k}^{j}(x)$ the integral kernel of $Z_{k}^{j}$, which is simply the inverse Fourier transform of $z_{k}^{j}$. A straightforward computation shows that $\nrm{z^{j}_{k}}_{L^{2}_{\xi}} \aleq 2^{k}$. Therefore, by Plancherel,
\begin{equation} \label{eq:no-null:kernel-size}
	\nrm{\zt^{j}_{k}}_{L^{2}_{x}} \aleq 2^{k}.
\end{equation}
Moreover, for any $N \geq 1$, it is not difficult to see that
\begin{equation} \label{eq:no-null:kernel-decay}
	\abs{\sum_{j=2}^{4} \rd_{j} \zt_{k}^{j}(x)} \aleq_{N} \big( 2^{\frac{3}{2} k} + (2^{\frac{k}{2}}\abs{x'})^{-3} \big) (1 + 2^{k} \abs{x^{1}})^{-N} 2^{\frac{5}{2} k}
\end{equation}
where the implicit constant is independent of $k$. 
%
%
Hence we can split $\zt_{k}^{j} = \zt_{k, \near}^{j} + \zt_{k, \far}^{j}$, where
\begin{equation*}
\zt_{k, \near}^{j}(x) := \zt_{k}^{j} (x) 1_{\set{x : \abs{x^{1}} \leq L 2^{k}, \ \abs{x'} \leq L 2^{k/2}}}(x),
\end{equation*}
and $L > 0$ is chosen large enough (independent of $k$) so that, by \eqref{eq:no-null:kernel-decay}, we have
\begin{equation} \label{eq:no-null:Z-far}
	\nrm{\sum_{j=2}^{4} \rd_{j} \zt_{k, \far}^{j}}_{L^{\frac{4}{3}}_{x}} \leq 2^{k} \veps^{4}.
\end{equation}
We denote the corresponding splitting of $Z_{k}^{j}$ by $Z_{k, \near}^{j} + Z_{k, \far}^{j}$. 

We are now ready to complete the proof of \eqref{eq:no-null:g-inv}. The contribution of $Z_{k, \far}^{j} \rd_{j}$ is acceptable, thanks to \eqref{eq:no-null:energy:g-inv}, \eqref{eq:no-null:Z-far} and the Sobolev embedding $\dot{H}^{1}_{x} \subseteq L^{4}_{x}$.
For $\sum_{j=2}^{4} Z_{k, \near}^{j} \rd_{j}$, we have
\begin{align*}
	2^{-k} \abs{\sum_{j=2}^{4} Z_{k, \near}^{j} \rd_{j} \psi^{(n)}(x)}
	\leq & 2^{-k} \sum_{j=2}^{4} \int \abs{\zt^{j}_{k, \near}(y)} \abs{\rd_{j} \psi^{(n)}(x-y)} \, \ud y \\
	\aleq & M w^{-\frac{1}{2}}(x) \nrm{w^{\frac{1}{2}} \rd_{j} \psi^{(n)}}_{L^{2}_{x}}
\end{align*}
where, by \eqref{eq:no-null:w-slow-vary}, \eqref{eq:no-null:kernel-size} and the definition of $\zt_{k, \near}^{j}$, $M$ obeys the bound
\begin{align*}
M := & \  \bb( 2^{-2k}\sum_{j=2}^{4} \int \frac{w(x)}{w(x-y)} \abs{\zt_{k, \near}^{j}}^{2}(y) \, \ud y \bb)^{\frac{1}{2}}
\\ 
\aleq_{L} & \ \bb( 2^{-2k} \sum_{j=2}^{4} \int_{\set{\abs{y^{1}} \leq L 2^{k}, \ \abs{y'} \leq L 2^{k/2}}} \abs{\zt_{k}^{j}}^{2} \, \ud y \bb)^{\frac{1}{2}} \aleq 1,
\end{align*}
which proves \eqref{eq:no-null:g-inv}.

\pfstep{Step 2}
In this step we upgrade \eqref{eq:no-null:g-inv} to the following \emph{gauge dependent} estimate: 
\begin{equation} \label{eq:no-null:g-dep-sp}
	\limsup_{n \to \infty} 2^{-2k} \abs{\zt_{2^{-k}} \ast \covD_{j} \phi^{(n)}(1, x)} \aleq \bb( 2^{-\frac{3}{8} k} + \big( (1-\abs{x})_{+} + 2^{-k} \big)^{\frac{1}{4}} + \veps^{4} \bb) E^{\frac{1}{2}}.
\end{equation}
The idea is that \eqref{eq:no-null:g-inv} has already broken the scaling invariance, so we can easily incorporate $A$ using the trivial bound $\nrm{A}_{L^{\infty}_{t} \dot{H}^{1}_{x}} \aleq E^{1/2}$. 

We begin by applying Step 1 to $\widetilde{\zt}_{2^{-k}}$, where we recall that $\zt = \widetilde{\zt} \ast \widetilde{\zt}$. 
We again introduce the shorthand $\widetilde{Z}_{k} (\cdot):= \widetilde{\zt}_{2^{-k}} \ast (\cdot)$.
By the simple pointwise inequality $\abs{\widetilde{Z}_{k} \phi^{(n)}} \leq \abs{\widetilde{Z}_{k} \abs{\phi^{(n)}}}$, which holds since $\widetilde{\zt} \geq 0$, we have
\begin{equation} \label{eq:no-null:g-inv:2}
	\limsup_{n \to \infty} \, 2^{-k} \abs{\widetilde{Z}_{k} \phi^{(n)}(1,x)}
	\aleq \bb( 2^{-\frac{3}{8} k} + \big( (1-\abs{x})_{+} + 2^{-k} \big)^{\frac{1}{4}} + \veps^{4} \bb) E^{\frac{1}{2}}.
\end{equation}
Note furthermore that $Z_{k} = \widetilde{Z}_{k}^{2}$. For $j= 1, \ldots, 4$, we may write
\begin{align*}
	2^{-2k} \abs{Z_{k} \covD_{j}^{(n)} \phi^{(n)}(1,x)}
	& \leq 2^{-2k} \abs{Z_{k} \rd_{j} \phi^{(n)}(1,x)}
		+2^{-2k} \abs{Z_{k} (A_{j}^{(n)} \phi^{(n)})(1,x)} \\
	& \aleq 2^{-k} \sup_{\abs{x - x'} \aleq 2^{-k}} \abs{\widetilde{Z}_{k} \phi^{(n)}(1,x')} 
			+ 2^{-2k} \abs{Z_{k} (A_{j}^{(n)} \phi^{(n)})(1,x)}.
\end{align*}
The first term on the last line is acceptable, thanks to \eqref{eq:no-null:g-inv:2}. To treat the second term, we insert $1 = (1 - \widetilde{Z}_{2^{-k+m}}) + \widetilde{Z}_{k+m}$ in front of $A^{(n)}, \phi^{(n)}$ for some $m > 0$ to be determined. By the simple inequality $\nrm{(1 - \widetilde{Z}_{k+m}) f}_{L^{2}_{x}} \aleq 2^{-k-m} \nrm{f}_{\dot{H}^{1}_{x}}$, each term involving $1 - \widetilde{Z}_{2^{-k+m}}$ is bounded by
\begin{equation*}
	\aleq 2^{-m} \nrm{A^{(n)}(1, \cdot)}_{\dot{H}^{1}_{x}} \nrm{\phi^{(n)}(1, \cdot)}_{\dot{H}^{1}_{x}},
\end{equation*}
which can be made $\leq \veps^{4} E^{\frac{1}{2}}$ by choosing $m$ large enough. For the remaining term, we have
\begin{align*}
	2^{-2k} \abs{Z_{k} (\widetilde{Z}_{k+m} A_{j}^{(n)} \widetilde{Z}_{k+m} \phi^{(n)})(1,x)}
	& \aleq  2^{-k} \nrm{\widetilde{Z}_{k+m} A_{j}^{(n)}}_{L^{\infty}_{x}} 2^{-k} \sup_{\abs{x - x'} \aleq 2^{-k}} \abs{\widetilde{Z}_{k+m} \phi^{(n)}(1,x')} \\
	& \aleq_{E, m} 2^{-k+m} \sup_{\abs{x - x'} \aleq 2^{-k}} \abs{\widetilde{Z}_{k+m} \phi^{(n)}(1,x')}
\end{align*}
which is acceptable in view of \eqref{eq:no-null:g-inv:2}.

\pfstep{Step 3}
We are ready to conclude the proof of the lemma. By \eqref{eq:no-null:g-inv} and the pointwise inequality $\abs{\zt_{2^{-k}} \ast \phi} \leq \zt_{2^{-k}} \ast \abs{\phi}$, we can achieve the desired smallness as in \eqref{eq:no-null} of $\phi^{(n)}$ by taking $K$ very large, $\gmm$ close enough to $1$ and $\veps > 0$ sufficiently small. For $\covD^{(n)}_{t}$, we have
\begin{equation*} 
	2^{-2k} \abs{\zt_{2^{-k}} \ast \covD_{t}^{(n)} \phi^{(n)}(1, x)} \leq
	2^{-2k} \abs{\zt_{2^{-k}} \ast \covD_{L}^{(n)} \phi^{(n)}(1, x)}
	+ \sum_{j=1} 2^{-2k} \abs{\zt_{2^{-k}} \ast \covD_{j}^{(n)} \phi^{(n)}(1, x)}.
\end{equation*}
Using \eqref{eq:no-null:key-w-est} for the first term (also exploiting the fact that $\zt_{2^{-k}}$ is supported in a ball of radius $\aleq 2^{-k}$) and \eqref{eq:no-null:g-dep-sp} for the second term, \eqref{eq:no-null} now follows after adjusting $K$, $\gmm$ and $\veps$ if necessary. \qedhere

\end{proof}

\subsection{Nontrivial energy in a time-like region} \label{subsec:t-like-e}
An important consequence of Lemma~\ref{lem:no-null} is that there is a uniform lower bound for $\phi^{(n)}$ in a time-like region at $t = 1$.
\begin{lemma} \label{lem:t-like-e:t=1}
Let $(A^{(n)}, \phi^{(n)})$ be an admissible $C_{t} \calH^{1}$ solution to \eqref{eq:MKG} satisfying \eqref{eq:ini-seq:conc}. Let $K(E) > 0$ and $\gmm(E) \in (0, 1)$ be as in Lemma~\ref{lem:no-null} and Remark~\ref{rem:no-null:K-gmm}. Assume that either $(1)$ $k_{n} \leq K(E)$ or $(2)$ $k_{n} > K(E)$ and $\abs{x_{n}} \leq 1 - \gmm(E)$. Then there exist $E_{1} = E_{1}(E) > 0$ and $\gmm_{1} = \gmm_{1} (E) \in (0, 1)$ such that if $\veps > 0$ is sufficiently small depending on $E$, then
\begin{equation} \label{eq:t-like-e:t=1}
	\int_{S^{1-\gmm_{1}}_{1}} \sum_{\mu=0}^{4} \abs{\covD_{\mu}^{(n)} \phi^{(n)}}^{2} + \frac{1}{r^{2}} \abs{\phi^{(n)}}^{2} \, \ud x \geq E_{1}(E).
\end{equation}
\end{lemma}
\begin{proof} 
Since the whole proof will take place on $\set{t = 1}$, we will ignore the difference between $\set{t =1}$ and $\bbR^{4}$. Furthermore, as the argument is the same for each $n$, we will henceforth suppress $n$ for simplicity. 
There are two scenarios to consider: 
\begin{itemize}
\item[A.] {\bf Nontrivial kinetic energy.} $2^{-2k} \abs{\zt_{2^{-k}} \ast \covD_{t} \phi(x)} \geq \frac{1}{2} \thcovED(E)$, or 
\item[B.] {\bf Nontrivial potential energy.} $2^{-k} \abs{\zt_{2^{-k}} \ast \phi(x)} \geq \frac{1}{2} \thcovED(E)$.
\end{itemize}

We first treat Scenario A. By Cauchy-Schwarz,
\begin{align*}
	\frac{1}{2} \thcovED 
	\leq  \int 2^{-2k} \zt_{2^{-k}}(y) \abs{\covD_{t} \phi(x-y)} \, \ud y 
	\aleq  \bb( \int_{B_{2^{-k}}(x)} \abs{\covD_{t} \phi}^{2} \, \ud y \bb)^{1/2},
\end{align*}
where we also used $\supp \, \zt \subseteq B_{1}(0)$. Hence in Case 2, \eqref{eq:t-like-e:t=1} immediately follows by taking $\gmm_{1} \geq \gmm + 2^{-k}$ so that $B_{2^{-k}}(x) \subseteq S^{1-\gmm_{1}}_{1}$. Note that we may still ensure that $\gmm_{1} < 1$ thanks to \eqref{eq:no-null:K-gmm}. 

Now assume that Case 1 holds, i.e., $k \leq K$. Splitting the convolution integral into $\int_{S_{1}^{1-\gmm_{1}}} + \int_{S_{1} \setminus S^{1-\gmm_{1}}_{1}} + \int_{\bbR^{4} \setminus S_{1}}$, applying Cauchy-Schwarz and using \eqref{eq:ini-seq:energy}, \eqref{eq:ini-seq:extr}, we have
\begin{align*}
	\thcovED
	\aleq \bb( \int_{S_{1}^{1-\gmm_{1}}} \abs{\covD_{t} \phi}^{2} \, \ud y \bb)^{1/2} + c_{0}(\gmm_{1}) E^{1/2} + \veps^{8} E^{1/2},
\end{align*}
where
\begin{equation*}
	c_{0}(\gmm_{1}) 
	:= \bb( \int_{S_{1} \setminus S_{1}^{1-\gmm_{1}}} \abs{\zt(2^{-k} y)}^{2} 2^{-4k} \, \ud y \bb)^{1/2}
	\aleq 2^{-2k} \abs{(S_{1} \setminus S_{1}^{1-\gmm_{1}}) \cap B_{2^{-k}}(x)}^{1/2}.
\end{equation*}
By elementary geometry and the assumption $k \leq K$, it follows that the last term is bounded by $\aleq (1-\gmm_{1})^{1/2} 2^{-K/2}$ uniformly in $x$. Taking $\gmm_{1}$ sufficiently close to $1$, the desired conclusion follows. 

We now consider Scenario B. We repeat the above argument with $\covD_{t} \phi$ replaced by $\phi$, while putting $\zt_{2^{-k}}$ [resp. $\phi$] in $L^{4/3}$ [resp. $L^{4}$] instead of $L^{2}$ [resp. $L^{2}$]. Then in Case 1, 
\begin{equation}
	\thcovED \aleq \bb( \int_{B_{2^{-k}}(x)} \abs{\phi}^{4} \, \ud y \bb)^{1/4},
\end{equation}
whereas in Case 2, 
\begin{equation}
	\thcovED \aleq \bb( \int_{S_{1}^{1-\gmm_{1}}} \abs{\phi}^{4} \, \ud y \bb)^{\frac{1}{4}} + c_{1}(\gmm_{1}) \nrm{\phi}_{L^{4}_{x}(\bbR^{4})} + \nrm{\phi}_{L^{4}_{x}(\bbR^{4} \setminus S_{1})},
\end{equation}
with $c_{2}(\gmm_{1}) \aleq (1-\gmm_{1})^{3/4} 2^{-3K/4}$. The desired conclusion then follows from \eqref{eq:ini-seq:energy}, \eqref{eq:ini-seq:extr}, the diamagnetic inequality (Lemma~\ref{lem:diamagnetic}) and the \emph{localized Sobolev inequalities}
\begin{align*}
	\nrm{f}_{L^{4}_{x}(B_{r}(0))} \aleq & \bb( \sum_{j=1}^{4} \nrm{\rd_{j} f}_{L^{2}_{x}(B_{r}(0))}^{2} \bb)^{1/2} + \nrm{\frac{1}{r} f}_{L^{2}(B_{r}(0))}, \\
	\nrm{f}_{L^{4}_{x}(\bbR^{4} \setminus B_{r}(0))} \aleq & \bb( \sum_{j=1}^{4} \nrm{\rd_{j} f}_{L^{2}_{x}(\bbR^{4} \setminus B_{r}(0))}^{2} \bb)^{1/2},
\end{align*}
which hold with a uniform constant for any $\frac{1}{2} < r < 1$. Both inequalities follow from the usual Sobolev inequality on $\bbR^{4}$ by extending $f$ to $\bbR^{4}$. We remark that the norm $\nrm{r^{-1} \phi}_{L^{2}_{x}}$ is not needed for the second inequality, since we can use localized Hardy's inequality as in Corollary~\ref{cor:Hardy:const-t}. \qedhere
\end{proof}

The uniform lower bound in a time-like region can be propagated towards $t = 0$ using the localized monotonicity formula in Proposition~\ref{prop:monotonicity:t-like}.
\begin{lemma} \label{lem:t-like-e}
Let $(A^{(n)}, \phi^{(n)})$ be a sequence of admissible $C_{t} \calH^{1}$ solutions to \eqref{eq:MKG} satisfying the conclusions of Lemma~\ref{lem:ini-seq}. Assume furthermore that each $(A^{(n)}, \phi^{(n)})$ obeys \eqref{eq:t-like-e:t=1}. Then there exist $E_{2} = E_{2}(E) > 0$ and $\gmm_{2} = \gmm_{2}(E) \in (0, 1)$ such that
\begin{equation} \label{eq:t-like-e}
	\int_{S^{(1-\gmm_{2}) t}_{t}} \mvC{X_{0}}_{T}[A^{(n)}, \phi^{(n)}]  \, \ud x \geq E_{2}(E) \quad \hbox{ for every } t \in [\veps_{n}^{\frac{1}{2}}, \veps_{n}^{\frac{1}{4}}].
\end{equation}
\end{lemma}

\begin{proof} 
Fix $n$ and $t_{0} \in [\veps_{n}^{1/2}, \veps_{n}^{1/4}]$. Applying Proposition~\ref{prop:monotonicity:t-like} with $\veps = \veps_{n}$, $\dlt_{0} = (1-\gmm_{2}) t_{0}$ and $\dlt_{1} = M \dlt_{0}$, where $\gmm_{2} \in (0, 1)$ and $M > 1$ will be chosen below, we obtain
\begin{equation} \label{eq:t-like-e:pf:1}
	\int_{S_{1}^{M (1-\gmm_{2}) t_{0}}} \mvC{X_{0}}_{T}[A, \phi] \, \ud x
	\leq \int_{S_{t}^{(1-\gmm_{2})t}} \mvC{X_{0}}_{T}[A, \phi] \, \ud x + C\bb( (M (1-\gmm_{2}))^{\frac{1}{2}}+ \abs{\log M}^{-1} \bb) E.
\end{equation}
On the other hand, by Lemma~\ref{lem:monotonicity} (in particular, the expression for $\mvC{X_{0}}_{T} = \frac{1}{2} (\mvC{X_{0}}_{L} + \mvC{X_{0}}_{\uL})$ and \eqref{eq:t-like-e:t=1}, we have
\begin{equation*}
	E_{1} \aleq (1-\gmm_{1})^{-\frac{1}{2}} \int_{S_{1}^{\dlt t}} \mvC{X_{0}}_{T}[A, \phi] \, \ud x.
\end{equation*}
Hence choosing $M$ sufficiently large and $\gmm_{2}$ close enough to $1$ to make the last term in \eqref{eq:t-like-e:pf:1} small, \eqref{eq:t-like-e} follows with $E_{2} = c E_{1} (1-\gmm_{1})^{\frac{1}{2}}$ for some $c > 0$. \qedhere 
\end{proof}

\subsection{Final rescaling} \label{subsec:final-rescale}
So far, under the assumption that Theorem~\ref{thm:main} fails, we have shown the existence of a sequence of solutions $(A^{(n)}, \phi^{(n)})$ that satisfies the conclusions of Lemma~\ref{lem:ini-seq} and a uniform lower bound \eqref{eq:t-like-e} in a time-like region. By Proposition~\ref{prop:monotonicity}, the sequence moreover obeys the uniform space-time bound
\begin{equation} \label{eq:asymp-ss:pre}
	\iint_{C_{[\veps_{n}, 1]}} \frac{1}{\rho_{\veps_{n}}} \abs{\iota_{X_{\veps_{n}}} F^{(n)}}^{2} 
			+ \frac{1}{\rho_{\veps_{n}}} \abs{(\covD^{(n)}_{X_{\veps_{n}}} + \frac{1}{\rho_{\veps_{n}}}) \phi^{(n)}}^{2}
			\, \ud t \ud x
	\aleq E.
\end{equation}
Our next goal is to upgrade \eqref{eq:asymp-ss:pre} to asymptotic self-similarity by a rescaling argument.
\begin{lemma} \label{lem:final-rescale}
Suppose that Theorem~\ref{thm:main} fails.
Then there exists a sequence of admissible $C_{t} \calH^{1}$ solutions $(A^{(n)}, \phi^{(n)})$ on $[1, T_{n}] \times \bbR^{4}$ with $T_{n} \to \infty$ satisfying the following properties:
\begin{enumerate}
\item Bounded energy in the cone
\begin{equation} \label{eq:final-rescale:energy}
	\calE_{S_{t}}[A^{(n)}, \phi^{(n)}] \leq E, \quad
	 \quad \hbox{ for every } t \in [1, T_{n}],
\end{equation}
\item Small energy outside the cone
\begin{equation} \label{eq:final-rescale:extr}
	\calE_{\set{t} \times \bbR^{4} \setminus S_{t}}[A^{(n)}, \phi^{(n)}] \leq \frac{1}{100} E \quad
	 \quad \hbox{ for every } t \in [1, T_{n}],
\end{equation}

\item Nontrivial energy in a time-like region
\begin{equation} \label{eq:final-rescale:nontrivial}
	\int_{S^{(1-\gmm_{2})t}_{t}} \mvC{X_{0}}_{T}[A^{(n)}, \phi^{(n)}] \, \ud x \geq E_{2} \quad \hbox{ for every } t \in [1, T_{n}],
\end{equation}
\item Asymptotic self-similarity
\begin{equation} \label{eq:final-rescale:asymp-ss}
	\iint_{K} \abs{\iota_{X_{0}} F^{(n)}}^{2} + \abs{(\covD^{(n)}_{X_{0}} + \frac{1}{\rho}) \phi^{(n)}}^{2} \, \ud t \ud x \to 0 \quad \hbox{ as } n \to \infty
\end{equation}
	for every compact subset $K$ of the interior of $C_{[1, \infty)}$.
\end{enumerate}
\end{lemma}
\begin{proof} 
Let $(A^{(n)}, \phi^{(n)})$ be a sequence of solutions satisfying the conclusions of Lemmas~\ref{lem:ini-seq} and \ref{lem:t-like-e}.
Consider the time interval $[\veps_{n}^{1/2}, \veps_{n}^{1/4}]$, on which \eqref{eq:t-like-e} applies. Given $T_{n} > 1$, we partition $\veps_{n}$ in to dyadic intervals of the form $I_{n}^{j} = [T_{n}^{j} \veps_{n}^{1/2}, T_{n}^{j+1} \veps_{n}^{1/2}]$; there are roughly $\abs{\log \veps_{n}}/ \log T_{n}$ many such intervals. We choose $T_{n}$ so that $\log T_{n} \aeq \abs{\log \veps_{n}}^{1/2}$. Observe that $T_{n} \to \infty$. Also, by the pigeonhole principle applied to \eqref{eq:asymp-ss:pre}, there exists $j(n)$ such that
\begin{equation} \label{eq:asymp-ss:pre-pigeonhole}
	\iint_{C_{I_{n}^{j(n)}}} \frac{1}{\rho_{\veps_{n}}} \abs{\iota_{X_{\veps_{n}}} F^{(n)}}^{2} 
			+ \frac{1}{\rho_{\veps_{n}}} \abs{(\covD^{(n)}_{X_{\veps_{n}}} + \frac{1}{\rho_{\veps_{n}}}) \phi^{(n)}}^{2}
			\, \ud t \ud x
	\aleq \frac{\log T_{n}}{\abs{\log \veps_{n}}} E \aeq \frac{1}{\abs{\log \veps_{n}}^{1/2}} E,
\end{equation}
which exhibits the desired decay as $n \to \infty$. 

We now rescale $C_{I_{n}^{j(n)}}$ to $C_{[1, T_{n}]}$; abusing the notation a bit (but conforming to the statement of the lemma), we denote the rescaled solutions again by $(A^{(n)}, \phi^{(n)})$. From \eqref{eq:ini-seq:energy} and \eqref{eq:ini-seq:extr} with $\veps^{8} \leq \frac{1}{100}$, \eqref{eq:final-rescale:energy} and \eqref{eq:final-rescale:extr} follow. Also, \eqref{eq:final-rescale:nontrivial} is a consequence of \eqref{eq:t-like-e}. 
Furthermore, \eqref{eq:asymp-ss:pre-pigeonhole} implies
\begin{equation} \label{eq:asymp-ss:pre-rescaled}
	\iint_{C_{[1, T_{n}]} } \frac{1}{\rho_{\veps'_{n}}} \abs{\iota_{X_{\veps'_{n}}} F^{(n)}}^{2} 
			+ \frac{1}{\rho_{\veps'_{n}}} \abs{(\covD^{(n)}_{X_{\veps'_{n}}} + \frac{1}{\rho_{\veps'_{n}}}) \phi^{(n)}}^{2}
			\, \ud t \ud x
	\to 0 \quad \hbox{ as } n \to \infty
\end{equation}
where $\veps'_{n} := (T_{n}^{j(n)} \veps_{n}^{1/2})^{-1} \veps_{n}$ obeys $\veps'_{n} \leq \veps_{n}^{1/2} \to 0$. For any compact subset $K$ of the interior of $C_{[1, \infty)}$, which is in particular situated away from the boundary $\rd C_{[1, \infty)}$, note that
\begin{equation*}
	\iint_{K} \bb( \frac{1}{\rho_{\veps'_{n}}} \abs{\iota_{X_{\veps'_{n}}} F^{(n)}}^{2} 
	- \frac{1}{\rho} \abs{\iota_{X_{0}} F^{(n)}}^{2} \bb) 
	+ 	\bb( \frac{1}{\rho_{\veps'_{n}}} \abs{(\covD^{(n)}_{X_{\veps'_{n}}} - \frac{1}{\rho_{\veps'_{n}}}) \phi^{(n)}}^{2} 
	- \frac{1}{\rho} \abs{(\covD^{(n)}_{X_{0}} + \frac{1}{\rho}) \phi^{(n)}}^{2} \bb)  \, \ud t \ud x \to 0 
\end{equation*}
by conservation of energy, localized Hardy's inequality and the dominated convergence theorem. Combined with \eqref{eq:asymp-ss:pre-rescaled}, the desired asymptotic self-similarity \eqref{eq:final-rescale:asymp-ss} follows. \qedhere
\end{proof}

\subsection{Concentration scales} \label{subsec:conc-scales}
Let $(A^{(n)}, \phi^{(n)})$ be a sequence of solutions given by Lemma~\ref{lem:final-rescale}. We now present a combinatorial result that establishes the following dichotomy: Either there is a uniform non-concentration of energy, or we can identify a sequence of points and decreasing scales at which energy concentrates.

To state the result, we need few definitions. For each $j = 1, 2, \cdots$ we define
\begin{align*}
	C_{j} :=& \set{(t,x) \in C^{1}_{[1, \infty)} : 2^{j} \leq t < 2^{j+1}}, \\
	\widetilde{C}_{j} :=& \set{(t,x) \in C^{1/2}_{[1/2, \infty)} : 2^{j} \leq t < 2^{j+1}}.
\end{align*}
In words, $C_{j}$ [resp. $\widetilde{C}_{j}$] is the set of points in the truncated cone $C_{[2^{j}, 2^{j+1})}$ at distance $\geq 1$ [resp. $\geq 1/2$] from the lateral boundary. For each $j \geq 1$, we have the following lemma.
\begin{lemma} \label{lem:conc-scales}
Let $(A^{(n)}, \phi^{(n)})$ be a sequence of admissible $C_{t} \calH^{1}$ solutions on $[1, T_{n}] \times \bbR^{4}$ with $T_{n} \to \infty$ satisfying \eqref{eq:final-rescale:energy}--\eqref{eq:final-rescale:asymp-ss} for some $E > 0$. Let $\eps_{0}$  be as in Proposition~\ref{prop:cpt}. Then for each $j = 1, 2, \cdots$, after passing to a subsequence, one of the following alternatives holds:
\begin{enumerate}
\item {\bf Concentration of energy}. There exist points $(t_{n}, x_{n}) \in \widetilde{C}_{j}$, scales $r_{n} \to 0$ and $0 < r = r(j) < 1/4$ such that the following bounds hold:
\begin{align} 
	\calE_{\set{t_{n}} \times B_{r_{n}}(x_{n})}[A^{(n)}, \phi^{(n)}] = & \  \frac{1}{c_{0}^{2}} \eps_{0}^{2}, \label{eq:conc-scales:conc:1} \\
	\sup_{x \in B_{r}(x_{n})} \calE_{\set{t_{n}} \times B_{r_{n}}(x)}[A^{(n)}, \phi^{(n)}] \leq & \ \frac{1}{c_{0}^{2}} \eps_{0}^{2}, \label{eq:conc-scales:conc:2} \\
	\frac{1}{4 r_{n}} \int_{t_{n}-2 r_{n}}^{t_{n}+2 r_{n}} \int_{B_{r}(x_{n})} 
				\abs{\iota_{X_{0}} F^{(n)}}^{2} + \abs{(\covD^{(n)}_{X_{0}} + \frac{1}{\rho}) \phi^{(n)}}^{2} \, \ud t \ud x 
	\to & \ 0 \quad \hbox{ as } n \to \infty. \label{eq:conc-scales:conc:3}
\end{align}

\item {\bf Uniform non-concentration of energy}. There exists $0 < r = r(j) < 1/4$ such that the following bounds hold:
\begin{align} 
\int_{S^{(1-\gmm_{2})t}_{t}} \mvC{X_{0}}_{T}[A^{(n)}, \phi^{(n)}] \, \ud x \geq  &\  E_{2} \quad \hbox{ for } t \in [2^{j}, 2^{j+1}), \label{eq:conc-scales:non-conc:1}\\
\sup_{(t, x) \in C_{j}} \calE_{\set{t} \times B_{r}(x)} [A^{(n)}, \phi^{(n)}] \leq & \ \frac{1}{c_{0}^{2}} \eps_{0}^{2}, \label{eq:conc-scales:non-conc:2}\\
\iint_{\widetilde{C}_{j}} \abs{\iota_{X_{0}} F^{(n)}}^{2} + \abs{(\covD^{(n)}_{X_{0}} + \frac{1}{\rho}) \phi^{(n)}}^{2} \, \ud t \ud x \to & \ 0 \quad \hbox{ as } n \to \infty. \label{eq:conc-scales:non-conc:3}
\end{align}
\end{enumerate}
Here $c_{0} > 0$ is a universal constant much larger than the implicit constants in Lemma~\ref{l:hardy+}.
\end{lemma}
\begin{proof} 
This lemma is essentially \cite[Lemma~6.3]{MR2657818}; for completeness we give a self-contained alternative proof, which relies on the use of the Hardy-Littlewood maximal function theorem to establish \eqref{eq:conc-scales:conc:3}.

\pfstep{Step 1}
Fix $j \in \set{1, 2, \ldots}$. We begin by identifying a `low energy barrier' around $C_{j}$ inside $\widetilde{C}_{j}$. Let $N > 0$ be a large integer to be determined later. We first partition the time interval $[2^{j}, 2^{j+1})$ into smaller intervals $I_{k}$, where 
\begin{equation*}
	I_{k} := [2^{j} + \frac{k-1}{10 N}, 2^{j} + \frac{k}{10 N }) \quad k = 1, \ldots, 10 N 2^{j}.
\end{equation*}
Accordingly, define $C_{j}^{k} := C_{j} \cap (I_{k} \times \bbR^{4})$ and $\widetilde{C}_{j}^{k} := \widetilde{C}_{j} \cap (I_{k} \times \bbR^{4})$. Next, we partition $\widetilde{C}_{j}^{k} \setminus C_{j}^{k}$ into $\cup_{\ell=1}^{N} \widetilde{C}^{k, \ell}_{j}$, where
\begin{equation*}
	\widetilde{C}^{k, \ell}_{j} = \set{(t,x) \in \widetilde{C}_{j}^{k} : \frac{1}{2} + \frac{\ell - 1}{2N} \leq t - \abs{x} < \frac{1}{2} + \frac{\ell}{2N}}, \quad \ell = 1, \ldots, N.
\end{equation*}
For each $n$ and $k$, we claim that there exists $1 \leq \ell(n, k) \leq N$ such that
\begin{equation} \label{eq:low-e-barrier:pre}
	\sup_{t \in I_{k}} \, \calE_{S_{t} \cap \widetilde{C}_{j}^{k, \ell(n, k)}}[A^{(n)}, \phi^{(n)}] \leq \frac{3}{N} E.
\end{equation}
Indeed, for each $k$ consider the left endpoint $\underline{t}_{k} := 2^{j} + (k-1)/(10N)$. The set $S_{\underline{t}_{k}} \cap (\widetilde{C}^{k}_{j} \setminus C^{k}_{j})$ is partitioned into $N$ annuli of the form $S_{\underline{t}_{k}} \cap \widetilde{C}^{k, \ell}_{j}$.
By the pigeonhole principle and the energy bound \eqref{eq:final-rescale:energy}, there exists $1 \leq \ell(n, k) \leq N-2$ such that
\begin{equation*}
	\sum_{\ell = \ell(n, k)}^{\ell(n, k) + 2} \calE_{S_{\underline{t}_{k}} \cap \widetilde{C}_{j}^{k, \ell}} [A^{(n)}, \phi^{(n)}] \leq \frac{3}{N} E.
\end{equation*}
As $\widetilde{C}_{j}^{k, \ell(n, k)}$ lies in the domain of dependence of $\cup_{\ell = \ell(n, k)}^{\ell(n,k)+2} S_{\underline{t}_{k}} \cap \widetilde{C}_{j}^{k, \ell}$, \eqref{eq:low-e-barrier:pre} now follows by the local conservation of energy.

We choose $N$ large enough so that
\begin{equation*} 
	\frac{3}{N} E \leq \frac{1}{c_{0}^{2}} \eps_{0}^{2}.
\end{equation*}
Hence, by \eqref{eq:low-e-barrier:pre}, $\widetilde{C}_{j}^{k, \ell(n, k)}$ serves as a `low energy barrier' that separates the behavior of the solution in 
 the interior $\widetilde{C}^{k, < \ell(n, k)}_{j} := (\cup_{\ell = 1}^{\ell(n,k)-1} \widetilde{C}_{j}^{k, \ell(n,k)}) \cup C_{j}^{k}$ from the outside. Fix $0 < r_{0} < 1/4$ (independent of $n$ and $k$) so that
\begin{equation} \label{eq:low-e-barrier}
	(t, x) \in \widetilde{C}_{j}^{k, < \ell(n, k)} \imp \set{t} \times B_{4 r}(x) \subseteq \widetilde{C}_{j}^{k, < \ell(n, k)} \cap \widetilde{C}_{j}^{k, \ell(n, k)}.
\end{equation}

\pfstep{Step 2}
For each $n$ and $k$, define $f_{n, k} : [0, r_{0}] \times I_{k} \to [0, \infty)$ by
\begin{equation*}
	f_{n, k}(r, t) := \sup \set{\calE_{\set{t} \times B_{r}(x)}[A^{(n)}, \phi^{(n)}] : (t, x) \in \widetilde{C}_{j}^{k, < \ell(n, k)}}.
\end{equation*}
We then define the \emph{lowest energy concentration scale} $r_{n, k}(t)$ as
\begin{equation} \label{}
	r_{n, k}(t) := 
	\left\{
	\begin{array}{ll}
	\inf \set{r \in [0, r_{0}] : f_{n}(t, r) \geq \frac{1}{c_{0}^{2}} \eps_{0}^{2}} &  \hbox{if } f_{n}(t, r_{0}) \geq \frac{1}{c_{0}^{2}} \eps_{0}^{2}, \\
	r_{0} &  \hbox{otherwise}.
	\end{array}
	\right.
\end{equation}
By the finite speed of propagation, each $r_{n, k}$ is Lipschitz continuous with constant $\leq 1$:
\begin{equation*}
	\abs{r_{n, k}(t_{1}) - r_{n,k}(t_{0})} \leq \abs{t_{1} - t_{0}}.
\end{equation*}

We first treat the case when there exists a common lower bound $0 < r(j) \leq r_{0}$ of $r_{n,k}$, i.e., $r_{n, k}(t) \geq r(j)$ for all $n, k$ and $t \in I_{k}$.
Unraveling the definition of $r_{n, k}$, we see that \eqref{eq:conc-scales:non-conc:2} holds. Moreover, \eqref{eq:conc-scales:non-conc:1} and \eqref{eq:conc-scales:non-conc:3} follow directly from \eqref{eq:final-rescale:nontrivial} and \eqref{eq:final-rescale:asymp-ss}, respectively. Thus we conclude that the second scenario (uniform non-concentration of energy) holds.

To complete the proof, it only remains to consider the alternative case and show that the first scenario (concentration of energy) holds. After passing to a subsequence, we may assume that there exists $k \in \set{1, \ldots, 10 N 2^{j}}$ such that
\begin{equation} \label{eq:conc-scale:scale-to-0}
\lim_{n \to \infty} \inf_{I_{k}} r_{n, k} = 0.
\end{equation}
Then we claim that there exist $(t_{n}, x_{n})$ and $r_{n}$ such that \eqref{eq:conc-scales:conc:1}--\eqref{eq:conc-scales:conc:3} hold with $r(j) = r_{0}$, up to passing to a subsequence. 

Define
\begin{equation*}
	\alp_{n}^{2} := \int_{2^{j-1}}^{2^{j+2}} \bt_{n}^{2}(t) \, \ud t, \quad \bt_{n}^{2}(t) := \int_{S_{t} \cap C^{1/2}_{[1/2,\infty)}} \abs{\iota_{X_{0}} F^{(n)}}^{2} + \abs{(\covD^{(n)}_{X_{0}} + \frac{1}{\rho}) \phi^{(n)}}^{2} \, \ud x.
\end{equation*}
Note that $\alp_{n}^{2} \to 0$ by \eqref{eq:final-rescale:asymp-ss}. By the Hardy-Littlewood maximal function theorem, for every $\alp > 0$ we have
\begin{equation} \label{eq:conc-scale:HLM}
	\abs{ \set{t \in [2^{j-1}, 2^{j+1}) : M [\bt_{n}^{2}](t) > \alp }} \aleq \frac{1}{\alp} \alp_{n}^{2},
\end{equation} 
where $M[\bt_{n}](t)$ is the Hardy-Littlewood maximal function on $[2^{j-1}, 2^{j+2})$, given by
\begin{equation*}
	M[\bt_{n}] (t):= \sup_{a > 0} \frac{1}{2a}\int_{(t-a, t+a) \cap [2^{j-1}, 2^{j+2})} \bt^{2}_{n}(t') \, \ud t'.
\end{equation*}

Roughly speaking, \eqref{eq:conc-scale:HLM} says that the desired conclusion \eqref{eq:conc-scales:conc:3} holds for `most of' $t \in I_{k}$. This fact, combined with the flexibility of the choice of $t_{n}$ such that $\lim_{n \to \infty} r_{n, k}(t_{n}) = 0$, will lead to the desired conclusions \eqref{eq:conc-scales:conc:1}--\eqref{eq:conc-scales:conc:3}.

More precisely, define the intervals $J_{n}, K_{n} \subseteq I_{k}$ by
\begin{equation*}
	J_{n} := \set{ t \in I_{k} : M[\bt_{n}^{2}] \leq \alp_{n}}, \quad
	K_{n} := (\overline{t}_{n} - \alp_{n}^{1/2}, \overline{t}_{n} + \alp_{n}^{1/2}) \cap I_{k},
\end{equation*}
where $\overline{t}_{n} \in I_{k}$ is a minimum of $r_{n, k}$, i.e., $r_{n, k}(\overline{t}_{n}) = \inf_{I_{k}} r_{n, k}$. By the uniform Lipschitz continuity of $r_{n,k}$ and the fact that $\alp_{n}^{2} \to 0$ as $n \to \infty$, we have
\begin{equation*}
	\sup_{t \in K_{n}} r_{n, k}(t) \to 0 \hbox{ as } n \to \infty.
\end{equation*}
Note that $\abs{I_{k} \setminus J_{n}} \aleq \alp_{n}$ by \eqref{eq:conc-scale:HLM} with $\alp = \alp_{n}$, whereas $\abs{K_{n}} = 2 \alp_{n}^{1/2}$. Using again the fact that $\alp_{n}^{2} \to 0$ as $n \to \infty$ and passing to a subsequence, it follows that $J_{n} \cap K_{n} \neq \0$ for all $n$. Choosing $t_{n}$ so that $t_{n} \in J_{n} \cap K_{n}$ and $r_{n} := r_{n, k}(t_{n})$, we have
\begin{equation*}
	\sup_{a > 0} \frac{1}{2a}\int_{t_{n}-a}^{t_{n}+a} \bt_{n}^{2}(t) \, \ud t \to 0, \quad r_{n} = r_{n, k}(t_{n}) \to 0 \quad \hbox{ as } n \to \infty.
\end{equation*}
In particular, \eqref{eq:conc-scales:conc:3} holds. Passing to a subsequence if necessary, we may assume that $r_{n, k}(t_{n}) < r_{0}$; then there exists $(t_{n}, x_{n}) \in \widetilde{C}_{j}^{k, <\ell(n, k)}$ such that \eqref{eq:conc-scales:conc:1} holds for all $n$ as well. Finally, thanks to the low energy barrier \eqref{eq:low-e-barrier} and the definition of $r_{n,k}$, \eqref{eq:conc-scales:conc:2} follows with $r(j) = r_{0}$. \qedhere
\end{proof}

\subsection{Compactness/rigidity argument} \label{subsec:cpt-rigid} We
are now ready to complete the proof of Theorem~\ref{thm:main}, by
using the tools developed in Sections~\ref{sec:cpt} and
\ref{sec:stationary-self-sim}.

\begin{proof} [Completion of proof of Theorem~\ref{thm:main}]
  Let $(A^{(n)}, \phi^{(n)})$ be a sequence of admissible $C_{t}
  \calH^{1}$ solutions on $[1, T_{n}] \times \bbR^{4}$ given by
  Lemma~\ref{lem:final-rescale}. We consider two cases according to
  Lemma~\ref{lem:conc-scales}, and show that both lead to
  contradictions.

  \pfstep{Case 1} Suppose that there exists $j \in \set{1, 2, \ldots}$
  such that the first scenario (concentration of energy) in
  Lemma~\ref{lem:conc-scales} holds. We need to set things up 
so that we can use  Proposition~\ref{prop:cpt}, and for that we 
also need local control of the $L^2$ norm of $\phi$. This is achieved 
via the improved form of Hardy's inequality in Lemma~\ref{l:hardy+}.
From  \eqref{eq:conc-scales:conc:2} we obtain 
\[
(\sigma r_n)^{-2} \|\phi^{(n)}(t_n) \|_{L^2_{x}(B_{\sigma r_n})}^2 \leq \frac{1}{10} \epsilon_0^2 
+ C \sigma^{2} E, \qquad \sigma < 1
\]
To eliminate the second term we choose 
\[
\sigma^{2} = c \epsilon_0^2 E^{-1}
\]
with a small universal constant $c$. Thus we have insured that the
hypothesis of Proposition~\ref{prop:cpt} are satisfied with respect to
the rescaled ball $B_{\sigma r_k}(x)$ with $x$ as in  \eqref{eq:conc-scales:conc:2}, i.e.,
\begin{equation}\label{eq:loc-rn}
 \calE_{\set{t_{n}} \times B_{8\sigma r_{n}}(x)}[A^{(n)}, \phi^{(n)}] +(\sigma r_n)^{-2} \|\phi^{(n)}(t_n) \|_{L^2_{x}(B_{8\sigma r_n}(x))}^2 \leq \epsilon_0^2 
\end{equation}

As $\widetilde{C}_{j}$ is
  pre-compact, we may assume that $(t_{n}, x_{n})$ has a limit
  $(t_{0}, x_{0})$ in the closure of $\widetilde{C}_{j}$ after passing
  to a subsequence. Consider the sequence
  \begin{equation*}
    (\widetilde{A}^{(n)}, \widetilde{\phi}^{(n)})(t,x) := r_{n} (A^{(n)}, \phi^{(n)}) (\sigma r_{n} t + t_{n}, \sigma r_{n} x + x_{n}).
  \end{equation*}
  By \eqref{eq:conc-scales:conc:1}, there is always a nontrivial
  amount of energy at the origin, i.e.,
  \begin{equation} \label{eq:end:conc:nontrivial} \calE_{\set{0}
      \times B_{\sigma^{-1}}(0)}[\widetilde{A}^{(n)}, \widetilde{\phi}^{(n)}] =
    \frac{1}{c_{0}^{2}} \eps_{0}^{2}.
  \end{equation}
  Fix any $x \in \bbR^{4}$. As $r_{n} \to 0$, observe that the point
  $r_{n} x + x_{n}$ belongs to $B_{r(j)}(x_{n})$ for sufficiently
  large $n$. Hence, by \eqref{eq:loc-rn}, we have
  \begin{equation} \label{eq:end:conc:smallness} \calE_{\set{0} \times
      B_{8}(x)}[\widetilde{A}^{(n)}, \widetilde{\phi}^{(n)}] + \| \widetilde{\phi}^{(n)}(0)\|_{L^2_{x}( B_{8}(x))}
    \leq \eps_{0}^{2} \quad \hbox{ for sufficiently large } n.
  \end{equation}
    Finally, by
  \eqref{eq:conc-scales:conc:3}, the convergence $(t_{n}, x_{n}) \to
  (t_{0}, x_{0})$ and smoothness of $X_{0}$, it follows that
  \begin{equation} \label{eq:end:conc:asymp-ss} \iint_{(-2, 2) \times
      B_{2}(x)} \abs{\iota_{Y} \widetilde{F}^{(n)}}^{2} +
    \abs{\widetilde{\covD}^{(n)}_{Y} \widetilde{\phi}^{(n)}}^{2} \,
    \ud t \ud x \to 0 \quad \hbox{ as } n \to \infty.
  \end{equation}
  where $Y = X_{0}(t_{0}, x_{0})$ is a constant time-like vector
  field. Note that the contribution of the term $\frac{1}{\rho}
  \phi^{(n)}$ dropped out by scaling.

  As a consequence, for each $x \in \bbR^{4}$ we can apply
  Proposition~\ref{prop:cpt} to obtain a weak solution $(A_{[x]},
  \phi_{[x]}) \in \Xw((-1, 1) \times B_{1}(x))$ to \eqref{eq:MKG} such
  that
  \begin{equation*}
    \iota_{Y} F_{[x]} = 0, \quad \covD_{[x] Y} \phi_{[x]} = 0,
  \end{equation*}
  and $(\widetilde{A}^{(n)}, \widetilde{\phi}^{(n)})$ converges to
  $(A_{[x]}, \phi_{[x]})$ up to gauge transformations on $(-1, 1)
  \times B_{1}(x)$ as in \eqref{eq:cpt:converge:non-cov},
  \eqref{eq:cpt:converge:cov}.
  By Lemma~\ref{lem:weak-cp-limit}, the weak solutions $(A_{[x]},
  \phi_{[x]})$ form weak compatible pairs (as in
  Definition~\ref{def:weak-cp}) on the open cover $\set{(-1, 1) \times
    B_{1}(x)}_{x \in (1/2)\bbZ^{4}}$ of $(-1, 1) \times
  \bbR^{4}$. Furthermore, by Proposition~\ref{prop:reg}, there exists
  an equivalent set of \emph{smooth} compatible pairs $(A_{[\alp]},
  \phi_{[\alp]})$ on some refined open cover $\calQ = \set{Q_{\alp}}$
  of $(-1, 1) \times \bbR^{4}$.

  Let $(A, \phi)$ be a global smooth pair on $(-1, 1) \times \bbR^{4}$
  equivalent to $(A_{[\alp]}, \phi_{[\alp]})$. We then extend $(A,
  \phi)$ to $\bbR^{1+4}$ as a smooth solution to \eqref{eq:MKG}
  satisfying $\iota_{Y} F = 0$ and $\covD_{Y} \phi = 0$ by pulling
  back along the flow of $Y$.
  Note that $(A, \phi)$ has finite energy (in fact, bounded by $\leq
  E$), as we have
  \begin{equation} \label{eq:end:conc:loc-conv}
    \vC{T}_{T}[\widetilde{A}^{(n)}, \widetilde{\phi}^{(n)}] \to
    \vC{T}_{T}[A, \phi] \quad \hbox{ locally in } L^{1}_{t,x} \hbox{
      on } (-1, 1) \times \bbR^{4}
  \end{equation}
  by \eqref{eq:cpt:converge:cov} and the gauge invariance of the
  energy density $\vC{T}_{T}$.  After applying a suitable Lorentz
  transform, we may furthermore assume that $Y = T$.  By
  Proposition~\ref{prop:trivial:st}, it follows that $\calE [A, \phi]
  = 0$, but this contradicts \eqref{eq:end:conc:nontrivial} and
  \eqref{eq:end:conc:loc-conv}.

  \pfstep{Case 2} Suppose that for every $j \in \set{1, 2, \ldots}$
  the second scenario (uniform non-con\-cen\-tration of energy) in
  Lemma~\ref{lem:conc-scales} holds. In this case there is no need to
  rescale. Indeed, \eqref{eq:conc-scales:non-conc:2} and
  \eqref{eq:conc-scales:non-conc:3} (as well as
  \eqref{eq:final-rescale:energy}, \eqref{eq:final-rescale:extr} and
  Lemma~\ref{l:hardy+}) allow us to apply the rescaled
  Proposition~\ref{prop:cpt} directly to $(A^{(n)}, \phi^{(n)})$ on
  $\set{t} \times \sgm B(x)$, with $\sigma$ as in Case 1, for $(t, x)
  \in C^{T}_{[T, \infty)}$ for some $T = T(\sgm) > 1$. Proceeding as
  in the previous case using Lemma~\ref{lem:weak-cp-limit} and
  Proposition~\ref{prop:reg}, we then obtain a global smooth pair $(A,
  \phi)$ on $C^{T}_{[T, \infty)}$ satisfying the following properties:
  \begin{itemize}
  \item The pair $(A, \phi)$ is a smooth solution to \eqref{eq:MKG}
    obeying the self-similarity condition
    \begin{equation*}
      \iota_{X_{0}} F = 0, \quad (\covD_{X_{0}} + \frac{1}{\rho}) \phi = \frac{1}{\rho} \covD_{X_{0}}(\rho \phi)= 0.
    \end{equation*}
  \item The following local convergences hold:
    \begin{align}
      \vC{T}_{T}[\widetilde{A}^{(n)}, \widetilde{\phi}^{(n)}] \to & \vC{T}_{T}[A, \phi] \quad \hbox{ locally in } L^{1}_{t,x} \hbox{ on } C^{T}_{[T, \infty)},  \label{eq:end:non-conc:loc-conv:1} \\
      \mvC{X_{0}}_{T}[\widetilde{A}^{(n)}, \widetilde{\phi}^{(n)}] \to
      & \mvC{X_{0}}_{T}[A, \phi] \quad \hbox{ locally in } L^{1}_{t,x}
      \hbox{ on } C^{T}_{[T,
        \infty)}. \label{eq:end:non-conc:loc-conv:2}
    \end{align}
  \end{itemize}

  We extend $(A, \phi)$ to a smooth self-similar solution to
  \eqref{eq:MKG} on the whole cone $C_{(0, \infty)} = \set{0 \leq r <
    t}$ by pulling back $(A, \rho \phi)$ along the flow of
  $X_{0}$. Note that $(A, \phi)$ has finite energy (again bounded by
  $\leq E$), thanks to the local convergence
  \eqref{eq:end:non-conc:loc-conv:1}. Hence by
  Proposition~\ref{prop:trivial:ss}, it follows that $\calE_{S_{t}}[A,
  \phi] = 0$ for every $t \in (0, \infty)$. However, this is a
  contradiction with \eqref{eq:conc-scales:non-conc:1} (in particular,
  for large enough $t$ so that $S^{(1-\gmm_{2})t}_{t} \subseteq
  C^{T}_{[T, \infty)}$) and
  \eqref{eq:end:non-conc:loc-conv:2}. \qedhere
\end{proof}
\bibliographystyle{amsplain}

\providecommand{\bysame}{\leavevmode\hbox to3em{\hrulefill}\thinspace}
\providecommand{\MR}{\relax\ifhmode\unskip\space\fi MR }
\providecommand{\MRhref}[2]{%
  \href{http://www.ams.org/mathscinet-getitem?mr=#1}{#2} }
\providecommand{\href}[2]{#2}

\end{document}